\documentclass[oneside,a4paper]{amsart}

\usepackage[T1]{fontenc}
\usepackage[utf8]{inputenc}
\usepackage{amsmath, amsthm, amsfonts, amssymb}

\usepackage[usenames, dvipsnames]{xcolor} 
\usepackage{url}
\usepackage{nicefrac}
\usepackage[all]{xy}
\usepackage{placeins} 
\usepackage{musicography} 

\usepackage{euscript} 
\usepackage{bbm} 
\usepackage{enumitem} 

\usepackage{hyperref}
\hypersetup{
    colorlinks, linkcolor=NavyBlue,
    citecolor=OliveGreen, urlcolor=RedOrange
}
\usepackage[nameinlink,capitalise,noabbrev]{cleveref} 
\usepackage{todonotes}
\setuptodonotes{fancyline, color=green!30}
\usepackage{faktor}

\usepackage{tikz-cd}
\usepackage{tikz}
\usetikzlibrary{arrows,backgrounds,
    decorations.pathreplacing,
    decorations.pathmorphing
}
\usetikzlibrary{matrix,arrows}
\usepackage{colonequals}

\newcommand{\euscr}[1]{\EuScript{#1}} 
\newcommand{\acat}{\euscr{A}} 
\newcommand{\ccat}{\euscr{C}} 
\newcommand{\dcat}{\euscr{D}} 
\newcommand{\Fun}{\textnormal{Fun}} 
\newcommand{\Hom}{\textnormal{Hom}} 
\newcommand{\Ext}{\textnormal{Ext}} 

\newcommand{\map}{\textnormal{map}} 
\newcommand{\Map}{\textnormal{Map}} 
\newcommand{\spaces}{\euscr{S}} 
\newcommand{\abeliangroups}{\euscr{A}b} 
\newcommand{\Ab}{\abeliangroups} 
\newcommand{\Perf}{\mathrm{Perf}} 
\newcommand{\inftycatcat}{\euscr{C}at_{\infty}} 

\newcommand{\catinfty}{\inftycatcat}
\newcommand{\largecatinfty}{\widehat{\euscr{C}at_{\infty}}}
\newcommand{\largecatinftyL}{\widehat{\euscr{C}at^{L}_{\infty}}}
\newcommand{\spectra}{\euscr{S}p} 

\newcommand{\Mod}{\euscr{M}od} 

\def\MU{\mathrm{MU}}

\newcommand{\nocontentsline}[3]{}
\newcommand{\tocless}[2]{\bgroup\let\addcontentsline=\nocontentsline#1{#2}\egroup}

\newcommand{\coker}{\mathrm{coker}}
\newcommand{\Hrm}{\mathrm{H}} 
\newcommand{\Alg}{\mathrm{Alg}}

\newcommand{\cofib}{\mathrm{cofib}}

\newcommand{\ZZ}{\mathbb{Z}}

\newcommand{\Fil}{\mathrm{Fil}}

\DeclareMathOperator{\Syn}{Syn}

\theoremstyle{plain}

\newtheorem{theorem}{Theorem}[section]
\newtheorem{lemma}[theorem]{Lemma}
\newtheorem{proposition}[theorem]{Proposition}
\newtheorem{corollary}[theorem]{Corollary}

\theoremstyle{definition}
\newtheorem{example}[theorem]{Example}
\newtheorem{warning}[theorem]{Warning}
\newtheorem{recollection}[theorem]{Recollection}

\newtheorem{definition}[theorem]{Definition}
\newtheorem{remark}[theorem]{Remark}
\newtheorem{notation}[theorem]{Notation}
\newtheorem{construction}[theorem]{Construction}

\newtheorem*{remark*}{Remark}
\newtheorem*{terminology*}{Terminology}
\newtheorem*{interpretation*}{Interpretation}
\newtheorem*{definition*}{Definition}
\newtheorem*{conjecture*}{Conjecture}
\newtheorem*{notation*}{Notation}
\newtheorem*{convention*}{Convention}

\theoremstyle{remark}

\numberwithin{equation}{section}

\makeatletter
  \def\subsection{\@startsection{subsection}{1}%
  \z@{.7\linespacing\@plus\linespacing}{.5\linespacing}%
  {\normalfont\bfseries\centering}}
\makeatother

\let\oldtocsection=\tocsection
\let\oldtocsubsection=\tocsubsection
\let\oldtocsubsubsection=\tocsubsubsection
\renewcommand{\tocsection}[2]{\hspace{0em}\oldtocsection{#1}{#2}}
\renewcommand{\tocsubsection}[2]{\hspace{1em}\oldtocsubsection{#1}{#2}}
\renewcommand{\tocsubsubsection}[2]{\hspace{2em}\oldtocsubsubsection{#1}{#2}}

\calclayout

\newcommand{\evensheaf}{\mathcal{F}}
\newcommand{\evencoh}{\mathrm{H}_{ev}}
\newcommand{\BPn}{\mathrm{BP}\langle n \rangle}
\newcommand{\PrL}{\mathrm{Pr}^{L}}

\newcommand{\presheaves}{\mathcal{P}}

\DeclareMathOperator{\sheaves}{Shv}
\DeclareMathOperator{\fil}{fil}

\DeclareMathOperator{\Tor}{Tor}
\DeclareMathOperator{\gr}{gr}
\DeclareMathOperator{\THH}{THH}
\DeclareMathOperator{\TC}{TC}
\DeclareMathOperator{\TP}{TP}

\mathchardef\mhyphen="2D

\begin{document}

\title[Perfect even modules and the even filtration]{Perfect even modules and the even filtration}

\author{Piotr Pstr\k{a}gowski}
\address{Harvard University, USA, and Institute for Advanced Study, USA}
\email{pstragowski.piotr@gmail.com}

\thanks{The author received support from NSF Grant number DMS-1926686 and Deutsche Forschungsgemeinschaft under Germany's Excellence Strategy – EXC-2047/1 – 390685813.}

\begin{abstract}
Inspired by the work of Hahn-Raksit-Wilson, we introduce a variant of the even filtration which is naturally defined on $\mathbf{E}_{1}$-rings and their modules. We show that our variant satisfies flat descent and so agrees with the Hahn-Raksit-Wilson filtration on ring spectra of arithmetic interest, showing that various ``motivic'' filtrations are in fact invariants of the $\mathbf{E}_{1}$-structure alone. We prove that our filtration can be calculated via appropriate resolutions in modules and apply it to the study of even cohomology of connective $\mathbf{E}_{1}$-rings, proving vanishing above the Milnor line, base-change formulas, and explicitly calculating cohomology in low weights. 
\end{abstract}

\maketitle
\setcounter{tocdepth}{2}
\tableofcontents

\section{Introduction}

In \cite{hahn2022motivic}, Hahn-Raksit-Wilson introduced the \emph{even filtration} attached to a commutative algebra in spectra. They show that when applied to ring spectra of arithmetic interest, such as the sphere or $\THH$ of commutative rings, their construction recovers various important filtrations, such as the Bhatt-Morrow-Scholze filtration of \cite{bhatt2019topological}, implying that these filtrations are invariants of the $\mathbf{E}_{\infty}$-ring structure alone. The Bhatt-Morrow-Scholze filtration can be used to define prismatic and syntomic cohomology of commutative rings, and the even filtration allows one to extend this construction to the context of $\mathbf{E}_{\infty}$-ring spectra. This provides strong evidence towards the long-standing conjecture of Rognes on the existence of a motivic filtration on algebraic $K$-theory of commutative ring spectra \cite{rognes2014algebraic}.

In this paper, we introduce a variant of the even filtration which informally measures the complexity of the $\infty$-category of perfect complexes with even cells. The filtration we construct, which one could call the \emph{perfect even filtration}, has the following properties: 
\begin{enumerate}
    \item it is naturally defined on $\mathbf{E}_{1}$-ring spectra and their modules, rather than only in $\mathbf{E}_{\infty}$-case,
    \item it satisfies flat descent, and as a consequence agrees with the even filtration in the examples considered in \cite{hahn2022motivic}, showing in particular that various motivic filtrations are invariants of the $\mathbf{E}_{1}$-ring structure alone, 
    \item in the examples where they disagree, for example on free $\mathbf{E}_{\infty}$-algebras, the perfect even filtration gives more reasonable answers than the even filtration, 
    \item the perfect even filtration has an essentially linear definition and so can be efficiently computed by resolutions of modules,
    \item the perfect even cohomology groups have excellent formal properties, especially in the case of connective $\mathbf{E}_{1}$-rings: they vanish above the Milnor line, satisfy a base-change formula in a neighbourhood of it, and in low weights can be calculated explicitly. 
\end{enumerate}

Our construction of the perfect even filtration is of categorical nature: we work with the $\infty$-category of modules, together with its notion of evenness, rather than the ring itself. This has the advantage that it naturally lends itself to generalization into other contexts, such as equivariant or motivic homotopy theory. We briefly discuss the possible generalizations at the end of the introduction.

Note that for $\THH(R)$ to have an $\mathbf{E}_{1}$-ring structure and hence have an induced perfect even filtration, $R$ needs to be only an $\mathbf{E}_{2}$-ring. Thus, the main construction of this paper suggests the existence of a good theory of prismatic cohomology in the context of $\mathbf{E}_{2}$-ring spectra, a topic we will pursue in the sequel to the current work written jointly with Raksit \cite{motivic_cohomology_of_e2_rings}.

\begin{terminology*} 
The above few paragraphs are the only part of the paper where we use the term \emph{perfect even filtration}. This is an apt name, but for brevity in the main body of the text we refer to the filtration introduced in the current work simply as the \emph{even filtration}. To distinguish it from the one introduced by Hahn-Raksit-Wilson, we will refer to the latter as the \emph{$\mathbf{E}_{\infty}$-even filtration}. 
\end{terminology*}

We now discuss our results in more detail. Let $R$ be an $\mathbf{E}_{1}$-algebra in spectra. We say that a left $R$-module $A$ is \emph{perfect even} if it belongs to the smallest subcategory 
\[
\Perf(R)_{ev} \subseteq \Perf(R) 
\]
containing $R$ and closed under even (de)suspensions, retracts and extensions. We say that a map of perfect even $R$-modules is an \emph{even epimorphism} if its fibre is again perfect even. Declaring even epimorphisms as coverings endows $\Perf(R)_{ev}$ with a Grothendieck topology. An $R$-module $M$ determines through the spectral Yoneda embedding an additive spectral sheaf 
\[
Y_{R}(M) \in \sheaves_{\Sigma}(\Perf(R)_{ev}, \spectra) 
\]
given by the formula
\begin{equation}
\label{equation:definition_of_spectraL_yoneda_in_introduction}
Y_{R}(M)(A) \colonequals \map_{\Mod_{R}}(A, M),
\end{equation}
where the right hand side is the mapping spectrum in $R$-modules. 

\begin{definition}
\label{definition:introduction_even_filtration}
The \emph{even filtration} on $M$ is the filtered spectrum given by sections
\[
\fil^{q}_{ev/R}(M) \colonequals \Gamma_{\Perf(R)_{ev}}(R, \tau_{\geq 2q} Y_{R}(M)), 
\]
where $\tau_{\geq 2q}(-)$ denotes the connective cover in sheaves of spectra. 
\end{definition}

\begin{notation}
To avoid cluttering notation, we write $\fil^{\ast}_{ev}(R) \colonequals \fil^{\ast}_{ev/R}(R)$ for the even filtration of an $\mathbf{E}_{1}$-ring considered as a module over itself. 
\end{notation}

Our first result shows that this filtration is compatible with the ring structure and is suitably multiplicative: 

\begin{theorem}[{\ref{theorem:even_filtration_is_lax_symmetric_monoidal_as_a_functor_on_pairs}}]
\label{theorem:introduction_the_even_filtration_is_multiplicative}
Let $R$ be an $\mathbf{E}_{1}$-ring and let $M$ be a left $R$-module. Then, $\fil^{\ast}_{ev}(R)$ has a canonical structure of an $\mathbf{E}_{1}$-algebra in filtered spectra over which $\fil^{\ast}_{ev/R}(M)$ is a left module. Moreover, the construction 
\[
(R, M) \mapsto (\fil^{\ast}_{ev}(R), \fil^{\ast}_{ev/R}(M))
\]
can be refined to a lax symmetric monoidal functor 
\[
\Mod(\spectra) \rightarrow \Mod(\Fil \spectra)
\]
between the $\infty$-categories of pairs of an $\mathbf{E}_{1}$-algebra and a left module in, respectively, spectra and filtered spectra. 
\end{theorem}
In particular, lax symmetric monoidality implies that if $R$ is an $\mathbf{E}_{n}$-algebra in spectra, then 
\begin{enumerate}
    \item $\fil^{\ast}_{ev}(R)$ is a $\fil^{\ast}_{ev}(S^{0})$-$\mathbf{E}_{n}$-algebra, 
    \item the $R$-module even filtration 
    \[
\fil^{\ast}_{ev/R} \colon \Mod_{R}(\spectra) \rightarrow \Mod_{\fil^{\ast}_{ev}(R)}(\Fil \spectra) 
    \]
    is $\mathbf{E}_{n-1}$-monoidal. 
\end{enumerate}
As we discuss below, $\fil^{\ast}_{ev}(S^{0})$ can be identified with the Adams-Novikov filtration of the sphere, so that the first property implies that $\fil^{*}_{ev}(R)$ has a canonical lift to an $\mathbf{E}_{n}$-algebra in synthetic spectra or, after $p$-completion, to an algebra in the $\infty$-category of $\mathbb{C}$-motivic spectra \cite{gheorghe2022c, pstrkagowski2018synthetic}. 

We show that the even filtration is given by the Whitehead tower $M \mapsto \tau_{\geq 2*} M$ if either the ring $R$ or the module $M$ has homotopy groups concentrated in even degrees. In particular, it follows that when restricted to $\mathbf{E}_{\infty}$-rings, there is a canonical comparison natural transformation
\[
\fil^{*}_{ev}(-) \rightarrow \fil^{*}_{\mathbf{E}_{\infty} \mhyphen ev}(-)
\]
into the Hahn-Raksit-Wilson $\mathbf{E}_{\infty}$-even filtration. As the main method of calculating the $\mathbf{E}_{\infty}$-even filtration is via flat descent, to establish that the comparison map is an equivalence for a large class of rings we first prove descent for the even filtration of \cref{definition:introduction_even_filtration}. 

We say that an $R$-module $M$ is \emph{even flat} if it can be written as a filtered colimit of perfect evens; these modules can be characterized in several different ways, see \cref{proposition:tensor_characterization_of_even_flat_modules}, and so can be effectively detected. If $R$ is an $\mathbf{E}_{2}$-ring\footnote{In the main body of the text we prove descent for maps of $\mathbf{E}_{1}$-rings, see \cref{theorem:faithfully_flat_descent_for_modules}, but we stick to algebras in the introduction for simplicity.}, then we say that an $\mathbf{E}_{1}$-$R$-algebra is \emph{faithfully even flat} if both $S$ and $\mathrm{cofib}(R \rightarrow S)$ are even flat as $R$-modules. This definition is motivated by a classical observation that a monomorphism of classical commutative rings is faithfully flat if both the target and the cokernel are flat. In \cref{proposition:fef_maps_are_hrw_eff}, we show that in the $\mathbf{E}_{\infty}$-context, our notion of faithfully even flat is equivalent to that of Hahn-Raksit-Wilson for connective rings, and strictly stronger in general. 

\begin{theorem}[{\ref{theorem:fef_descent_for_algebras_over_e2_rings}}]
\label{theorem:introduction_fef_descent}
Let $R$ be an $\mathbf{E}_{2}$-ring and let $S$ be a faithfully even flat $\mathbf{E}_{1}$-$R$-algebra. Then for any $R$-module $M$ the canonical map 
\[
\fil^{*}_{ev / R}(M) \rightarrow \varprojlim\fil^{*}_{ev / S^{\otimes_{R} \bullet}}(S^{\otimes_{R} \bullet} \otimes_{R} M)
\]
is an equivalence of filtered spectra after completion.
\end{theorem}

Since the $\mathbf{E}_{\infty}$-even filtration also satisfies flat descent, we deduce the following:

\begin{theorem}[{\ref{theorem:comparison_between_ev_and_einfty_ev_filtrations}}]
Let $R$ be an $\mathbf{E}_{\infty}$-ring which admits a faithfully even flat map $R \rightarrow S$ into an $\mathbf{E}_{\infty}$-ring with $\pi_{*}S$ even. Then the comparison functor
\[
\fil^{*}_{ev}(R) \rightarrow \fil^{*}_{\mathbf{E}_{\infty} \mhyphen ev}(R)
\]
is an equivalence of filtered spectra after completion. 
\end{theorem}
Note that the even filtration considered in this paper is always exhaustive in the sense that $\varinjlim \fil^{*}_{ev}(R) \simeq R$, but it is not always complete. Thus, even in cases where the two filtrations agree up to completion, $\fil^{*}_{ev}(R)$ can be considered as a naturally occuring ``decompletion'' of the Hahn-Raksit-Wilson filtration. 

The comparison of \cref{theorem:introduction_fef_descent} covers most of examples considered in \cite{hahn2022motivic}. In particular, it applies to $S^{0}, \mathrm{HH}(R/k), \THH(R)$ and $\THH(R)^{\wedge}_{p}$, where by results of Hahn-Raksit-Wilson one recovers, respectively, the Adams-Novikov filtration, the Hochschild-Kostant-Rosenberg filtration, the Bhatt-Lurie filtration and the filtration of Bhatt-Morrow-Scholze. 

However, we warn the reader that \cref{theorem:introduction_fef_descent} does not apply to the the motivic filtrations on $\TC^{-}, \TP$ and $\TC$, as here an appropriate definition of the $\mathbf{E}_{\infty}$-even filtration requires one to take the circle action and the cyclotomic structure into account. In upcoming joint work with Raksit \cite{motivic_cohomology_of_e2_rings}, we show how to recover the filtrations on $\TC^{-}$ and $\TP$ using a spherical lift of the filtered circle of \cite{raksit2020hochschild}, but we do not touch on this subject in the current work. 

The even filtration and the $\mathbf{E}_{\infty}$-even filtration do not agree in general. In \cref{example:roberts_example_of_even_and_einfty_even_being_different}, due to Robert Burklund, we give an instructive instance of this phenomena for a connective $\mathbf{E}_{\infty}$-ring. However, one could argue that when they disagree, it is the even filtration of \cref{definition:introduction_even_filtration} which gives more reasonable answers: in Burklund's example, we are able to determine the structure of the even filtration completely, but the nature of the $\mathbf{E}_{\infty}$-even filtration seems somewhat difficult. 

Outside of the connective context, the situation is even more striking: in \cref{warning:more_extreme_example_of_even_and_einfty_even-filtration_diverging}, we describe a periodic $\mathbf{E}_{\infty}$-ring whose $\mathbf{E}_{\infty}$-even filtration is identically zero, but whose even filtration is exhaustive, complete and easy to calculate. This simplicity comes down to the fact that even if one starts with an $\mathbf{E}_{\infty}$-ring $R$, it is much easier to produce $\mathbf{E}_{1}$-$R$-algebras rather than $\mathbf{E}_{\infty}$-$R$-algebras, which allows one to apply \cref{theorem:introduction_fef_descent}. 

We now describe how one can calculate the even filtration in practice. Since the even filtration is defined in terms of Postnikov towers in sheaves of spectra, it is controlled by sheaf cohomology. For any half-integer $q$, we write 
\[
\evensheaf_{M}(q) \colonequals \pi_{2q} Y_{R}(M)
\]
for the sheaf of homotopy groups of (\ref{equation:definition_of_spectraL_yoneda_in_introduction}) and call it the \emph{even sheaf of weight $q$}. We say that $M$ is \emph{homologically even} if these sheaves vanish for $q \in \mathbb{Z} + \nicefrac{1}{2}$. For example, all perfect even modules are homologically even, in particular $R$ itself.

\begin{definition}
The \emph{even cohomology} of $R$ with coefficients in $M$ is given by sheaf cohomology 
\[
\evencoh^{p, q}(R, M) \colonequals \mathrm{H}^{p}_{\Perf(R)_{ev}}(R, \evensheaf_{M}(q)).  
\]
Here, the right hand side is given by derived functors of $\Hrm^{0}(R, -) \colon \sheaves(\Perf(R)_{ev}, \abeliangroups) \rightarrow \abeliangroups$. 
\end{definition}
We show that in \cref{theorem:associated_graded_of_even_filtration_and_even_cohomology} that if $M$ is homologically even\footnote{The assumption that $M$ is homologically even can be removed, but to get the most elegant results in this case it is preferable to introduce a refinement of the even filtration into a filtration indexed by half-integers, see \cref{remark:integer_and_half_integer_grading}. We mainly work in the homologically even case as it covers the examples we are after, since $R$ is homologically even when considered as a module over itself.}, then the associated graded object of the even filtration can be described in terms of even cohomology in the sense that 
\begin{equation}
\label{equation:introduction_sheaf_cohomology_calculates_homotopy_of_associated_graded_of_even_filtration}
\pi_{2q-p} (\mathrm{gr}^{q}_{ev}(M))  \simeq \evencoh^{p, q}(R, M).
\end{equation}
The even filtration thus induces a spectral sequence of signature 
\[
\evencoh^{p, q}(R, M) \Rightarrow \pi_{2q-p}(M),
\]
which we call the \emph{even spectral sequence}. 

The identification of (\ref{equation:introduction_sheaf_cohomology_calculates_homotopy_of_associated_graded_of_even_filtration}) is useful since, as any form of sheaf cohomology, even cohomology can be efficiently computed using injective resolutions inside the category of sheaves of abelian groups. To do so, one needs access to injective sheaves, a bountiful source of which are $R$-modules themselves. This means that even cohomology (and hence the even filtration) can be efficiently computed through resolutions in the $\infty$-category of $R$-modules, as we now explain. 

Let $M$ be homologically even. By iteratively attaching even cells in $R$-modules along each odd degree homotopy class we construct a map $M \rightarrow E_{0}$ into a module with even homotopy groups with the property that the cofibre $C_{0} \colonequals \mathrm{cofib}(M \rightarrow E_{0})$ is homologically even. We can then attach even cells to $C_{0}$ to obtain a map $C_{0} \rightarrow E_{1}$ with $\pi_{*}E_{1}$ even and $C_{2} \colonequals \mathrm{cofib}(C_{0} \rightarrow E_{1})$ homologically even. Proceeding inductively in this form we produce a diagram of $R$-modules 
\begin{equation}
\label{equation:introduction_homological_resolution_complex_diagram}
\begin{tikzcd}
	& {E_{0}} & {E_{1}} & {E_{2}} & \ldots \\
	M & {C_{0}} & {C_{1}} & {C_{2}}
	\arrow[from=2-1, to=1-2]
	\arrow[from=1-2, to=2-2]
	\arrow[from=2-2, to=1-3]
	\arrow[from=1-2, to=1-3]
	\arrow[from=1-3, to=2-3]
	\arrow[from=2-3, to=1-4]
	\arrow[from=1-3, to=1-4]
	\arrow[from=1-4, to=2-4]
	\arrow[from=2-4, to=1-5]
	\arrow[from=1-4, to=1-5]
\end{tikzcd}
\end{equation}
whose top row is a chain complex in the homotopy category. The following shows that this diagram encodes the even cohomology of $M$. 

\begin{theorem}
[{\ref{proposition:even_cohomology_computed_using_cochain_complex_of_r_modules}}]
\label{theorem:introduction_calculation_of_even_coh_through_a_cpx}
If $M$ is homologically even and (\ref{equation:introduction_homological_resolution_complex_diagram}) a diagram as described above, then there is a canonical isomorphism
\[
\evencoh^{p, q}(R, M) \simeq \mathrm{H}^{p}(\pi_{2q}E_{\bullet}),
\]
between the $(p,q)$-th even cohomology of $M$ and the cohomology of the cochain complex
\[
\pi_{2q} E_{0} \rightarrow \pi_{2q} E_{1} \rightarrow \pi_{2q} E_{2} \rightarrow \ldots 
\]
computed in the $p$-th spot. 
\end{theorem}
We also prove a more refined version of \cref{theorem:introduction_calculation_of_even_coh_through_a_cpx} which gives a description of the even filtration itself as a d\'{e}calage of an appropriate cosimplicial resolution through modules with even homotopy, see \cref{proposition:homological_resolutions_are_limits_on_even_filtrations_up_to_completion} and \cref{remark:even_filtration_of_a_homological_even_as_decalage_of_its_resolution}. 

We show that \cref{theorem:introduction_calculation_of_even_coh_through_a_cpx} gives an effective way of calculating even cohomology, particularly in the connective case. To see this, note that if both $R$ and $M$ are connective, then in the construction of $E_{0}$ of (\ref{equation:introduction_homological_resolution_complex_diagram}) one needs to use only positive-dimensional cells. Iterating this construction thus leads to a raising connectivity of the diagram, giving a vanishing line. 

\begin{theorem}[{\ref{theorem:vanishing_of_even_coh_of_connective_ring_above_milnor_line}, \ref{theorem:completness_of_the_even_filtration}, \ref{corollary:complete_convergence_of_the_even_spectral_sequence_for_homologically_even_connectives}}]
Let $R$ be a connective $\mathbf{E}_{1}$-ring and $M$ a connective, homologically even $R$-module. Then the even cohomology groups of $M$ vanish above the Milnor line; that is, we have 
\[
\evencoh^{p, q}(R, M) = 0
\]
for $p > q$. In particular, the even filtration $\fil^{*}_{ev/R}(M)$ is complete and  the even spectral sequence
\[
\evencoh^{p, q}(R, M) \Rightarrow \pi_{2q-p}M
\]
is strongly convergent. 
\end{theorem}

Strikingly, in low weights we are able to calculate the even cohomology groups explicitly with virtually no assumptions on the ring $R$.

\begin{theorem}[{\ref{proposition:weight_one_cohomology_of_a_connective_homologically_even_module}, \ref{corollary:bidegree_two_two_even_coh_of_connective_hom_even_module}}]
Let $R$ be a connective $\mathbf{E}_{1}$-ring and $M$ a connective, homologically even $R$-module. Then there are canonical isomorphisms
\begin{enumerate}
\item $\evencoh^{0,0}(R, M) \simeq \pi_{0} M$, 
\item $\evencoh^{0, 1}(R, M) \simeq \mathrm{coker}(\pi_{1} R \otimes_{\mathbb{Z}} \pi_{1} M \rightarrow \pi_{2} M)$, 
\item $\evencoh^{1, 1}(R, M) \simeq \pi_{1} M$, 
\item $\evencoh^{2, 2}(R, M) \simeq \mathrm{im} (\pi_{1} R \otimes_{\mathbb{Z}} \pi_{1} M \rightarrow \pi_{2} M)$.
\end{enumerate}
\end{theorem}

Before stating our last main result, we mention that the various notions of evenness one can attach to an $R$-module introduced in this article (perfect even, even flat, homologically even, as well as the classical notion of having homotopy groups concentrated in even degrees) interact with each other in interesting ways, and \S\ref{section:calculus_of_evenness} is devoted to the study of their relationships. For an example of the kind of result we prove, we show that if $E$ is a right $R$-module with $\pi_{*}E$ even and $M$ is homologically even, then 
\[
\pi_{*}(E \otimes_{R} M)
\]
is concentrated in even degrees if at least one of $E$ or $M$ is even flat, see \cref{proposition:tensor_characterization_of_even_flat_modules} and \cref{theorem:characterization_of_homologically_even_modules}. Moreover, these kind of tensor properties characterize these classes with respect to each other, giving more evidence that our notion of even flatness is the right one. 

These results are important in showing how even cohomology behaves under either extension or restriction of scalars attached to a map of $\mathbf{E}_{1}$-rings, which we do in detail in \S\ref{subsection:evenness_and_extension_restriction_of_scalars}. As an application of these methods, we prove the following base-change result which describes behaviour of even cohomology in a neighbourhood of the ``Milnor line'' $p = q$. 

\begin{theorem}[{\ref{theorem:base_change_around_milnor_diagonal}}]
Let $f \colon R \rightarrow S$ be a map of connective $\mathbf{E}_{1}$-rings such that $S$ is homologically even as a right $R$-module and let $M$ be an even flat $R$-module. Then the base-change of the canonical comparison map 
\[
\pi_{0}S \otimes_{\pi_{0} R} \evencoh^{p, q}(R, M) \rightarrow \evencoh^{p, q}(S, S \otimes_{R} M)
\]
is a surjection for $p \geq q-1$. 
\end{theorem}

As mentioned at the beginning of the introduction, a tantalizing prospect of having an $\mathbf{E}_{1}$-even filtration is that it allows one to define a ``motivic'' filtration on $\THH(R)$ and its variants as soon as $R$ is an $\mathbf{E}_{2}$-ring. This would be important in applications, since many chromatically important spectra (such as the Brown-Peterson spectrum or its truncated variants) cannot be made $\mathbf{E}_{\infty}$, but can often be made $\mathbf{E}_{2}$ \cite{basterra2013multiplication, senger2017brown, lawson2018secondary, hahn2022redshift}. 

As one piece of evidence that the even filtration introduced in  the present work is the right way to define prismatic cohomology of $\mathbf{E}_{2}$-ring spectra, we observe in \cref{example:even_filtration_of_thh_bpn} that the descent filtration associated to $\THH(\BPn) \rightarrow \THH(\BPn/\MU)$, used by Hahn-Wilson in their proof of Lichtenbaum-Quillen conjectures for $\BPn$ \cite{hahn2022redshift}, coincides with the even filtration $\fil^{*}_{ev}(\THH(\BPn))$. In particular, it is canonically attached to $\BPn$ as a $\mathbf{E}_{2}$-ring spectrum and does not depend on the structure of an $\MU$-algebra. To keep this article focused and at manageable length, we do not pursue the idea of prismatic cohomology of $\mathbf{E}_{2}$-ring spectra in the current work and instead will pick it up in upcoming joint work with Raksit \cite{motivic_cohomology_of_e2_rings}. 

With a nod towards future applications, observe that the only input needed to define the even filtration of \cref{definition:even_filtration_of_an_r_module} is the $\infty$-category of modules, together with the notion of a perfect even. This suggests a natural generalization of our construction to filtrations defined in other contexts, using an appropriate notion of ``evenness''. For example:

\begin{enumerate}
    \item In equivariant homotopy theory, the role of the Postnikov filtration of a spectrum is played by the equivariant slice filtration \cite{dugger2005atiyah, hill2012equivariant}. As the generating objects for the slice filtration are essentially representation spheres, this suggests that in $C_{2}$-equivariant homotopy theory, the role of perfect even $R$-modules should be played by modules built out of $\Sigma^{n \rho} R$, where $\rho$ is the regular representation of $C_{2}$. 
    \item As shown by Hahn-Raksit-Wilson, the even filtration relative to the $\infty$-category of spectra is essentially the Adams-Novikov filtration relating stable homotopy theory to formal groups. The work of Bachmann-Kong-Wang-Xu on the Chow-Novikov $t$-structure shows that similar phenomena are visible in the stable motivic category $\mathrm{SH}(k)$ if as generating objects one takes the suspension spectra of smooth projective varieties \cite{bachmann2020chow}. This suggests that the right notion of a ``perfect even'' in the motivic world should perhaps be that of a motive of a smooth projective variety. 
\end{enumerate}
Ideas related to the second observation are applied in joint work with Haine to construct spectral refinements of the weight filtration on cohomology of algebraic varieties \cite{haine2023spectral}. 

\tocless\subsection{Notation and terminology} 

If $\ccat$ is a stable $\infty$-category, we denote the corresponding mapping spaces by $\Map_{\ccat}$ and mapping spectra by $\map_{\ccat}$. If $\ccat$ is clear from the context, we sometimes drop the subscript. 

We sometimes refer to monoidal functors as \emph{strongly} monoidal to emphasize that they are monoidal and not just lax monoidal. 

If $R$ is an $\mathbf{E}_{1}$-algebra, by a module $M$ we generally mean a \emph{left} $R$-module. When we work with right $R$-modules, we will be explicit about it. We will generally identify right $R$-modules with left modules over $R^{op}$, the opposite algebra. In particular, any of the various properties of left modules we introduce (such as even flat, homologically even or perfect even) apply to right $R$-modules by considering them as left modules over $R^{op}$. 

We write $\catinfty$ for the $\infty$-category of small $\infty$-categories and $\largecatinfty$ for the $\infty$-category of large $\infty$-categories. We write $\largecatinftyL \subseteq \largecatinfty$ for the (non-full) subcategory spanned by cocomplete large $\infty$-categories and cocontinuous functors and $\PrL \subseteq \largecatinftyL$ for its full subcategory spanned by presentable $\infty$-categories. 

\tocless\subsection{Acknowledgments}

I would like to thank Ben Antieau, Robert Burklund, Jeremy Hahn, Lars Hesselholt, Jacob Lurie, Arpon Raksit, Noah Riggenbach and Dylan Wilson for insightful conversations related to this work. I would like to thank the anonymous referees for suggesting improvements. I would also like to thank Princebucks\footnote{100 Nassau St, Princeton, NJ 08542.}, where most of this paper was written. 

\vspace{8px}

\section{The even filtration} 
\label{section:the_even_filtration}

This section is devoted to the construction of the even filtration and its most basic properties. 

\begin{notation}
Throughout, $R$ denotes an $\mathbf{E}_{1}$-algebra in spectra. By an $R$-module we mean a \emph{left} module in spectra. 
\end{notation}

\subsection{Perfect even modules} 
\label{subsection:perfect_even_modules}

\begin{definition}
\label{definition:perfect_even_module}
We say that an $R$-module $A$ is \emph{perfect even} if it belongs to the smallest subcategory 
\[
\Perf(R)_{ev} \subseteq \Mod_{R}(\spectra)
\]
which contains $\Sigma^{2k} R$ for $k \in \mathbb{Z}$ and is closed under extensions and retracts. 
\end{definition} 

\begin{warning}
Beware that $\Perf(R)_{ev}$ is additive and admits even (de)suspensions, but it is usually not a stable $\infty$-category. 
\end{warning}

If the ring $R$ is understood, we write $\Perf_{ev} \colonequals \Perf(R)_{ev}$. 

\begin{definition}
\label{definition:even_epimorphism}
We say that 

\begin{enumerate}
    \item a map $f \colon A \rightarrow B$ of perfect even $R$-modules is an \emph{even epimorphism} if its fibre in $R$-modules is also perfect even, 
    \item a family of maps $\{ f_{i} \colon A_{i} \rightarrow B \}$ of perfect even modules is \emph{covering} if it consists of a single even epimorphism. 
\end{enumerate}
\end{definition}

\begin{lemma}
The notion of a covering family defines a Grothendieck pretopology on $\Perf(R)_{ev}$. 
\end{lemma}

\begin{proof}
Consulting the axioms of a pretopology, we see that we have to verify that 
\begin{enumerate}
    \item equivalences are even epimorphisms, 
    \item even epimorphisms are stable under pullback, 
    \item even epimorphisms are closed under composition. 
\end{enumerate}

The first property follows from the fact that the fibre of an equivalence is zero and hence perfect even. The second is a consequence of the fact that taking a pullback along a map does not change fibres; that is, that 
\[
\mathrm{fib}(B \rightarrow A) \simeq \mathrm{fib}(B \times_{A} C \rightarrow C).
\]
For the third property, suppose that we have two composable even epimorphisms 
\[
C \rightarrow B \rightarrow A. 
\]
We then have a cofibre sequence 
\[
\mathrm{fib}(C \rightarrow B) \rightarrow \mathrm{fib}(C \rightarrow A) \rightarrow \mathrm{fib}(B \rightarrow A).
\]
Since the outer terms are perfect even and the latter are closed under extensions, we deduce that $\mathrm{fib}(C \rightarrow A)$ is perfect even as needed. 

\end{proof}

As a consequence of the fact that covering families are singleton, we have the following elegant characterization of additive sheaves:

\begin{theorem}
\label{theorem:additive_presheaves_on_perf_ev_well_behaved}
If $\acat$ is an additive, presentable $\infty$-category, then the following holds: 
\begin{enumerate}
    \item an $\acat$-valued presheaf $X \colon \Perf_{ev}^{op} \rightarrow \acat$ is additive and a sheaf with respect to the even epimorphism topology if and only if for every even epimorphism $f \colon A \rightarrow B$, the sequence 
    \[
    X(B) \rightarrow X(A) \rightarrow X(\mathrm{fib}(f))
    \]
    is fibre, 
    \item the sheafication functor on $\acat$-valued presheaves preserves additive presheaves so that its restriction
    \[
    L_{\Sigma} \colon \presheaves_{\Sigma}(\Perf_{ev}, \acat) \rightarrow \sheaves_{\Sigma}(\Perf_{ev}, \acat)
    \]
    to additive presheaves is an exact accessible localization.
\end{enumerate}
\end{theorem}

\begin{proof}
If $\acat = \spectra_{\geq 0}$ is the $\infty$-category of connective spectra (or equivalently, the $\infty$-category of spaces, by \cite[Lemma 2.1]{pstrkagowski2018synthetic}), this is a combination of \cite[Proposition 2.5, Corollary 2.7, Theorem 2.8]{pstrkagowski2018synthetic}. 

We now argue that these two properties hold if $\acat$ is an arbitrary presentable additive $\infty$-category. For the first one, notice that $X \in \presheaves_{\Sigma}(\Perf_{ev}, \acat)$ is a sheaf if and only if the mapping space functor 
\[
\Map_{\acat}(a, X(-)) \colon \Perf_{ev}^{op} \rightarrow \spectra_{\geq 0}
\]
is a sheaf for all $a \in \acat$. Since the same is true for the condition of sending even epimorphisms to fibre sequences, we deduce the statement for $\acat$-valued sheaves. Here, we use that in an additive $\infty$-category mapping spaces canonically lift to connective spectra, and we do not need to assume that $\acat$ is presentable. 

The second statement for $\spectra_{\geq 0}$ implies that the left adjoint $L_{\Sigma}$ of the inclusion 
\[
\sheaves_{\Sigma}(\Perf_{ev}, \spectra_{\geq 0}) \rightarrow \presheaves_{\Sigma}(\Perf_{ev}, \spectra_{\geq 0})
\]
is the ordinary sheafication functor $L$ restricted to additive presheaves. It follows that we have a commutative diagram in $\PrL$, the $\infty$-category of presentable $\infty$-categories and left adjoints, of the form 
\[
\begin{tikzcd}
	\presheaves(\Perf_{ev}, \spectra_{\geq 0}) & \sheaves(\Perf_{ev}, \spectra_{\geq 0}) \\
	\presheaves_{\Sigma}(\Perf_{ev}, \spectra_{\geq 0}) & \sheaves_{\Sigma}(\Perf_{ev}, \spectra_{\geq 0})
	\arrow["{L_{\Sigma}}", from=2-1, to=2-2]
	\arrow[hook, from=2-2, to=1-2]
	\arrow["L", from=1-1, to=1-2]
	\arrow[hook, from=2-1, to=1-1]
\end{tikzcd},
\]
where the vertical arrows are the inclusions of additive (pre)sheaves into (pre)sheaves. Since an additive presentable $\infty$-category is (uniquely) a $\spectra_{\geq 0}$-module in $\PrL$ by \cite[Theorem 4-6]{gepner2016universality}, applying $- \otimes_{\spectra_{\geq 0}} \acat$ to the above commutative square we obtain a commutative diagram 
\[
\begin{tikzcd}
	\presheaves(\Perf_{ev}, \acat) & \sheaves(\Perf_{ev}, \acat) \\
	\presheaves_{\Sigma}(\Perf_{ev}, \acat) & \sheaves_{\Sigma}(\Perf_{ev}, \acat)
	\arrow[from=2-1, to=2-2]
	\arrow[hook, from=2-2, to=1-2]
	\arrow[from=1-1, to=1-2]
	\arrow[hook, from=2-1, to=1-1]
\end{tikzcd}
\]
which is the needed claim as the horizontal arrows are given by sheafication functors on $\acat$-valued (additive) presheaves. 
\end{proof}

\begin{remark}
\label{remark:universal_property_of_additive_sheaves}
Suppose $\dcat$ is an additive, cocomplete $\infty$-category and that we have a functor 
\[
F \colon \Perf(R)_{ev} \rightarrow \dcat 
\]
with left Kan extension $\mathrm{Lan}_{F} \colon \presheaves(\Perf(R)_{ev}) \rightarrow \dcat$. As a consequence of \cref{theorem:additive_presheaves_on_perf_ev_well_behaved}, $\mathrm{Lan}_{F}$ factors (necessarily uniquely) through a cocontinuous functor 
\[
\sheaves_{\Sigma}(\Perf(R)_{ev}) \rightarrow \dcat 
\]
on additive sheaves if and only if $F$ is additive and preserves cofibre sequences. 
\end{remark}

We will be mainly interested in the situation where the additive $\infty$-category is given either by abelian groups or spectra. The two are closely related: 

\begin{corollary}
\label{corollary:properties_of_the_t_structure_on_sheaves}
The $\infty$-category of additive sheaves of spectra on $\Perf(R)_{ev}$ admits a unique $t$-structure in which $X \colon \Perf(R)_{ev}^{op} \rightarrow \spectra$ is coconnective if and only if $X(A) \in \spectra_{\leq 0}$ for all perfect even $A$. Moreover:

\begin{enumerate}
    \item the sheafication functor 
    \[
    L_{\Sigma} \colon \presheaves_{\Sigma}(\Perf(R)_{ev}, \spectra) \rightarrow \sheaves_{\Sigma}(\Perf(R)_{ev}, \spectra)
    \]
    is both left and right $t$-exact, 
    \item taking homotopy groups induces a canonical equivalence
    \[
    \sheaves_{\Sigma}(\Perf(R)_{ev}, \spectra)^{\heartsuit} \simeq \sheaves_{\Sigma}(\Perf(R)_{ev}, \Ab)
    \]
    between the heart and the category of additive sheaves of abelian groups and 
    \item the $t$-structure on $\sheaves_{\Sigma}(\Perf(R)_{ev}; \spectra)$ is compatible with filtered colimits and accessible; that is, coconnective objects are closed under filtered colimits and $\sheaves_{\Sigma}(\Perf(R)_{ev}, \spectra)_{\geq 0}$ is presentable. 
\end{enumerate}
\end{corollary}

\begin{proof}
The first property follows from $(2)$ of \cref{theorem:additive_presheaves_on_perf_ev_well_behaved}. The second is a consequence of the first and the description $\presheaves_{\Sigma}(\Perf_{ev}, \spectra)^{\heartsuit} \simeq \presheaves_{\Sigma}(\Perf_{ev}, \Ab)$. 

The fact that the t-structure is compatible with filtered colimits follows from part $(1)$ of \cref{theorem:additive_presheaves_on_perf_ev_well_behaved}, since fibre sequences are stable under filtered colimits. It follows from \cite[1.4.4.13, (5)]{higher_algebra} that it is also accessible. 
\end{proof}

\subsection{Even sheaves and homology of perfect evens}  

The spectral Yoneda embedding associated to an $R$-module $M$ is the presheaf
\[
Y_{R}(M) \colon \Perf(R)_{ev}^{op} \rightarrow \spectra
\]
of spectra given by the formula
\[
Y_{R}(M)(A) \colonequals \map_{\Mod_{R}}(A, M),
\]
where the right hand side is the mapping spectrum in $R$-modules. As a consequence of \cref{theorem:additive_presheaves_on_perf_ev_well_behaved}, this is in fact a sheaf with respect to the even epimorphism topology. Of particular importance are its sheaves of homotopy groups, so that we give them a dedicated name: 

\begin{definition}
\label{definition:even_sheaf_associated_to_a_module}
The \emph{even sheaf} associated to $M$ is given by 
\[
\evensheaf_{M} \colonequals \pi_{0} Y_{R}(M)
\]
It is an additive sheaf of abelian groups on perfect even $R$-modules. More generally, for any $q \in \nicefrac{1}{2} \mathbb{Z}$, the \emph{even sheaf of weight $q$} is given by
\[
\evensheaf_{M}(q) \colonequals \pi_{2q} Y_{R}(M).
\]
\end{definition}

\begin{remark}
Our grading convention in terms of half-integer Serre twists is inspired by algebraic geometry, and it is compatible with the one employed in joint work with Hesselholt on geometry of graded-commutative rings \cite{diracgeometry1, diracgeometry2}.

In more detail, in various cohomology theories of algebraic geometry, a single twist $\mathbb{Z}(1)$ usually denotes the reduced cohomology of $\mathbb{P}^{1}$. Since the latter is topologically a $2$-sphere, it follows that weight should be in correspondence with \emph{twice} the topological dimension. 
\end{remark}

\begin{remark}
\label{remark:concrete_description_of_even_sheaves_and_suspension_isomorphism}
Concretely, $\evensheaf_{M} \colonequals L([-, M])$ is given by the sheafication of the presheaf of abelian groups given by 
\[
A \mapsto \pi_{0} \map_{\Mod_{R}}(A, M).
\]
More generally, $\evensheaf_{M}(q)$ is given by the sheafication of the presheaf 
\[
A \mapsto \pi_{2q} \map_{\Mod_{R}}(A, M) \simeq \pi_{0} \map_{\Mod_{R}}(\Sigma^{2q} A, M) \simeq \pi_{0} \map_{\Mod_{R}}(A, \Sigma^{-2q} M). 
\]
The last equivalence shows that even sheaves of non-zero weight can be identified with the even sheaf associated to (de)suspensions of $M$; that is, 
\[
\evensheaf_{M}(q) \simeq \evensheaf_{\Sigma^{-2q} M}
\]
\end{remark}

\begin{proposition} 
\label{proposition:even_sheaf_functor_is_homological}
The even sheaf functor 
\[
\evensheaf_{-} \colon \Mod_{R} \rightarrow \sheaves_{\Sigma}(\Perf_{ev}; \Ab)
\]
is homological; that is, if $M_{1} \rightarrow M_{2} \rightarrow M_{3}$ is a cofibre sequence of $R$-modules, then 
\[
\evensheaf_{M_{1}} \rightarrow \evensheaf_{M_{2}} \rightarrow \evensheaf_{M_{3}}
\]
is exact in the middle as a sequence of sheaves of abelian groups.
\end{proposition}

\begin{proof}
This is clear from \cref{remark:concrete_description_of_even_sheaves_and_suspension_isomorphism}, since sheafication is exact. 
\end{proof}

\begin{remark}
\label{remark:long_exact_sequence_of_even_sheaves}
Note that by extending $M_{1} \rightarrow M_{2} \rightarrow M_{3}$ using (de)suspensions, we see that associated to a cofibre sequence we in fact have a long exact sequence of the form
\[
\ldots \rightarrow \evensheaf_{M_{2}}(\nicefrac{1}{2}) 
 \rightarrow \evensheaf_{M_{3}}(\nicefrac{1}{2}) \rightarrow \evensheaf_{M_1} \rightarrow \evensheaf_{M_2} \rightarrow \evensheaf_{M_3} \rightarrow \evensheaf_{M_1}(\nicefrac{-1}{2}) \rightarrow \evensheaf_{M_2}(\nicefrac{-1}{2}) \rightarrow \ldots,
\]
where we use the isomorphism of \cref{remark:concrete_description_of_even_sheaves_and_suspension_isomorphism}. 
\end{remark}

\begin{construction}[Local grading of the sheaf $\infty$-category] 
The double suspension functor on $\Perf(R)_{ev}$ induces via precomposition a $t$-exact autoequivalence of the $\infty$-category of sheaves which we denote by 
\[
(-)(1) \colon \sheaves_{\Sigma}(\Perf_{ev}(R), \spectra) \rightarrow \sheaves_{\Sigma}(\Perf_{ev}(R), \spectra). 
\]
For $n \in \mathbb{Z}$, we denote the $n$-fold composite of this functor with itself by $(-)(n)$. Explicitly, for any sheaf $X$ and any $A \in \Perf_{ev}(R)$ we have 
\[
(X(n))(A) \colonequals X(\Sigma^{2n} A). 
\]
\end{construction}

\begin{remark}
\label{remark:local_grading_of_even_sheaves_compatible_with_weight_notation}
Our notation concerning the local grading is compatible with that of \cref{definition:even_sheaf_associated_to_a_module} in the sense that if $M$ is an $R$-module then 
\[
\evensheaf_{M}(\nicefrac{k}{2})(n) \simeq \evensheaf_{M}(\nicefrac{k}{2}+n).
\]
\end{remark}

We will be mainly interested in the even filtration in case of the following class of modules: 

\begin{definition}
\label{definition:homologically_even_module}
We say that an $R$-module $M$ is \emph{homologically even} if $\evensheaf_{M}(q) = 0$ for any half-weight $q \in \nicefrac{1}{2} + \mathbb{Z}$. 
\end{definition}

\begin{remark}
\label{remark:homologically_even_is_vanishing_of_one_half_weight_homology}
By \cref{remark:local_grading_of_even_sheaves_compatible_with_weight_notation}, $M$ is homologically even if and only if  $\evensheaf_{M}(\nicefrac{-1}{2}) = 0$. 
\end{remark}

\begin{lemma}
\label{lemma:every_perfect_even_is_homogically_even}
Let $A$ be a perfect even $R$-module. Then $A$ is homologically even. In particular, $R$ is homologically even as a module over itself. 
\end{lemma}

\begin{proof}
By \cref{remark:homologically_even_is_vanishing_of_one_half_weight_homology}, it is enough to show the vanishing of $\evensheaf_{A}(\nicefrac{-1}{2})$, which we can identify with the sheafication of the presheaf 
\[
(B \in \Perf_{ev}) \mapsto (\pi_{0} \map_{\Mod_{R}}(B, \Sigma A) \in Ab). 
\]
Thus, we have to show that every homotopy class of maps $B \rightarrow \Sigma A$ is locally zero in the even epimorphism topology. However, we have a cofibre sequence 
\[
A \rightarrow F \rightarrow B \rightarrow \Sigma A
\]
and $F \rightarrow B$ is the required even epimorphism, since its fibre is perfect even. 
\end{proof}

We also verify that the notion of an even epimorphism introduced in \cref{definition:even_epimorphism} is the one detected by even sheaves. 

\begin{lemma}
\label{lemma:characterization_of_even_epimorphisms_as_even_homology_epimorphism}
A map $B \rightarrow A$ of perfect even $R$-modules is an even epimorphism (that is, has a perfect even fibre) if and only if $\evensheaf_{B} \rightarrow \evensheaf_{A}$ is an epimorphism in $\sheaves_{\Sigma}(\Perf_{ev}, \Ab)$. 
\end{lemma}

\begin{proof}
If $B \rightarrow A$ is an even epimorphism, then by definition it is a singleton covering family in the even epimorphism topology. It follows that the induced map of sheaves is an epimorphism. 

Conversely, suppose that $\evensheaf_{B} \rightarrow \evensheaf_{A}$ is an epimorphism of sheaves. It follows that there exists a commutative diagram of perfect even spectra 
\[
\begin{tikzcd}
	& C \\
	B & A
	\arrow["p", from=1-2, to=2-2]
	\arrow["q"', from=2-1, to=2-2]
	\arrow["s"', dashed, from=1-2, to=2-1]
\end{tikzcd}
\]
with $p$ an even epimorphism. We can extend $p$ and $q$ to a larger diagram 
\[
\begin{tikzcd}
	& D & D \\
	F & {B\times_{A}C} & C \\
	F & B & A
	\arrow["p", from=2-3, to=3-3]
	\arrow["q"', from=3-2, to=3-3]
	\arrow[from=3-1, to=3-2]
	\arrow["{q'}", from=2-2, to=2-3]
	\arrow[from=2-2, to=3-2]
	\arrow[from=1-3, to=2-3]
	\arrow[equal, from=1-2, to=1-3]
	\arrow[from=1-2, to=2-2]
	\arrow[from=2-1, to=2-2]
	\arrow[equal, from=2-1, to=3-1]
\end{tikzcd}
\]
where both lower rows and two right columns are cofibre sequences of $R$-modules. Since $p$ is assumed to be an even epimorphism, $D$ is perfect even and hence so is $B \times_{A} C$ as an extension. The map $s$ in the original triangle provides a splitting of $q'$ and we deduce that $F$ is a retract of $B \times_{A} C$ and hence it is also perfect even, which is what we wanted to show. 
\end{proof}

\subsection{The even filtration and even cohomology} 

Throughout this section, we will consider $R$ to be fixed and write $\Perf_{ev} \colonequals \Perf(R)_{ev}$. 

\begin{notation}
\label{notation:section_of_a_sheaf}
If $X \colon \Perf_{ev}^{op} \rightarrow \spectra$ is a sheaf and $A$ is perfect even, we will write 
\[
\Gamma_{\Perf_{ev}}(A, X) \colonequals X(A)
\]
for the sections of $X$ over $A$. 
\end{notation}

If $M$ is an $R$-module and $Y$ is its spectral Yoneda embedding, we have a canonical identification of spectra 
\[
\Gamma_{\Perf_{ev}}(R, Y_{R}(M)) \colonequals Y_{R}(M)(R) \simeq \map_{\Mod_{R}}(R, M) \simeq M. 
\]
This is the key insight that allows one to define the even filtration. Indeed, as the left hand side is given by sections of a sheaf of spectra, it has a canonical filtration induced by the $t$-structure: 

\begin{definition}
\label{definition:even_filtration_of_an_r_module}
Let $R$ be an $\mathbf{E}_{1}$-ring and $M$ be an $R$-module. The \emph{even filtration} of $M$ is given by the filtered spectrum
\[
\fil_{ev}^{q} (M) \colonequals \Gamma_{\Perf_{ev}(R)}(R, \tau_{\geq 2q} Y_{R}(M)),
\]
where the connective covers $\tau_{\geq 2q} Y_{R}(M)$ are calculated in the sheaf $\infty$-category. 
\end{definition}

\begin{remark}
The use of connective covers in \cref{definition:even_filtration_of_an_r_module} is similar to the construction of the Bhatt-Morrow-Scholze filtration on $\mathrm{THH}$ and its variants through the use of the quasisyntomic site \cite{bhatt2019topological}. The difference is that in the present case the site $\Perf(R)_{ev}$ used to define the filtration depends on the ring itself; on the other hand, it is somewhat linear in nature. 
\end{remark}

We record that the even filtration is exhaustive and commutes with filtered colimits. 

\begin{proposition}
\label{proposition:connectivity_of_even_filtration_maps_and_exhaustivity_of_the_even_filtration}
The canonical maps 
\[
\pi_{k} \fil^{q}_{ev}(M) \rightarrow \pi_{k} M
\]
induced by $\tau_{\geq 2q} Y_{R}(M) \rightarrow Y_{R}(M)$ are an isomorphism for $k \geq 2q$ and injective for $k = 2q-1$. In particular 
\[
\varinjlim \fil_{ev}^{q} M \simeq M. 
\]
\end{proposition}

\begin{proof}
Since sheafication is right $t$-exact, the cofibre of $\tau_{\geq 2q} Y_{R}(M) \rightarrow Y_{R}(M)$ in sheaves is $(2q-1)$-coconnective. As coconnectivity in sheaves is detected pointwise, the claim follows.
\end{proof}

\begin{lemma}
Let $M \simeq \varinjlim M_{\alpha}$ be a filtered colimit of $R$-modules. Then $\fil^{*}_{ev}(M) \simeq \varinjlim \fil^{*}_{ev}(M_{\alpha})$. 
\end{lemma}

\begin{proof}
This is immediate from property $(3)$ in \cref{corollary:properties_of_the_t_structure_on_sheaves}. 
\end{proof}

As the even filtration is induced by the Postnikov filtration in sheaves, its associated graded object can be described in terms of sheaf cohomology, which we now recall.

\begin{recollection}
\label{recollection:sheaf_cohomology}
The functor $\Gamma_{\Perf_{ev}}(R, -) \colon \sheaves_{\Sigma}(\Perf_{ev}, \abeliangroups) \rightarrow \abeliangroups$ is left exact and so admits right derived functors, which are known as sheaf cohomology and denoted by $\Hrm^{p}_{\Perf_{ev}}(R, -)$. Since 
\[
\Gamma_{\Perf_{ev}}(R, -) \simeq \Hom_{\sheaves_{\Sigma}(\Perf_{ev}, \abeliangroups)}(\evensheaf_{R}, -) 
\]
by the Yoneda lemma, for any sheaf of abelian groups $\evensheaf$ we have an isomorphism 
\[
\Hrm^{p}_{\Perf_{ev}}(R, \evensheaf) \simeq \Ext^{p}_{\sheaves_{\Sigma}(\Perf_{ev}, \abeliangroups)}(\evensheaf_{R}, \evensheaf)
\]
between sheaf cohomology and extension groups in the category of additive sheaves of abelian groups. 
\end{recollection} 

\begin{definition}
\label{definition:even_cohomology_with_coefficients_in_a_module}
If $M$ is an $R$-module, the \emph{even cohomology of $R$ with coefficients in $M$} is given by sheaf cohomology groups 
\[
\evencoh^{p, q}(R, M) \colonequals \Hrm^{p}_{\Perf_{ev}}(R, \evensheaf_{M}(q)).
\]
If the ring $R$ is understood, we write 
\[
\evencoh^{p, q}(M) \colonequals \evencoh^{p, q}(R, M).
\]
\end{definition}

\begin{remark}
\label{remark:even_cohomology_of_suspension_in_terms_of_weights}
As a consequence of \cref{remark:concrete_description_of_even_sheaves_and_suspension_isomorphism}, for an arbitrary $R$-module $M$ we have 
\[
\evencoh^{p, q}(M) \simeq \evencoh^{p, q+\nicefrac{1}{2}}(\Sigma M);
\]
that is, the suspension functor corresponds to changing the weight in even cohomology. 
\end{remark}
We observe that unlike the even filtration, which involves sheaves of spectra indexed by an $\infty$-site, the even cohomology itself is a ``mildly'' derived phenomenon, as it involves only classical derived functors in the category of sheaves of abelian groups. More precisely, since sheaf cohomology is given by $\Ext$-groups, it can be identified 
\[
\Hrm^{p}_{\Perf_{ev}}(\evensheaf) \simeq \pi_{-p} \map_{\dcat(\sheaves_{\Sigma}(\Perf_{ev}, \abeliangroups))}(\evensheaf_{R}, \evensheaf_{M}) 
\]
with homotopy of the mapping space of the derived $\infty$-category, which is much more simple than the $\infty$-category of sheaves of spectra. However, the two are closely related, as we now show, which makes the even filtration more computable in practice than one might assume at first. 

\begin{notation}
\label{notation:sheaf_of_spectra_associated_to_a_sheaf_of_abelian_groups}
We write $i \colon \sheaves(\Perf(R)_{ev}^{op}, \abeliangroups) \rightarrow\sheaves(\Perf(R)_{ev}, \spectra)$ for the fully faithful embedding which identifies the source with the heart of the target with respect to the t-structure of \cref{corollary:properties_of_the_t_structure_on_sheaves}. Explicitly, $i$ is given by considering a sheaf of abelian groups as a presheaf of spectra with homotopy groups concentrated in degree zero and sheafifying. 
\end{notation}

\begin{lemma}
\label{lemma:for_perf_ev_sheaf_cohomology_is_section_of_a_sheaf_in_the_heart}
Let $\evensheaf \colon \Perf_{ev}^{op} \rightarrow \abeliangroups$ be a sheaf of abelian groups. Then, there's a canonical isomorphism 
\begin{equation}
\label{equation:iso_between_sheaf_cohomology_of_perf_r_and_values_of_associated_sheaf_of_spectra}
\Hrm^{\ast}_{\Perf_{ev}}(R, \evensheaf) \simeq \pi_{-\ast} \Gamma_{\Perf_{ev}}(R, i \evensheaf),
\end{equation}
where the left hand side is the sheaf cohomology of \cref{recollection:sheaf_cohomology} and the right hand side is given by the value at $R$ of the associated sheaf of spectra.
\end{lemma}

\begin{proof}
Since sheafifying a sheaf of abelian groups considered as a presheaf of spectra with homotopy concentrated in degree zero doesn't change $\pi_{0}$, there's a preferred isomorphism 
\[
\Hrm^{0}_{\Perf_{ev}}(R, \evensheaf) \simeq \pi_{0} \Gamma_{\Perf_{ev}}(R, i(\evensheaf)).
\]
Since $i$ takes short exact sequences to cofibre sequences, the long exact sequence of homotopy gives a structure of a $\delta$-functor on the right hand side of (\ref{equation:iso_between_sheaf_cohomology_of_perf_r_and_values_of_associated_sheaf_of_spectra}), so that the isomorphism in degree zero determines a natural transformation of $\delta$-functors 
\[
\Hrm^{\ast}_{\Perf_{ev}}(R, \evensheaf) \rightarrow \pi_{-\ast} \Gamma_{\Perf_{ev}}(R, i(\evensheaf)),
\]
see \cite{grothendieck1957quelques}. 

By \cite[\href{https://stacks.math.columbia.edu/tag/010T}{Tag 010T}]{stacks-project}, to show that the given natural transformation is an isomorphism, it is enough to show that if $\mathcal{I} \in \sheaves_{\Sigma}(\Perf_{ev}, \abeliangroups)$ is an injective object, then 
\[
\pi_{-\ast} \Gamma_{\Perf_{ev}}(R, i(\mathcal{I})) \simeq \pi_{-\ast} (i(\mathcal{I})(R))
\]
vanishes for $\ast > 0$. This is equivalent to showing that if $I$ is considered as a presheaf of spectra concentrated in degree zero, then it is already a sheaf. By the criterion of \cref{theorem:additive_presheaves_on_perf_ev_well_behaved}, this is equivalent to 
\begin{equation}
\label{equation:short_exact_sequence_needed_to_show_even_cohomology_is_sheaf_cohomology}
0 \rightarrow \mathcal{I}(A_{3}) \rightarrow \mathcal{I}(A_{2}) \rightarrow \mathcal{I}(A_{1}) \rightarrow 0
\end{equation}
being short exact for any cofibre sequence $A_{1} \rightarrow A_{2} \rightarrow A_{3}$ of perfect even spectra. Since
\[
\mathcal{I}(A) \simeq \Hom_{\sheaves_{\Sigma}(\Perf_{ev}, \abeliangroups)}(\evensheaf_{A}, \mathcal{I})
\]
by the Yoneda lemma for any perfect even $A$, exactness of (\ref{equation:short_exact_sequence_needed_to_show_even_cohomology_is_sheaf_cohomology}) follows from injectivity of $\mathcal{I}$ and the fact that 
\[
0 \rightarrow \evensheaf_{A_{1}} \rightarrow \evensheaf_{A_{2}} \rightarrow \evensheaf_{A_{3}} \rightarrow 0 
\]
is short exact which follows from \cref{lemma:every_perfect_even_is_homogically_even}
 and \cref{remark:long_exact_sequence_of_even_sheaves}. 
\end{proof}

\begin{warning}
The analogue of \cref{lemma:for_perf_ev_sheaf_cohomology_is_section_of_a_sheaf_in_the_heart} (that is, an isomorphism between sheaf cohomology and the homotopy groups of sheafication in spectra) is automatic for sheaves on an ordinary site; that is, for sheaves on a classical category. However, beware that it need not hold in general for sheaves on $\infty$-sites! 

Indeed, sheaf cohomology is an invariant of the abelian category of sheaves of abelian groups, which only depends on the homotopy category of the $\infty$-site in question. Consequently, \cref{lemma:for_perf_ev_sheaf_cohomology_is_section_of_a_sheaf_in_the_heart} can be interpreted as saying that the $\infty$-site $\Perf(R)_{ev}$ has favourable properties. 
\end{warning}

\begin{theorem}
\label{theorem:associated_graded_of_even_filtration_and_even_cohomology}
Let $M$ be a homologically even $R$-module, so that $\evensheaf_{M}(\nicefrac{k}{2}) = 0$ for all odd $k$. We then have canonical isomorphism 
\[
\Hrm^{p, q}_{ev}(R, M) \simeq \pi_{2q-p} \gr^{q}_{ev}(M) 
\]
between the even cohomology of \cref{definition:even_cohomology_with_coefficients_in_a_module} and the homotopy of the associated graded of the even filtration. 
\end{theorem}

\begin{proof}
The even filtration is induced by the Whitehead tower of $Y_{R}(M)$ of the form 
\[
\begin{tikzcd}
	\ldots & {\tau_{\geq 2q+2}Y_{R}(M)} & {\tau_{\geq 2q}Y_{R}(M)} & {\tau_{\geq 2q-2}Y_{R}(M)} & \ldots \\
	& {\Sigma^{2q+2} i(\evensheaf_{M}(q+1))} & {\Sigma^{2q} i(\evensheaf_{M}(q))} & {\Sigma^{2q-2} i(\evensheaf_{M}(q-1))}
	\arrow[from=1-4, to=2-4]
	\arrow[from=1-3, to=1-4]
	\arrow[from=1-4, to=1-5]
	\arrow[from=1-2, to=1-3]
	\arrow[from=1-1, to=1-2]
	\arrow[from=1-2, to=2-2]
	\arrow[from=1-3, to=2-3]
\end{tikzcd},
\]
where the bottom terms are the associated graded, so that the composite of any horizontal map followed by a vertical map is a cofibre sequence. Here, the associated graded is given by (de)suspensions of the even sheaves of \cref{definition:even_sheaf_associated_to_a_module}, considered as sheaves of spectra using the embedding $i$ of \cref{notation:sheaf_of_spectra_associated_to_a_sheaf_of_abelian_groups}. Applying the functor $\Gamma_{\Perf_{ev}}(R, -)$ to the above diagram we see that 
\[
\gr^{q}_{ev}(M) \simeq \Sigma^{2q} \Gamma_{\Perf_{ev}}(R, i(\evensheaf_{M}(q)))
\]
and the needed statement follows from \cref{lemma:for_perf_ev_sheaf_cohomology_is_section_of_a_sheaf_in_the_heart}. 
\end{proof}

\begin{remark}
\label{remark:even_cohomology_vanishes_in_negative_degrees_and_this_is_consistent_with_identification_with_graded_of_the_even_filtration}
Note that as derived functors, sheaf cohomology of \cref{recollection:sheaf_cohomology} is really only defined in non-negative degrees, but it is often convenient to follow the convention that these groups are defined and vanish when $p < 0$. With this interpretation, \cref{theorem:associated_graded_of_even_filtration_and_even_cohomology} is also true for $p < 0$, since $\gr^{q}_{ev}(M)$ is $2q$-coconnective as a spectrum of section of a sheafication of a $2q$-coconnective presheaf. 
\end{remark}

As any filtered spectrum, the even filtration defines a spectral sequence, as in \cite[\S 1.2.2]{higher_algebra}. In these terms, \cref{theorem:associated_graded_of_even_filtration_and_even_cohomology} can be interpreted as identifying the second page:

\begin{definition}
\label{definition:even_spectral_sequence}
If $M$ is homologically even, we call the spectral sequence associated to the even filtration of $M$ the \emph{even spectral sequence}. It is of signature 
\[
E_{2}^{p, q} \colonequals \Hrm_{ev}^{p, q}(R, M) \Rightarrow \pi_{2q-p}(M). 
\]
with differentials of bidegree $|d_{r}| = (2r-1, r-1)$. 
\end{definition}

\begin{remark}[Adams grading]
\label{remark:p_q_grading_convention_for_even_cohomology}
As a consequence of \cref{theorem:faithfully_flat_descent_for_modules}, in many cases the even spectral sequence of \cref{definition:even_spectral_sequence} can be identified with the Adams spectral sequence associated to a faithfully even flat map $R \rightarrow S$ into a $\pi_{*}$-even $\mathbf{E}_{1}$-ring. In particular, this happens for the sphere spectrum, in which case it can be identifed with the Adams-Novikov spectral sequence. From this perspective, our grading convention is non-standard, and in standard Adams grading $(s, t)$ we instead have 
\[
E_{2}^{s, t} \simeq \Hrm_{ev}^{s, \frac{1}{2} \cdot t}(R).
\]

Our convention is justified by the fact that even cohomology groups can often be identified with arithmetic phenomena, and our convention is closer to how things are graded in the arithmetic case. For example, for the sphere spectrum we have 
    \[
    \Hrm^{p, q}_{ev}(S^{0}) \simeq \Hrm^{p}(\mathcal{M}_{\mathrm{fg}}, \omega^{\otimes q}),
    \]
the $p$-th quasi-coherent cohomology of the moduli stack of formal groups with coefficients in the $q$-th tensor power of the canonical line bundle. 
\end{remark}

\begin{remark}[Integer and half-integer grading]
\label{remark:integer_and_half_integer_grading}
In principle, it is possible to consider the even filtration as half-integer graded; that is, it makes sense to consider
\[
\fil_{ev}^{q} \colonequals \Gamma_{\Perf_{ev}(R)}(R, \tau_{\geq 2q} Y_{R}(M))
\]
for $q \in \nicefrac{1}{2} \cdot \mathbb{Z}$. In practice, we are mainly interested in the even filtration for homologically even $M$. In this case, which includes $R$ itself by \cref{lemma:every_perfect_even_is_homogically_even}, we have 
\[
\fil_{ev}^{q}(M) \simeq \fil_{ev}^{q-\nicefrac{1}{2}}(M)
\]
for all $q \in \mathbb{Z}$, so that it is more convenient to consider the even filtration as only integer-graded. 

Our choice of notation is dictated by compatibility with the even filtration of Hahn-Raksit-Wilson \cite{hahn2022motivic} and subsequently, the various motivic filtrations. That being said, the convention of \cref{definition:even_filtration_of_an_r_module} is really most appropriate only when $M$ is homologically even. In the general case, it is preferable to work with the half-integer graded even filtration, whose associated graded is always given by the even cohomology groups (which now might be non-zero in half-integer weight). 

Similar tension exists in the classical case of $\MU$-homology, since the category of even graded $\MU_{*}\MU$-comodules has a beautiful geometric interpretation as the category of quasi-coherent sheaves on the moduli stack of formal groups. However, the language of quasi-coherent sheaves is less convenient when talking about spectra whose $\MU$-homology is not concentrated in even degrees, as one then needs to keep track of a pair of sheaves as in \cite[Lecture 11]{lurie2010chromatic}.
\end{remark}

\subsection{Example: Modules with even homotopy groups}

In case of modules with even homotopy, even cohomology takes a particularly simple form: 

\begin{lemma}
\label{lemma:pistar_even_means_homologically_even_and_even_filtration_is_whitehead}
Let $E$ be an $R$-module and suppose that $\pi_{*}E$ is concentrated in even degrees. Then 
\begin{enumerate}
    \item $E$ is homologically even, 
    \item we have $\fil_{ev}^{q} E \simeq\ \tau_{\geq 2q} E$ for all $q \in \mathbb{Z}$. 
\end{enumerate}
\end{lemma}

\begin{proof}
Since $\pi_{*}E$ is concentrated in even degrees, so is the $R$-linear cohomology group
\[
E_{R}^{*}(A) \simeq \pi_{-*} \map_{\Mod_{R}}(A, E)
\]
for any perfect even $A$. Since the half-weight even sheaves are defined as the sheafication of the presheaf of odd homotopy groups, which vanish, we deduce that they are all zero, proving $(1)$. 

For $(2)$, observe that long exact sequence of cohomology shows that if 
\[
A \rightarrow B \rightarrow C
\]
is a cofibre sequence of perfect evens, then 
\[
0 \rightarrow E_{R}^{*}(C) \rightarrow E_{R}^{*}(B) \rightarrow E_{R}^{*}(A) \rightarrow 0
\]
is short exact, as all three groups are concentrated in even degrees. It follows from the criterion of \cref{theorem:additive_presheaves_on_perf_ev_well_behaved} that the presheaf of spectra on $\Perf_{ev}$ defined by 
\[
A \mapsto \tau_{\geq 2q} (\Gamma(A, E)) \simeq \tau_{\geq 2q} (\map_{\Mod_{R}}(A, E))
\]
is a sheaf and thus 
\[
\fil_{ev}^{q} E \colonequals \Gamma(R, \tau_{\geq 2q} Y_{R}(E)) \simeq \tau_{\geq 2q} (\Gamma(R, Y_{R}(E))) \simeq \tau_{\geq 2q} (\map_{\Mod_{R}}(R, E)) \simeq \tau_{\geq 2q} E,
\]
as claimed.
\end{proof}

\begin{corollary}
\label{corollary:even_cohomology_of_even_modules}
If $E$ is an $R$-module with $\pi_{*}E$ even, then its even cohomology groups are given by 
\[
\Hrm_{ev}^{0, q}(R, E) \simeq \pi_{2q} E
\]
in cohomological degree zero and vanish otherwise. 
\end{corollary}

\begin{proof}
This follows from \cref{lemma:pistar_even_means_homologically_even_and_even_filtration_is_whitehead} and \cref{theorem:associated_graded_of_even_filtration_and_even_cohomology}.
\end{proof}

\subsection{Example: Rings with even homotopy groups}

In the previous section, we have seen that the even filtration of an $R$-module $E$ such that $\pi_{*}E$ is even coincides with the Postnikov filtration. We now verify that this conclusion can be extended to \emph{all} $R$-modules assuming that the ring itself satisfies this condition.

\begin{lemma}
\label{lemma:cofibre_seqs_of_perfect_evens_split_if_r_has_even_homotopy}
Let $R$ be an $\mathbf{E}_{1}$-ring such that $\pi_{*}R$ is even. Then every cofibre sequence 
\[
A \rightarrow B \rightarrow C 
\]
of perfect even $R$-modules splits. 
\end{lemma}

\begin{proof}
Observe that if $M$, $N$ are $R$-modules, then the property that 
\[
\pi_{*} \map_{\Mod_{R}}(N, M)
\]
is concentrated in even degrees is stable under retracts, extensions and even (de)suspensions in each variable separately. Since 
\[
\pi_{\ast} \map_{\Mod_{R}}(R, R) \simeq \pi_{\ast} R
\]
is concentrated in even degrees by assumption, we see that all pairs of perfect even modules have this property. We deduce that all odd degree maps between perfect even $R$-modules are null. In particular, this applies to the boundary map $C \rightarrow \Sigma A$, so that the cofibre sequence in the statement splits. 
\end{proof}

\begin{corollary}
\label{corollary:if_r_has_even_homotopy_then_additive_presheaves_are_even_sheaves}
Let $R$ be an $\mathbf{E}_{1}$-ring such that $\pi_{*}R$ is even. Then, every additive presheaf $X \colon \Perf(R)_{ev}^{op} \rightarrow X$ is a sheaf with respect to the even epimorphism topology. In particular, connective covers of additive sheaves are calculated pointwise. 
\end{corollary}

\begin{proof}
The first part is immediate from \cref{lemma:cofibre_seqs_of_perfect_evens_split_if_r_has_even_homotopy} and the characterization of additive sheaves given in \cref{theorem:additive_presheaves_on_perf_ev_well_behaved}. The second follows from the fact that additive presheaves are closed under taking pointwise connective covers. 
\end{proof}

\begin{proposition}
\label{proposition:even_filtration_over_r_with_even_htpy_is_the_postnikov_filtration}
Let $R$ be an $\mathbf{E}_{1}$-ring such that $\pi_{*}R$ is even. Then for every $R$-module $M$ we have
\[
\fil^{q}_{ev}(M) \simeq \tau_{\geq 2q} M.
\]
\end{proposition}

\begin{proof}
This follows from \cref{corollary:if_r_has_even_homotopy_then_additive_presheaves_are_even_sheaves}, since 
\[
\Gamma(R, \tau_{\geq 2q} Y_{R}(M)) \simeq \tau_{\geq 2q} (\map_{\Mod_{R}}(R, M)) \simeq \tau_{\geq 2q} M. 
\]
\end{proof}

\section{Multiplicative properties of the even filtration} 

In this section we establish multiplicative properties of the even filtration, namely that its formation is suitably lax symmetric monoidal. In particular, we will show that the even filtration of an $\mathbf{E}_{n}$-algebra in spectra is canonically an $\mathbf{E}_{n}$-algebra in filtered spectra.  Our analysis is essentially based on the study of monoidal properties of the $\infty$-categories of additive sheaves introduced in \S\ref{section:the_even_filtration}. 

\subsection{Generators of additive sheaves} 

Throughout this section, $R$ denotes an $\mathbf{E}_{1}$-algebra. We will describe a particularly convenient set of generators for the associated $\infty$-category of additive sheaves. 

\begin{notation}
\label{notation:r_linear_nu}
If $M \in \Mod_{R}(\spectra)$, we write
\[
Y_{R}(M) = \map_{\Mod_{R}}(-, M) \in \sheaves_{\Sigma}(\Perf(R)_{ev}, \spectra)
\]
for its spectral Yoneda embedding and 
\[
\nu_{R}(M) \colonequals \tau_{\geq 0} Y_{R}(M) 
\]
for its connective cover. 
\end{notation}

\begin{remark}
\label{remark:synthetic_analogue_in_terms_of_the_restricted_yoneda_embedding}
Since $\Perf_{ev}(R)$ is additive, by \cite[{Proposition 2.19}]{pstrkagowski2018synthetic} the stabilization functor 
\[
\sheaves_{\Sigma}(\Perf(R)_{ev}) \rightarrow \sheaves_{\Sigma}(\Perf(R)_{ev}, \spectra)
\]
from sheaves of spaces into sheaves of spectra is fully faithful and identifies the source with the connective part of the t-structure of \cref{corollary:properties_of_the_t_structure_on_sheaves}. The functor $\nu_{R}$ can be identified with the composite 
\[
\begin{tikzcd}
	{\Mod_{R}(\spectra)} & \sheaves_{\Sigma}(\Perf(R)_{ev}) & {\sheaves_{\Sigma}(\Perf(R)_{ev}, \spectra)}
	\arrow["{y_{R}}", from=1-1, to=1-2]
	\arrow[from=1-2, to=1-3]
\end{tikzcd}
\]
of the restricted Yoneda embedding and the stabilization functor. In particular, when considered as valued in the connective part of additive sheaves, $\nu_{R}$ is continuous. Beware that this is not true if we consider it as a functor valued in all additive sheaves of spectra. 
\end{remark}

A reader familiar with synthetic spectra should think of $\nu_{R}(M)$ as the $R$-linear variant of the synthetic analogue construction of \cite[{Definition 4.3}]{pstrkagowski2018synthetic}, justifying the notation.

\begin{lemma}
\label{lemma:yoneda_lemma_in_additive_sheaves}
For any $A \in \Perf(R)_{ev}$ and $X \in \sheaves_{\Sigma}(\Perf(R)_{ev}, \spectra)$, we have 
\[
\map(\nu_{R}(A), X) \simeq X(A),
\]
where the mapping spectrum on the left is the one in additive sheaves. 
\end{lemma}

\begin{proof}
This is immediate from \cref{remark:synthetic_analogue_in_terms_of_the_restricted_yoneda_embedding} and the Yoneda lemma. 
\end{proof}

\begin{lemma}
\label{lemma:compact_generators_of_additive_sheaves}
The objects $\nu_{R}(\Sigma^{2q} R)$ are compact and generate the $\infty$-category $\sheaves(\Perf(R)_{ev}, \spectra)$ under colimits and desuspensions. 
\end{lemma}

\begin{proof}
By \cref{lemma:yoneda_lemma_in_additive_sheaves}, for any additive sheaf $X$ we have $\map(\nu_{R}(\Sigma^{2q} R), X) \simeq X(\Sigma^{2q} R)$ and it follows that $\nu_{R}(\Sigma^{2q} R)$ are compact since filtered colimits in additive sheaves are calculated pointwise as a consequence of part (1) of \cref{theorem:additive_presheaves_on_perf_ev_well_behaved}. 

We move on to the claim that these objects generate. Consider the class $\ccat$ of those perfect even $R$-modules $A$ such that $\nu_{R}(A)$ belongs to the thick subcategory generated by $\nu_{R}(\Sigma^{2q} R)$. We claim that $\ccat$ contains all perfect even modules, which together with \cref{lemma:yoneda_lemma_in_additive_sheaves} implies the claim. 

By definition of perfect even modules, it is enough to show that $\ccat \subseteq \Perf(R)_{ev}$ is closed under extensions and retracts. The second property is clear and the first follows from the fact that if 
\[
C \rightarrow B \rightarrow A 
\]
is a cofibre sequence of perfect even modules then 
\[
\nu_{R}(C) \rightarrow \nu_{R}(B) \rightarrow \nu_{R}(A) 
\]
is a cofibre sequence of sheaves as a consequence of \cref{lemma:characterization_of_even_epimorphisms_as_even_homology_epimorphism}. 
\end{proof}

We record that all even connective covers of the spectral Yoneda embedding can be described in terms of the functor $\nu_{R}$: 

\begin{remark}
\label{remark:behavior_of_nu_under_double_suspension}
If $M$ is an $R$-module, we have an equivalence 
\[
\tau_{\geq 2q} Y_{R}(M) \simeq \Sigma^{2q} \nu_{R}(\Sigma^{-2q} M), 
\]
since both sides can be identified with the sheafication of the presheaf 
\[
\tau_{\geq 2q} \map_{\Mod_{R}}(-, M) \simeq \Sigma^{2q} (\tau_{\geq 0} \map_{\Mod_{R}}(-, \Sigma^{-2q} M)). 
\]
\end{remark}

\subsection{Monoidal structure on additive sheaves} 
\label{subsection:monoidal_structure_on_additive_sheaves}

Recall that the $\infty$-category of $\mathbf{E}_{1}$-rings is symmetric monoidal. Informally, if $A, B$ are $\mathbf{E}_{1}$-algebras then their tensor product $A \otimes B$ is calculated in spectra, with multiplication defined by the composite
\[
(A \otimes B) \otimes (A \otimes B) \simeq (A \otimes A) \otimes (B \otimes B) \rightarrow A \otimes B. 
\]
Associated to an $\mathbf{E}_{1}$-algebra we have the $\infty$-site of perfect even modules of \S\ref{subsection:perfect_even_modules} and the corresponding $\infty$-category of sheaves. In this section, we will show that this association can be made multiplicative.

\begin{notation}
\label{notation:presheaves_are_left_adjoint_to_inclusion_of_catl_to_cat}
We write $\largecatinfty$ for the $\infty$-category of large $\infty$-categories and $\largecatinftyL$ for the $\infty$-category of cocomplete large $\infty$-categories and cocontinuous functors. Both are symmetric monoidal, $\largecatinfty$ with respect to the cartesian product, and $\largecatinftyL$ with respect to the tensor product of \cite[{\S 4.8.1}]{higher_algebra}. The inclusion
\[
\largecatinftyL \hookrightarrow \largecatinfty
\]
is lax symmetric monoidal by construction and it admits a left adjoint 
\[
\presheaves(-) \colon \largecatinfty \rightarrow \largecatinftyL
\]
which by \cite[4.8.1.8]{higher_algebra} is strongly symmetric monoidal. Here
\[
\presheaves(\ccat) \subseteq \Fun(\ccat^{op}, \spaces) 
\]
is the full subcategory of small presheaves; that is, those which can be written as a small colimit of representables. If $\ccat$ itself is small, then $\presheaves(\ccat) = \Fun(\ccat^{op}, \spaces)$. The equivalence 
\[
\Fun^{L}(\presheaves(\ccat), \dcat) \simeq \Fun(\ccat, \dcat),
\]
where the left hand side is the subcategory of cocontinuous functors, is provided by left Kan extension along the Yoneda embedding $\ccat \hookrightarrow \presheaves(\ccat)$. 
\end{notation}

\begin{remark}
A reader might be more familiar with the tensor product of presentable $\infty$-categories. By \cite[{Proposition 4.8.1.15}]{higher_algebra}, 
\[
\PrL \subseteq \largecatinftyL 
\]
is a symmetric monoidal full subcategory. In fact, we will only be interested in the tensor product of presentable $\infty$-categories, but it is convenient to work with all of $\largecatinftyL$ so that the presheaves construction $\ccat \mapsto \presheaves(\ccat)$ can be considered as a left adjoint. 
\end{remark}

\begin{construction}
\label{construction:lax_symmetric_monoidal_structure_on_additive_sheaves_on_perfect_evens_functor}
We will describe 
\begin{enumerate}
\item a lax symmetric monoidal structure on the functor 
\[
\sheaves_{\Sigma}(\Perf(-)_{ev}) \colon \Alg_{\mathbf{E}_{1}}(\spectra) \rightarrow \largecatinftyL
\]
\item a lax symmetric monoidal natural transformation 
\[
\sheaves_{\Sigma}(\Perf(-)_{ev}) \rightarrow \Mod_{-}(\spectra)
\]
of lax symmetric monoidal functors $\Alg_{\mathbf{E}_{1}}(\spectra) \rightarrow \largecatinfty$ which pointwise can be identified with the left Kan extension of the inclusion $\Perf(-)_{ev} \rightarrow \Mod_{-}(\spectra)$.  
\end{enumerate}

Recall that the functor $\Mod_{-}(\spectra) \colon \Alg_{\mathbf{E}_{1}}(\spectra) \rightarrow \largecatinfty$ has a canonical lax symmetric monoidal structure \cite[Theorem 4.8.5.16, (4)]{higher_algebra}. Concretely, if $A, B$ are $\mathbf{E}_{1}$-algebras, the functor 
\[
\Mod_{R}(\spectra) \times \Mod_{S}(\spectra) \rightarrow \Mod_{R \otimes S}(\spectra) 
\]
determined by the lax symmetric monoidal structure is given on objects by 
\[
(M, N) \mapsto M \otimes N,
\]
where the tensor product is calculated in spectra. This functor is exact in each variable separately and the tensor product of the units is the unit in the target. It follows that it restricts to a functor
\begin{equation}
\label{equation:monoidal_structure_on_perfect_evens}
\Perf(R)_{ev} \times \Perf(S)_{ev} \rightarrow \Perf(R \otimes S)_{ev}. 
\end{equation}
Thus, a restriction of the lax symmetric monoidal structure on $\Mod_{-}(\spectra)$ provides one on the functor 
\[
\Perf(-)_{ev} \colon \Alg_{\mathbf{E}_{1}}(\spectra) \rightarrow \catinfty
\]
such that the pointwise inclusion
\begin{equation}
\label{equation:inclusion_of_perfect_evens_into_all_modules_as_a_natural_transformation}
\Perf(-)_{ev} \hookrightarrow \Mod_{-}(\spectra)
\end{equation}
is a symmetric monoidal transformation of lax symmetric monoidal functors $\Alg_{\mathbf{E}_{1}}(\spectra) \rightarrow \largecatinfty$. Since the target of (\ref{equation:inclusion_of_perfect_evens_into_all_modules_as_a_natural_transformation}) is valued in $\largecatinftyL$, using the adjunction of \cref{notation:presheaves_are_left_adjoint_to_inclusion_of_catl_to_cat} we obtain a lax symmetric monoidal structure on the association 
\[
R \in \Alg_{\mathbf{E}_{1}}(\spectra) \mapsto \presheaves(\Perf(R)_{ev}, \spectra) \in \largecatinftyL
\]
and a symmetric monoidal transformation 
\begin{equation}
\label{equation:lax_symmetric_monoidal_transformation_from_presheaves_on_perfect_evens_to_modules}
\presheaves(\Perf(-)_{ev}, \spectra) \rightarrow \Mod_{-}(\spectra)
\end{equation}
of lax symmetric monoidal functors $\Alg_{\mathbf{E}_{1}}(\spectra))$. 

Concretely, if $R, S$ are $\mathbf{E}_{1}$-rings, this lax symmetric monoidal structure determines a map 
\[
\presheaves(\Perf(R)_{ev}) \otimes \presheaves(\Perf(S)_{ev}) \rightarrow \presheaves(\Perf(R \otimes S)_{ev}),
\]
which, if we identify it with a functor 
\begin{equation}
\label{equation:lax_symmetric_monoidal_structure_on_presheaves}
\presheaves(\Perf(R)_{ev}) \times \presheaves(\Perf(S)_{ev}) \rightarrow \presheaves(\Perf(R \otimes S)_{ev})
\end{equation}
cocontinuous in each variable separately, is given by left Kan extension of (\ref{equation:monoidal_structure_on_perfect_evens}). Since the latter is additive and preserves cofibres in each variable separately, \cref{remark:universal_property_of_additive_sheaves} implies that (\ref{equation:lax_symmetric_monoidal_structure_on_presheaves}) induces through localization a cocontinuous functor 
\[
\sheaves_{\Sigma}(\Perf(R)_{ev}) \otimes \sheaves_{\Sigma}(\Perf(S)_{ev}) \rightarrow \sheaves_{\Sigma}(\Perf(R \otimes S)_{ev}). 
\]
This construction refines to a lax symmetric monoidal structure on the functor 
\[
\sheaves_{\Sigma}(\Perf(-)_{ev}; \spectra) \colon \Alg_{\mathbf{E}_{1}}(\spectra) \rightarrow \largecatinftyL.
\]
Since $\Perf(R)_{ev} \hookrightarrow \Mod_{R}(\spectra)$ is additive and preserves cofibres, (\ref{equation:lax_symmetric_monoidal_transformation_from_presheaves_on_perfect_evens_to_modules}) induces a symmetric monoidal transformation 
\[
\sheaves_{\Sigma}(\Perf(-)_{ev}) \rightarrow \Mod_{-}(\spectra) 
\]
of lax symmetric monoidal functors $\Alg_{\mathbf{E}_{1}}(\spectra) \rightarrow \largecatinftyL$, as needed. 
\end{construction}

\begin{remark}
\label{remark:monoidal_structure_on_sheaves_of_spectra_induced_from_that_of_sheaves}
For an $\mathbf{E}_{1}$-ring $R$, we have 
\[
\sheaves_{\Sigma}(\Perf(-)_{ev}, \spectra) \simeq \sheaves_{\Sigma}(\Perf(-)_{ev}) \otimes \spectra,
\]
where the tensor product on the left is that of cocomplete $\infty$-categories. Since $\spectra \in \largecatinftyL$ is a commutative algebra, \cref{construction:lax_symmetric_monoidal_structure_on_additive_sheaves_on_perfect_evens_functor} also determines a lax symmetric monoidal structure on 
\[
R \mapsto \sheaves_{\Sigma}(\Perf(R)_{ev}, \spectra).
\]
The unit map $\spaces \rightarrow \spectra$ determines a symmetric monoidal transformation 
\[
\sheaves_{\Sigma}(\Perf(R)_{ev}) \rightarrow \sheaves_{\Sigma}(\Perf(R)_{ev}, \spectra)
\]
of lax symmetric monoidal functors $\Alg_{\mathbf{E}_{1}}(\spectra) \rightarrow \largecatinftyL$. 
\end{remark}

\begin{remark}
\label{remark:explicit_functoriality_of_additive_sheaves_category_in_termsof_nu}
By unwrapping the construction, the functoriality of $R \mapsto \sheaves_{\Sigma}(\Perf(R)_{ev}, \spectra)$ provided by \cref{remark:monoidal_structure_on_sheaves_of_spectra_induced_from_that_of_sheaves} can be described explicitly as a left Kan extension. Namely, the composite 
\[
\Perf(R)_{ev} \rightarrow \sheaves_{\Sigma}(\Perf(R)_{ev}) \rightarrow  \sheaves_{\Sigma}(\Perf(R)_{ev}, \spectra) 
\]
of the Yoneda embedding and the functor of \cref{remark:monoidal_structure_on_sheaves_of_spectra_induced_from_that_of_sheaves} can be identified with the restriction of functor $\nu_{R}(-)$ of \cref{notation:r_linear_nu} to perfect evens. This implies that if $f \colon R \rightarrow S$ is a map of $\mathbf{E}_{1}$-rings, then the induced functor $f^{\ast}$ between additive sheaves of spectra is the unique cocontinuous functor such that the diagram 
\[
\begin{tikzcd}
	{\Perf(R)_{ev}} & \sheaves_{\Sigma}(\Perf(R)_{ev}, \spectra) \\
	{\Perf(S)_{ev}} & \sheaves_{\Sigma}(\Perf(S)_{ev}, \spectra)
	\arrow["{S \otimes_{R}-}", from=1-1, to=2-1]
	\arrow["\nu_{R}", from=1-1, to=1-2]
	\arrow["f^{\ast}", from=1-2, to=2-2]
	\arrow["\nu_{S}", from=2-1, to=2-2]
\end{tikzcd}
\]
commutes. 
\end{remark}

\begin{corollary}
\label{corollary:additive_sheaves_on_perf_ev_monoidal_for_en_algebras}
Let $R$ be an $\mathbf{E}_{n}$-algebra in spectra. Then $\sheaves_{\Sigma}(\Perf(R)_{ev}, \spectra)$ has a canonical presentable $\mathbf{E}_{n-1}$-monoidal structure. 
\end{corollary}

\begin{proof}
An $\mathbf{E}_{n}$-algebra in spectra can be identified with an $\mathbf{E}_{n-1}$-algebra in $\Alg_{\mathbf{E}_{1}}(\spectra)$. Since the construction $R \mapsto \sheaves_{\Sigma}(\Perf(R)_{ev}; \spectra)$ is lax symmetric monoidal by \cref{remark:monoidal_structure_on_sheaves_of_spectra_induced_from_that_of_sheaves}, the claim follows. 
\end{proof}

\begin{remark}
\label{remark:monoidal_structure_on_additive_sheaves_is_inherited_by_left_kan_extension_from_perfect_evens}
If $R$ is an $\mathbf{E}_{n}$-algebra, $\Perf_{ev}(R)$ is a monoidal subcategory of the $\mathbf{E}_{n-1}$-monoidal $\infty$-category $\Mod_{R}(\spectra)$. Using \cref{remark:explicit_functoriality_of_additive_sheaves_category_in_termsof_nu}, we see that the monoidal structure on additive sheaves of \cref{corollary:additive_sheaves_on_perf_ev_monoidal_for_en_algebras} is inherited from $\Perf(R)_{ev}$ through left Kan extension. 
\end{remark}

\subsection{A filtered variant of Schwede-Shipley} 
\label{subsection:a_filtered_variant_of_schwede_shipley}

In this section, we describe a filtered variant of the Schwede-Shipley recognition result for $\infty$-categories of modules over ring spectra \cite{schwede2003stable}. This will be used in the next section to describe the $\infty$-categories of additive sheaves in terms of filtered spectra. 

\begin{recollection}
\label{recollection:schwede_shipley_recognition}
If $\ccat$ is a presentable, stable $\infty$-category, then $\ccat$ can be identified with a module over $\spectra$ in $\PrL$, the $\infty$-category of presentable $\infty$-categories and left adjoints. The $\spectra$-module structure determines for any $a, b \in \ccat$ a mapping spectrum $\map(a, b)$. A choice of an object $c \in \ccat$, which is the same as a map of modules $\spectra \rightarrow \ccat$ in $\Pr^{L}$, determines a functor 
\begin{equation}
\label{equation:schwede_shipley_comparison_functor}
\map(c, -) \colon \ccat \rightarrow \Mod_{\map(c, c)}(\spectra)
\end{equation}
and by the Schwede-Shipley recognition theorem, the following are equivalent: 
\begin{enumerate}
    \item the functor (\ref{equation:schwede_shipley_comparison_functor}) is an equivalence of pointed $\infty$-categories, 
    \item $c$ is a compact and a generator; that is, it generates $\ccat$ under colimits and desuspensions. 
\end{enumerate}

If $R$ is an $\mathbf{E}_{1}$-ring, then $\Mod_{R}(\spectra)$ has a canonical distinguished object $R$ (considered as a left module over itself), promoting it to an $\mathbf{E}_{0}$-$\spectra$-algebra. Moreover, the resulting construction 
\[
\Mod_{-}(\spectra) \colon \Alg_{\mathbf{E_{1}}}(\spectra) \rightarrow \Alg_{\mathbf{E_{0}}}(\Mod_{\spectra}(\PrL))
\]
is fully faithful by \cite[Proposition 7.1.2.6]{higher_algebra}. In other words, $\mathbf{E}_{1}$-algebras can be identified with a subcategory of pointed $\spectra$-modules satisfying either of the above two equivalent conditions. 
\end{recollection} 

\begin{notation}
\label{notation:generators_of_filtered_spectra}
We write $\Fil \spectra \colonequals \Fun(\ZZ^{op}, \spectra)$ for the $\infty$-category of filtered spectra, where we consider $\ZZ$ as a poset. This is symmetric monoidal through left Kan extension of the group structure on the integers. For $q \in \ZZ$, we write $f_{q}(S^{0})$ for the free filtered spectrum generated in weight $q$, explicitly described by
\[
(f_{q}(S^{0}))_{k} \colonequals \begin{cases}
  S^{0}  & k \leq q \\
  0 & k > q
\end{cases}.
\]
\end{notation}

\begin{recollection}
If $\ccat$ is a $\Fil \spectra$-module in $\Pr^{L}$, then the module structure determines for any $a, b \in \ccat$ a filtered mapping spectrum $\underline{\map}^{*}(a, b)$. Concretely, we have 
\begin{equation}
\label{equation:explicit_formula_for_internal_filtered_mapping_spectrum}
\underline{\map}^{q}(a, b) \colonequals \map_{\ccat}(f_{q}(S^{0}) \otimes a, b) 
\end{equation}
where the tensor product denotes the module structure. A choice of an object $c \in \ccat$, which is the same as a map of modules $\Fil \spectra \rightarrow \ccat$ in $\Pr^{L}$, gives a functor 
\begin{equation}
\label{equation:comparison_functor_into_modules_in_filtered_spectra}
\underline{\map}^{*}(c, -) \colon \ccat \rightarrow \Mod_{\underline{\map}^{*}(c, c)}(\Fil \spectra). 
\end{equation}
\end{recollection}

\begin{proposition}[Filtered Schwede-Shipley]
\label{proposition:fitlered_schwede_shipley}
For a presentable $\infty$-category $\ccat$ equipped with a $\Fil \spectra$-module structure and a choice of an object $c \in \ccat$, the following are equivalent: 
\begin{enumerate}
    \item the comparison functor of (\ref{equation:comparison_functor_into_modules_in_filtered_spectra}) is an equivalence, 
    \item $c$ is compact and the objects $f_{q}(S^{0}) \otimes c$ for $q \in \mathbb{Z}$ generate $\ccat$ under colimits and desuspensions. 
\end{enumerate}
\end{proposition}

\begin{proof}
We will verify that the six conditions outlined in \cite[4.8.5.8]{higher_algebra} hold. Since $\ccat$ is assumed to be a $\Fil \spectra$-module in $\Pr^{L}$, the first three are automatic. Since $\Fil \spectra$ is generated under colimits and desuspensions by dualizable objects, namely $f_{q}(S^{0})$, the sixth condition also follows. 

We have to verify the fourth and fifth conditions, which are that (\ref{equation:comparison_functor_into_modules_in_filtered_spectra}) is conservative and preserves geometric realizations. Since $c$ is compact, so is $f_{q}(S^{0}) \otimes c$ for each $q \in \ZZ$, and the formula (\ref{equation:explicit_formula_for_internal_filtered_mapping_spectrum}) implies that the needed functor preserves filtered colimits. As an exact functor between presentable stable $\infty$-categories, it follows it is a left adjoint. Since $f_{q}(S^{0}) \otimes c$ jointly generate, it is also conservative, as needed. 
\end{proof}

\subsection{Digression: Synthetic spectra}
\label{subsection:synthetic_spectra_and_the_even_filtration_of_the_sphere}

Introduced in \cite{pstrkagowski2018synthetic}, the $\infty$-category of synthetic spectra is an $\infty$-categorical deformation which categorifies the Adams spectral sequence. In this short section, which is not needed in the rest of the paper, we observe that sheaves on perfect even $S^{0}$-modules can be identified with even synthetic spectra based on complex bordism $\MU$.

Since the latter essentially encode the Adams-Novikov filtration on the sphere, this identification provides some informal motivation for our construction of the even filtration as giving an ``$R$-module analogue of the Adams-Novikov filtration''.

\begin{proposition}
\label{proposition:finite_spectrum_is_even_mu_projective_iff_its_perfect_even}
We have that 
\begin{enumerate}
    \item a finite spectrum $A$ is perfect even as an $S^{0}$-module if and only its $MU$-homology is concentrated in even degrees and projective as a $\MU_{\ast}$-module, 
    \item a map between perfect even spectra is an even epimorphism if and only if it is surjective on $\MU_{*}$. 
\end{enumerate}
\end{proposition}

\begin{proof}
Since the sphere has projective $\MU$-homology concentrated in even degrees, so does any perfect even. The converse follows from \cite[Remark 6.3]{pstrkagowski2018synthetic}. 
\end{proof}

\begin{corollary}
There exists a canonical symmetric monoidal equivalence 
\[
\sheaves_{\Sigma}(\Perf(S^{0})_{ev}, \spectra) \simeq \Syn_{\MU}^{ev}
\]
between additive sheaves on perfect even $S^{0}$-modules with respect to the even epimorphism topology and the $\infty$-category  of even $\MU$-based synthetic spectra of \cite[\S 5.2]{pstrkagowski2018synthetic}. 
\end{corollary}

\begin{proof}
By definition $\Syn_{\MU}^{ev}$ is the $\infty$-category of additive sheaves of spectra on the site $\spectra_{\MU}^{fpe}$ of finite spectra with even, $\MU_{\ast}$-projective homology with coverings given by $\MU_{\ast}$-projective maps, see \cite[Definition 5.10]{pstrkagowski2018synthetic} By \cref{proposition:finite_spectrum_is_even_mu_projective_iff_its_perfect_even}, the forgetful functor
\[
\Perf(S^{0})_{ev} \rightarrow \spectra_{\MU}^{fpe} 
\]
from perfect even $S^{0}$-modules to finite spectra with even, $\MU_{\ast}$-projective homology is an equivalence of Grothendieck sites. This induces an adjoint equivalence between $\infty$-categories of sheaves. The symmetric monoidal structure on both sides is induced by left Kan extension from tensor product of finite spectra, so the equivalence is symmetric monoidal. 
\end{proof}

\subsection{Additive sheaves as modules over a filtered algebra} 
\label{subsection:additive_sheaves_as_modules_over_a_filtered_algebra}

In this section, we describe the additive sheaves in terms of modules in filtered spectra, relating them to the even filtration of \cref{definition:even_filtration_of_an_r_module}.

\begin{remark}[Previous work]
Over the sphere, additive sheaves on perfect evens can be identified with synthetic spectra as we observed by \cref{proposition:finite_spectrum_is_even_mu_projective_iff_its_perfect_even}. In this case, our results here recover the filtered description of synthetic spectra previously proven by Gheorge-Isaksen-Krause-Ricka \cite[Remark 6.13]{gheorghe2022c} and Burklund-Hahn-Senger \cite[Proposition C.22]{burklund2020galois}.
\end{remark}

\begin{construction}
\label{construction:comparison_functor_between_synthetic_spectra_and_modules_over_the_even_filtered_sphere}
Consider the $\infty$-category $\sheaves_{\Sigma}(\Perf(S^{0})_{ev}, \spectra)$ of additive sheaves on perfect even spectra. This has a canonical symmetric monoidal structure of \cref{corollary:additive_sheaves_on_perf_ev_monoidal_for_en_algebras}, which coincides with that obtained through left Kan extension from \cref{remark:monoidal_structure_on_additive_sheaves_is_inherited_by_left_kan_extension_from_perfect_evens}. The additive sheaf $Y_{S^{0}}(S^{0})$, explicitly given by the Spanier-Whitehead dual 
\[
A \mapsto A^{\ast} \colonequals \map_{\spectra}(A, S^{0})
\]
is symmetric monoidal when considered as a functor $\Perf(S^{0})_{ev}^{op} \rightarrow \spectra$. It follows from \cite[2.2.6.8]{higher_algebra} that the symmetric monoidal structure determines a commutative algebra object in the $\infty$-category of additive sheaves. 

Since taking Postnikov towers is lax symmetric monoidal, the construction 
\[
q \mapsto \tau_{\geq -2q} Y_{S^{0}}(S^{0})
\]
determines a lax symmetric monoidal functor 
\begin{equation}
\label{equation:filtered_postnikov_tower_of_sphere_as_functor_from_z_into_spectra}
\tau_{\geq -2 \ast} Y_{S^{0}}(S^{0}) \colon \ZZ \mapsto \sheaves_{\Sigma}(\Perf(S^{0})_{ev}, \spectra).
\end{equation}
Using \cref{remark:behavior_of_nu_under_double_suspension} and the fact that $\nu_{S^{0}}$ is strongly symmetric monoidal when restricted to $\Perf(S^{0})_{ev}$, we see that (\ref{equation:filtered_postnikov_tower_of_sphere_as_functor_from_z_into_spectra}) is strongly symmetric monoidal. Its left Kan extension is a a symmetric monoidal left adjoint 
\[
F: \Fil \spectra \rightarrow \sheaves_{\Sigma}(\Perf(S^{0})_{ev}, \spectra),
\]
whose action on the generators of 
(\ref{notation:generators_of_filtered_spectra}) is given by 
\[
F(f_{q}(S^{0})) = \tau_{\geq -2q} Y_{S^{0}}(S^{0}).
\]
\end{construction} 

The functor $F$ can be thought of as a map of commutative monoids in $\PrL$. We have previously observed in \cref{construction:lax_symmetric_monoidal_structure_on_additive_sheaves_on_perfect_evens_functor} and \cref{remark:monoidal_structure_on_sheaves_of_spectra_induced_from_that_of_sheaves} that the association 
\[
R \in \Alg_{\mathbf{E}_{1}}(\spectra) \mapsto \sheaves_{\Sigma}(\Perf(R)_{ev}, \spectra) \in \PrL
\]
is lax symmetric monoidal. Since $S^{0} \in \Alg_{\mathbf{E}_{1}}(\spectra)$ is the monoidal unit, it follows that for any $\mathbf{E}_{1}$-algebra $R$ the presentable $\infty$-category $\sheaves_{\Sigma}(\Perf(R)_{ev}, \spectra)$ has a canonical structure of a $\sheaves_{\Sigma}(\Perf(S^{0})_{ev}, \spectra)$-module. 

\begin{notation}
\label{notation:fil_sp_module_structure_on_additive_sheaves_over_perf_r_ev}
If $R$ is an $\mathbf{E}_{1}$-algebra, we will consider $\sheaves_{\Sigma}(\Perf(R)_{ev}, \spectra)$ as a module over $\Fil \spectra$ via restriction along functor $F$ of \cref{construction:comparison_functor_between_synthetic_spectra_and_modules_over_the_even_filtered_sphere}. 
\end{notation} 

By the filtered variant of Schwede-Shipley theorem of \cref{proposition:synthetic_spectra_are_modules_over_the_filtered_sphere}, a $\Fil \spectra$-module structure on an object of $\PrL$ can be used to identify a given $\infty$-category with modules over the filtered endomorphism ring of a generator. In our case, the needed generator will be given by $\nu_{R}(R)$ of \cref{notation:r_linear_nu}. 

\begin{lemma}
\label{lemma:postnikov_tower_of_spectra_yoneda_in_terms_of_action_of_filtered_spectra}
Let $R$ be an $\mathbf{E}_{1}$-algebra and $M$ an $R$-module. Then 
\begin{equation}
\label{equation:two_filtered_objects_one_compares_when_describing_even_filtration_as_internal_filtered_mapping_spectrum}
F(f_{-\ast} (S^{0})) \otimes \nu_{R}(M) \simeq \tau_{\geq 2*} Y_{R}(M) 
\end{equation}
as filtered objects of $\sheaves_{\Sigma}(\Perf(R)_{ev}, \spectra)$, where the tensor product is the action of \cref{notation:fil_sp_module_structure_on_additive_sheaves_over_perf_r_ev}.
\end{lemma}

\begin{proof}
The action of $\Fil \spectra$ is defined by restriction along the functor $F$ of \cref{construction:comparison_functor_between_synthetic_spectra_and_modules_over_the_even_filtered_sphere} and we have 
\[
F(f_{-q} (S^{0})) \simeq \tau_{\geq 2q} Y_{S^{0}}(S^{0}) \simeq \Sigma^{2q} \nu_{S^{0}}(\Sigma^{-2q} S^{0}) \in \sheaves_{\Sigma}(\Perf(S^{0})_{ev}, \spectra),
\]
where the first equivalence is the construction of $F$ as a left Kan extension and the second is \cref{remark:behavior_of_nu_under_double_suspension}. Thus, after unwrapping the definitions we see that the right hand side of (\ref{equation:two_filtered_objects_one_compares_when_describing_even_filtration_as_internal_filtered_mapping_spectrum}) can be identified with 
\[
(\Sigma^{2q} \nu_{S^{0}}(\Sigma^{-2q} S^{0})) \otimes \nu_{R}(M) \simeq \Sigma^{2q} \nu_{R}(\Sigma^{-2q} S^{0} \otimes M) \simeq \Sigma^{2q} \nu_{R} (\Sigma^{-2q} M) \simeq \tau_{\geq 2q} Y_{R}(M). 
\]
Here, the tensor product on the left is the action of additive sheaves for $S^{0}$ on additive sheaves for $R$, which by construction for objects of the form $\nu(-)$ restricts to the standard action of $\Perf(S^{0})_{ev}$ on $\Perf(R)_{ev}$. 

By the arguments of \cite[{Proposition 4.28, Lemma 4.29}]{pstrkagowski2018synthetic}, as $q$ changes, the induced maps on the left hand side of (\ref{equation:two_filtered_objects_one_compares_when_describing_even_filtration_as_internal_filtered_mapping_spectrum}) are connective covers, so that it can be identified with the (double-speed) Postnikov tower of $Y_{R}(M)$, as needed. 
\end{proof}
\begin{lemma}
\label{lemma:filtered_endomorphism_ring_of_canonical_generator_of_additive_sheaves_is_the_even_filtration}
Let $R$ be an $\mathbf{E}_{1}$-algebra and $M$ an $R$-module. Then there is an equivalence of filtered spectra 
\begin{equation}
\label{equation:even_filtration_as_an_internal_mapping_spectrum}
\underline{\map}^{\ast}(\nu_{R}(R), \nu_{R}(M)) \simeq \fil_{ev/R}^{\ast}(M) 
\end{equation}
between the filtered mapping spectrum induced by the $\Fil \spectra$-module structure of \cref{notation:fil_sp_module_structure_on_additive_sheaves_over_perf_r_ev} and the even filtration of \cref{definition:even_filtration_of_an_r_module}. 
\end{lemma}

\begin{proof}
By definition of the filtered mapping spectrum we have 
\[
\underline{\map}^{q}(\nu_{R}(R), \nu_{R}(R)) \simeq \map_{\sheaves_{\Sigma}(\Perf(R)_{ev}, \spectra)}(f_{q}(S^{0}) \otimes \nu_{R}(R), \nu_{R}(M)),
\]
where the tensor product is the one defined by the $\Fil \spectra$-action. Since $f_{q} S^{0}$ is invertible under the tensor product of filtered spectra with inverse, we can rewrite the above as  
\[
\map_{\sheaves_{\Sigma}(\Perf(R)_{ev}, \spectra)}(\nu_{R}(R), f_{-q} (S^{0}) \otimes \nu_{R}(M)).
\]
Using the Yoneda lemma of \cref{lemma:yoneda_lemma_in_additive_sheaves}, this can be further rewritten as
\[
\Gamma_{\Perf(R)_{ev}}(R, f_{-q} (S^{0}) \otimes \nu_{R}(M)).
\]

Since the even filtration is defined as 
\[
\Gamma_{\Perf(R)_{ev}}(R, \tau_{\geq 2q} Y_{R}(M)), 
\]
the previous paragraph shows that (\ref{equation:even_filtration_as_an_internal_mapping_spectrum}) follows from an equivalence of filtered additive sheaves 
\[
f_{-q}(S^{0}) \otimes \nu_{R}(M) \simeq \tau_{\geq 2q} Y_{R}(M),
\]
which is \cref{lemma:postnikov_tower_of_spectra_yoneda_in_terms_of_action_of_filtered_spectra}. 
\end{proof}

\begin{corollary}
\label{corollary:even_filtration_of_an_e1_ring_is_a_filtered_e1_algebra_and_the_even_filtration_of_a_module_is_a_filtered_module}
Let $R$ be an $\mathbf{E}_{1}$-algebra and $M$ be an $R$-module. Then
\begin{enumerate}
    \item $\fil^{\ast}_{ev/R}(R)$ has a canonical structure of an $\mathbf{E}_{1}$-algebra in filtered spectra, 
    \item $\fil^{\ast}_{ev/R}(M)$ has a canonical structure of a module over $\fil^{\ast}_{ev/R}(R)$. 
    \end{enumerate}
\end{corollary}

\begin{proof}
By \cref{lemma:filtered_endomorphism_ring_of_canonical_generator_of_additive_sheaves_is_the_even_filtration}, we have an equivalence
\[
\fil^{\ast}_{ev/R}(M) \simeq \underline{\map}^{\ast}(\nu_{R}(R), \nu_{R}(M)). 
\]
and both claims follow at once. 
\end{proof}

\begin{notation}
Throughout the rest of the paper, if $R$ is an $\mathbf{E}_{1}$-algebra, we will consider $\fil_{ev}^{\ast}(R)$ as a filtered $\mathbf{E}_{1}$-ring through the equivalence of \cref{corollary:even_filtration_of_an_e1_ring_is_a_filtered_e1_algebra_and_the_even_filtration_of_a_module_is_a_filtered_module}. 
\end{notation}

Using the structure of a $\Fil \spectra$-module and identifying the filtered endomorphism ring with the even filtration using \cref{lemma:filtered_endomorphism_ring_of_canonical_generator_of_additive_sheaves_is_the_even_filtration}, we obtain a comparison functor 
\begin{equation}
\label{equation:functor_from_additive_sheaves_into_modules_over_even_filtered_ring}
\underline{\map}^{\ast}(\nu_{R}(R), -) \colon \sheaves_{\Sigma}(\Perf(R)_{ev}, \spectra) \rightarrow \Mod_{\fil_{ev/R}^{\ast}(R)}(\Fil \spectra) 
\end{equation}
In \S \ref{subsection:a_filtered_variant_of_schwede_shipley}, we proved a criterion for such a functor to be an equivalence which we can now apply. 

\begin{proposition}
\label{proposition:synthetic_spectra_are_modules_over_the_filtered_sphere}
The functor of (\ref{equation:functor_from_additive_sheaves_into_modules_over_even_filtered_ring}) is an equivalence of $\infty$-categories. 
\end{proposition}

\begin{proof}
By \cref{proposition:fitlered_schwede_shipley}, we have to verify that the objects 
\[
F(f^{\ast}_{k}(S^{0})) \simeq \nu_{R}(R) \simeq \Sigma^{-2k} \nu_{S^{0}}(\Sigma^{-2k} S^{0}) 
\]
are compact and for $k \in \ZZ$ jointly generate additive sheaves under colimits and desuspensions. This is exactly \cref{lemma:compact_generators_of_additive_sheaves}. 
\end{proof}

\subsection{Monoidality of the even filtration} 

In \S\ref{subsection:additive_sheaves_as_modules_over_a_filtered_algebra}, we observed that if $R$ is an $\mathbf{E}_{1}$-algebra, then the associated $\infty$-category of additive sheaves on perfect evens can be identified with modules in filtered spectra over $\fil_{ev}^{\ast}(R)$. Using this observation, one can deduce monoidality properties of the even filtration itself. In this section, we will show that the construction 
\[
(R, M) \mapsto \fil^{\ast}_{ev/R}(M) 
\]
can be refined to a lax symmetric monoidal functor from pairs of an $\mathbf{E}_{1}$-algebra and and a module to the corresponding $\infty$-category of pairs in filtered spectra. 

To make our arguments easier to follow, we will first analyze the case of the $\mathbf{E}_{1}$-algebra even filtration $R \mapsto \fil^{\ast}_{ev/R}(R)$. 

\begin{construction}
\label{construction:lifting_additive_sheaves_to_pointed_fil_sp_modules}
Consider the lax symmetric monoidal functor $\Alg_{\mathbf{E}_{1}}(\spectra) \rightarrow \PrL$ 
\begin{equation}
\label{equation:r_goes_to_additive_sheaves_in_construction_of_lift_to_pointed_filsp_modules}
R \mapsto \sheaves_{\Sigma}(\Perf(R)_{ev}, \spectra) 
\end{equation}
of \cref{construction:lax_symmetric_monoidal_structure_on_additive_sheaves_on_perfect_evens_functor} and \cref{remark:monoidal_structure_on_sheaves_of_spectra_induced_from_that_of_sheaves}. Since $S^{0} \in \Alg_{\mathbf{E}_{1}}(\spectra)$ is the monoidal unit, (\ref{equation:r_goes_to_additive_sheaves_in_construction_of_lift_to_pointed_filsp_modules}) admits a lift to a lax symmetric monoidal functor 
\begin{equation}
\Alg_{\mathbf{E}_{1}}(\spectra) \rightarrow \Mod_{\sheaves_{\Sigma}(\Perf(S^{0})_{ev}, \spectra)}(\PrL) 
\end{equation}
which by restricting scalars along the functor $\Fil\spectra \rightarrow \sheaves_{\Sigma}(\Perf(S^{0})_{ev}, \spectra)$ of \cref{construction:comparison_functor_between_synthetic_spectra_and_modules_over_the_even_filtered_sphere} yields a lift of (\ref{equation:r_goes_to_additive_sheaves_in_construction_of_lift_to_pointed_filsp_modules}) valued in $\Fil \spectra$-modules. Since $S^{0} \in \Alg_{\mathbf{E}_{1}}(\spectra)$ is initial, every $\mathbf{E}_{1}$-ring is canonically an $\mathbf{E}_{0}$-algebra object of $\Alg_{\mathbf{E}_{1}}(\spectra)$, so that (\ref{equation:r_goes_to_additive_sheaves_in_construction_of_lift_to_pointed_filsp_modules}) admits a further lift which is a lax symmetric monoidal functor 
\[
\Alg_{\mathbf{E}_{1}}(\spectra) \rightarrow \Alg_{\mathbf{E}_{0}}(\Mod_{\Fil \spectra}(\PrL)). 
\]
\end{construction}

\begin{remark}
\label{remark:e0_algebra_structure_on_additive_sheaves_determined_by_synthetic_analogue_of_the_unit}
The structure of an $\mathbf{E}_{0}$-algebra in $\Mod_{\Fil \spectra}(\PrL)$ is the same as a choice of an object in the underlying $\infty$-category. If $R$ is an $\mathbf{E}_{1}$-ring with unit map $\eta \colon S^{0} \rightarrow R$, the corresponding object determined by the lift of 
\cref{construction:lifting_additive_sheaves_to_pointed_fil_sp_modules} is the image of the monoidal unit under the induced functor 
\[
\eta^{\ast} \colon \sheaves(\Perf(S^{0})_{ev}, \spectra) \rightarrow \sheaves(\Perf(R)_{ev}, \spectra). 
\]
Since the monoidal unit of the source is given by $\nu_{S^{0}}(S^{0})$ of \cref{notation:r_linear_nu}, the distinguished object is given by 
\[
\nu_{R}(R) \simeq \eta^{*} \nu_{S^{0}}(S^{0}) 
\]
where the equivalence is that of \cref{remark:explicit_functoriality_of_additive_sheaves_category_in_termsof_nu}. 
\end{remark}

\begin{theorem}
\label{theorem:even_filtration_is_lax_symmetric_monoidal}
The $\mathbf{E}_{1}$-algebra even filtration can be canonically refined to a lax symmetric monoidal functor 
\[
\fil_{ev/-}^{*}(-) \colon \Alg_{\mathbf{E}_{1}}(\spectra) \rightarrow \Alg_{\mathbf{E}_{1}}(\Fil \spectra). 
\]
\end{theorem}

\begin{proof}
By \cite[4.8.5.16, 4.8.5.17]{higher_algebra}, the module $\infty$-category construction determines a symmetric monoidal functor 
\begin{equation}
\label{equation:fully_faithful_embedding_of_filtered_e1_algebras_into_e0_filsp_modules}
\Mod_{-}(\Fil \spectra) \colon \Alg_{\mathbf{E}_{1}}(\Fil \spectra) \rightarrow \Alg_{\mathbf{E}_{0}}(\Mod_{\Fil \spectra}(\PrL))
\end{equation}
which is moreover fully faithful by \cite[4.8.5.20]{higher_algebra}. In \cref{construction:lifting_additive_sheaves_to_pointed_fil_sp_modules}, we refined the association
\[
R \mapsto \sheaves_{\Sigma}(\Perf(R)_{ev}, \spectra) 
\]
to a lax symmetric monoidal functor 
\[
\Alg_{\mathbf{E}_{1}}(\spectra) \rightarrow \Alg_{\mathbf{E}_{0}}(\Mod_{\Fil \spectra}(\PrL)). 
\]
Since the $\mathbf{E}_{0}$-algebra structure on additive sheaves corresponds to the choice of the object $\nu_{R}(R)$ as observed in \cref{remark:e0_algebra_structure_on_additive_sheaves_determined_by_synthetic_analogue_of_the_unit}, by \cref{proposition:synthetic_spectra_are_modules_over_the_filtered_sphere} this functor factors through the image of the fully faithful embedding (\ref{equation:fully_faithful_embedding_of_filtered_e1_algebras_into_e0_filsp_modules}) as in the diagram
\[
\begin{tikzcd}
	\Alg_{\mathbf{E}_{1}}(\spectra) & \Alg_{\mathbf{E}_{0}}(\Mod_{\Fil \spectra}(\PrL)) \\
	& \Alg_{\mathbf{E}_{1}}(\Fil \spectra) 
	\arrow[from=1-1, to=1-2]
	\arrow[hook, from=2-2, to=1-2]
	\arrow["{\exists_{!}}"', dotted, from=1-1, to=2-2]
\end{tikzcd}. 
\]
By \cref{proposition:synthetic_spectra_are_modules_over_the_filtered_sphere}, the diagonal arrow can be identified with a lax symmetric monoidal lift of the even filtration functor, as desired. 
\end{proof}

\begin{corollary}
\label{corollary:filtered_en_algebra_structure_on_even_filtration_of_en_ring}
Let $R$ be an $\mathbf{E}_{n}$-ring for $n \geq 1$. Then the even filtration $\fil^{\ast}_{ev/R}(R)$ has a canonical structure of a filtered $\mathbf{E}_{n}$-algebra. 
\end{corollary}

\begin{proof}
An $\mathbf{E}_{n}$-ring for $n \geq 1$ can be identified with an $\mathbf{E}_{n-1}$-algebra in the $\infty$-category $\Alg_{\mathbf{E}_{1}}(\spectra)$, so this follows immediately from \cref{theorem:even_filtration_is_lax_symmetric_monoidal}. 
\end{proof}

\begin{remark}
\label{remark:determining_property_of_the_even_filtration_as_a_lax_symmetric_monoidal_functor}
By construction, the functor $R \mapsto \fil^{\ast}_{ev/R}(R)$ of \cref{theorem:even_filtration_is_lax_symmetric_monoidal} has the property that 
\[
\Mod_{\fil^{\ast}_{ev/R}(R)}(\Fil \spectra) \simeq \sheaves_{\Sigma}(\Perf(R)_{ev}, \spectra) 
\]
as lax symmetric monoidal functors $\Alg_{\mathbf{E}_{1}}(\spectra) \rightarrow \PrL$, with the equivalence being given by the functor $\underline{\map}^{\ast}(\nu_{R}, -)$ of \cref{proposition:synthetic_spectra_are_modules_over_the_filtered_sphere}. The somewhat elaborate setup of \cref{construction:lifting_additive_sheaves_to_pointed_fil_sp_modules} is necessary to ensure that the data needed to establish this equivalence (namely the $\Fil \spectra$-module structure and the choice of the object $\nu_{R}(R)$) can be chosen suitably functorially. 
\end{remark}

We now move on to another way in which the even filtration is multiplicative, namely as a functor on the $\infty$-category of pairs of an algebra and a module. The strategy is analogous to that of \cref{theorem:even_filtration_is_lax_symmetric_monoidal}, but the category theory is just slightly more involved, using the notion of symmetric monoidal cocartesian fibrations which we now recall. 

\begin{recollection}
\label{recollection:symmetric_monoidal_cocartesian_fibration}
If $\ccat$ is a symmetric monoidal $\infty$-category, we say that a symmetric monoidal functor $\dcat \rightarrow \ccat$ is a \emph{symmetric monoidal cocartesian fibration} if in the corresponding diagram 
\[
\begin{tikzcd}
	{\dcat^{\otimes}} && {\ccat^{\otimes}} \\
	& {\euscr{F}\mathrm{in}_{*}}
	\arrow[from=1-1, to=2-2]
	\arrow[from=1-3, to=2-2]
	\arrow[from=1-1, to=1-3]
\end{tikzcd}
\]
of cocartesian fibrations over finite pointed sets, the horizontal arrow is also a cocartesian fibration. The Grothendieck construction for $\infty$-categories can be refined to a multiplicative variant which yields an equivalence of $\infty$-categories between symmetric monoidal cocartesian fibrations over $\ccat$ and lax symmetric monoidal functors $\ccat \rightarrow \largecatinfty$, see \cite[{Proposition A.2.1}]{hinich_rectification_of_algebras}.
\end{recollection}

\begin{notation}
\label{notation:infty_category_of_pairs_of_algebra_and_a_module}
If $\ccat$ is a presentably symmetric monoidal $\infty$-category, we write $\Mod(\ccat)$ for the corresponding $\infty$-category of algebras over the operad $\euscr{LM}$ of \cite[{Definition 4.2.1.7}]{higher_algebra}. Informally, $\Mod(\ccat)$ is the $\infty$-category of pairs $(R, M)$, where $R \in \Alg_{\mathbf{E}_{1}}(\ccat)$ and $M$ is a left $R$-module. The forgetful functor 
\begin{equation}
\label{equation:forgetful_functor_from_mod_into_e1_algebras}
\Mod(\ccat) \rightarrow \Alg_{\mathbf{E}_{1}}(\ccat)
\end{equation}
is a symmetric monoidal cocartesian fibration classifying the lax\footnote{Beware the subtle point that the association $R \mapsto \Mod_{R}(\ccat)$ is symmetric monoidal when considered as valued in $\largecatinftyL$ by \cite[Theorem 4.8.5.16, (4)]{higher_algebra}, but only lax symmetric monoidal when considered as valued in $\largecatinfty$.} symmetric monoidal functor $\Alg_{\mathbf{E}_{1}}(\ccat) \rightarrow \largecatinfty$ given by $R \mapsto \Mod_{R}(\ccat)$ in the sense of \cref{recollection:symmetric_monoidal_cocartesian_fibration}. 
\end{notation}

\begin{construction}
\label{construction:nu_as_a_map_of_cocartesian_fibrations}
Let $\sheaves_{\Sigma}^{\spectra}$ denote the symmetric monoidal cocartesian fibration classifying the functor 
\[
R \mapsto \sheaves_{\Sigma}(\Perf(R)_{ev}, \spectra) 
\]
in the sense of \cref{recollection:symmetric_monoidal_cocartesian_fibration}. We will construct a commutative diagram of symmetric monoidal $\infty$-categories
\[
\begin{tikzcd}
	{\Mod(\spectra)} & {} & \sheaves_{\Sigma}^{\spectra} \\
	& \Alg_{\mathbf{E}_{1}}(\spectra)
	\arrow[from=1-1, to=2-2]
	\arrow[from=1-3, to=2-2]
	\arrow[from=1-1, to=1-3]
\end{tikzcd}
\]
where the horizontal arrow is lax symmetric monoidal and fibrewise over $R \in \Alg_{\mathbf{E}_{1}}(\spectra)$ can be identified with the construction $M \mapsto \nu_{R}(M)$ of \cref{notation:r_linear_nu}. 

By \cref{construction:lax_symmetric_monoidal_structure_on_additive_sheaves_on_perfect_evens_functor}, we have a symmetric monoidal natural transformation 
\begin{equation}
\label{equation:natural_transformation_from_sheaves_into_modules_in_proof_of_monoidality_of_the_even_filtration}
\sheaves_{\Sigma}(\Perf(R)_{ev}) \rightarrow \Mod_{R}(\spectra)
\end{equation}
which is pointwise given by left Kan extension along the inclusion $\Perf(R)_{ev} \hookrightarrow \Mod_{R}(\spectra)$. If we write $\sheaves_{\Sigma} \rightarrow \Alg_{\mathbf{E}_{1}}(\spectra)$ for the cocartesian fibration classifying $R \mapsto \sheaves_{\Sigma}(\Perf(R)_{ev})$, (\ref{equation:natural_transformation_from_sheaves_into_modules_in_proof_of_monoidality_of_the_even_filtration}) determines a commutative diagram 
\[
\begin{tikzcd}
	{\sheaves_{\Sigma}} & {} & \Mod(Sp) \\
	& \Alg_{\mathbf{E}_{1}}(\spectra)
	\arrow[from=1-1, to=2-2]
	\arrow[from=1-3, to=2-2]
	\arrow["F", from=1-1, to=1-3]
\end{tikzcd}
\]
where $F$ is symmetric monoidal. Since (\ref{equation:natural_transformation_from_sheaves_into_modules_in_proof_of_monoidality_of_the_even_filtration}) is pointwise a left adjoint, by \cite[{Corollary 7.3.2.7}]{higher_algebra}, $F$ is also a left adjoint and moreover its right adjoint $G$ fits into a commutative diagram 
\[
\begin{tikzcd}
	{\sheaves_{\Sigma}} & {} & \Mod(Sp) \\
	& \Alg_{\mathbf{E}_{1}}(\spectra)
	\arrow[from=1-1, to=2-2]
	\arrow[from=1-3, to=2-2]
	\arrow["R", from=1-3, to=1-1]
\end{tikzcd}
\]
and is fibrewise over $\Alg_{\mathbf{E}_{1}}(\spectra)$ given by the right adjoint of (\ref{equation:natural_transformation_from_sheaves_into_modules_in_proof_of_monoidality_of_the_even_filtration}), which we can identify with the restricted Yoneda embedding. Moreover, $G$ is lax symmetric monoidal by another application of \cite[{Corollary 7.3.2.7}]{higher_algebra}. The needed functor is then given by the composite 
\[
\begin{tikzcd}
	\Mod(\spectra) & \sheaves_{\Sigma} & \sheaves_{\Sigma}^{\spectra} \\
	& \Alg_{\mathbf{E}_{1}}(\spectra)
	\arrow[from=1-1, to=2-2]
	\arrow[from=1-2, to=2-2]
	\arrow[from=1-3, to=2-2]
	\arrow[from=1-1, to=1-2]
	\arrow[from=1-2, to=1-3]
\end{tikzcd}
\]
where the second one is the one induced by \cref{remark:monoidal_structure_on_sheaves_of_spectra_induced_from_that_of_sheaves}. It can be fibrewise over $R \in \Alg_{\mathbf{E}_{1}}(\spectra)$ identified with $\nu_{R}$ since the latter is as a consequence of \cref{remark:synthetic_analogue_in_terms_of_the_restricted_yoneda_embedding} the composite 
\[
\Mod_{R} \rightarrow \sheaves_{\Sigma}(\Perf(R)_{ev}) \rightarrow \sheaves_{\Sigma}(\Perf(R)_{ev}, \spectra)
\]
of the restricted Yoneda embedding (which is right adjoint to the left Kan extension) and the stabilization functor. 
\end{construction}

\begin{remark}
\label{remark:right_adjoint_between_cocartesian_fibration_specifying_lax_functoriality_of_nu}
Beware that while the functor $F \colon \sheaves_{\Sigma} \rightarrow \Mod(\spectra)$ appearing in  \cref{construction:nu_as_a_map_of_cocartesian_fibrations} comes from an honest natural transformation 
\[
\sheaves_{\Sigma}(\Perf(-)_{ev}) \rightarrow \Mod_{-}(\spectra) 
\]
of functors $\Alg_{\mathbf{E}_{1}}(\spectra) \rightarrow \largecatinfty$, the same is not true for its right adjoint $G \colon \Mod(\spectra) \rightarrow \sheaves_{\Sigma}$, as it need not preserve cocartesian morphisms. Instead, it can be thought of as classifying a \emph{lax} natural transformation. Concretely, if $f \colon R \rightarrow S$ is a map of $\mathbf{E}_{1}$-rings, then the diagram
\[
\begin{tikzcd}
	{\sheaves_{\Sigma}(\Perf(R)_{ev})} & {\sheaves_{\Sigma}(\Perf(S)_{ev})} \\
	{\Mod_{R}(\spectra)} & {\Mod_{S}(\spectra)}
	\arrow["y_{R}", from=2-1, to=1-1]
	\arrow["y_{S}", from=2-2, to=1-2]
	\arrow["S \otimes_{R} -", from=2-1, to=2-2]
	\arrow["f^{\ast}", from=1-1, to=1-2]
\end{tikzcd},
\]
where $f^{\ast}$ is the left Kan extension from perfect evens and the vertical arrows are the restricted Yoneda embedding, need not commute! Instead, the data of $G$ provides a canonical natural transformation 
\[
f^{\ast} y_{R}(-) \rightarrow y_{S}(S \otimes_{R} -). 
\]
obtained using the push-pull formula from the corresponding commutative diagram of left adjoints defined by $F$, see \cite[\S 7.3.1]{lurie_higher_topos_theory}. Consequently, the functor $\Mod(\spectra) \rightarrow \sheaves_{\Sigma}^{\spectra}$ of \cref{construction:nu_as_a_map_of_cocartesian_fibrations} can be thought of as specifying a family of natural transformations 
\[
f^{\ast} \nu_{R}(-) \rightarrow \nu_{S}(S \otimes_{R} -). 
\]
\end{remark}

\begin{remark}
\label{remark:explicit_description_of_functoriality_of_synthetic_analogues}
The natural transformations \cref{remark:right_adjoint_between_cocartesian_fibration_specifying_lax_functoriality_of_nu}
\[
f^{\ast} \nu_{R} \rightarrow \nu_{S}(S \otimes_{R} -). 
\]
associated to a map $f \colon R \rightarrow S$ of $\mathbf{E}_{1}$-rings can be made explicit by unwrapping the construction. If $M$ is an $R$-module, by adjunction between left Kan extension and composition, to specify the needed map it is enough to define its adjoint 
\begin{equation}
\label{equation:adjoint_of_functoriality_of_nu}
\nu_{R}(M) \rightarrow f_{\ast} \nu_{S}(S \otimes_{R} M),
\end{equation}
where $f_{\ast} \colon \sheaves_{\Sigma}(\Perf(S)_{ev}, \spectra) \rightarrow \Perf(R)_{ev}, \spectra)$ is given by composition with the morphism of $\infty$-sites 
\[
S \otimes_{R} - \colon \Perf(R)_{ev} \rightarrow \Perf(S)_{ev},
\]
which preserves sheaves by \cite[{Proposition A.11}]{pstrkagowski2018synthetic}. Since the target of (\ref{equation:adjoint_of_functoriality_of_nu}) is a sheaf and $\nu_{R}(M)$ is a sheafication of $\tau_{\geq 0} \map_{\Mod_{R}}(-, M)$, it is enough to specify a map of presheaves 
\[
\map_{\Mod_{R}}(-, M) \rightarrow f_{\ast} \nu_{S}(S \otimes_{R} M).
\]
The needed morphism is given by the composite of 
\[
\tau_{\geq 0} \map_{\Mod_{R}}(-, M) \rightarrow \tau_{\geq 0} \map_{\Mod_{S}}(S \otimes_{R} -, S \otimes_{R} M)
\]
induced by functoriality of $S \otimes {R} -$, the identification 
\[
\tau_{\geq 0} \map_{\Mod_{S}}(S \otimes_{R} -, S \otimes_{R} M) \simeq f_{\ast} (\tau_{\geq 0} \map_{Mod_{S}}(-, S \otimes_{R} M))
\]
and the map 
\[
f_{\ast} (\tau_{\geq 0} \map_{Mod_{S}}(-, S \otimes_{R} M)) \rightarrow f_{\ast} \nu_{S}(S \otimes_{R} M)
\]
obtained by applying $f_{\ast}(-)$ to the sheafication map in presheaves on $\Perf(S)_{ev}$. 
\end{remark}

\begin{theorem}
\label{theorem:even_filtration_is_lax_symmetric_monoidal_as_a_functor_on_pairs}
The construction
\[
(R, M) \mapsto (\fil^{\ast}_{ev/R}(R), \fil^{\ast}_{ev/R}(M))
\]
can be refined to a lax symmetric monoidal functor 
\[
\fil_{ev}^{*}(-) \colon \Mod(\spectra) \rightarrow \Mod(\Fil \spectra)
\]
between the $\infty$-categories of pairs of an algebra and a module.
\end{theorem}

\begin{proof}
As in \cref{construction:nu_as_a_map_of_cocartesian_fibrations}, we write $\sheaves_{\Sigma}^{\spectra} \rightarrow \Alg_{\mathbf{E}_{1}}(\spectra)$ for the symmetric monoidal cocartesian fibration classifying
\[
R \mapsto \sheaves_{\Sigma}(\Perf(R)_{ev}, \spectra). 
\]
Using the natural equivalence of \cref{remark:determining_property_of_the_even_filtration_as_a_lax_symmetric_monoidal_functor}, the even filtration functor of \cref{theorem:even_filtration_is_lax_symmetric_monoidal} can be completed to a cartesian diagram of symmetric monoidal $\infty$-categories 
\[
\begin{tikzcd}
	\sheaves_{\Sigma}^{\spectra} & {\Mod(\Fil\spectra)} \\
	{\Alg_{\mathbf{E}_{1}}(\spectra)} & {\Alg_{\mathbf{E}_{1}}(\Fil\spectra)}
	\arrow[from=1-1, to=1-2]
	\arrow[from=1-1, to=2-1]
	\arrow[from=1-2, to=2-2]
	\arrow["\fil_{ev/-}^{\ast}(-)", from=2-1, to=2-2]
\end{tikzcd},
\]
where the vertical arrows are symmetric monoidal. Here the horizontal arrows are lax symmetric monoidal and the upper one is given fibrewise over $R \in \Alg_{\mathbf{E}_{1}}(\spectra)$ by the equivalence
\[
\underline{\map}^{\ast}(\nu_{R}(R), -) \colon \sheaves_{\Sigma}(\Perf(R)_{ev}, \spectra) 
\]
of \cref{proposition:synthetic_spectra_are_modules_over_the_filtered_sphere}. 
Composing the above square with the commutative triangle of \cref{construction:nu_as_a_map_of_cocartesian_fibrations} we obtain a square 
\[
\begin{tikzcd}
	\Mod(\spectra) & {\Mod(\Fil\spectra)} \\
	{\Alg_{\mathbf{E}_{1}}(\spectra)} & {\Alg_{\mathbf{E}_{1}}(\Fil\spectra)}
	\arrow[from=1-1, to=1-2]
	\arrow[from=1-1, to=2-1]
	\arrow[from=1-2, to=2-2]
	\arrow["\fil_{ev/-}^{\ast}(-)", from=2-1, to=2-2]
\end{tikzcd}.
\]
Fibrewise over $R \in \Alg_{\mathbf{E}_{1}}(\spectra)$, the upper horizontal arrow can be identified with the composite 
\[
\begin{tikzcd}
	\Mod_{R}(\spectra) & {\sheaves_{\Sigma}(\Perf(R)_{ev}, \spectra)} && {\Mod_{\fil^{\ast}_{ev/R}(R)}(\Fil \spectra)}
	\arrow["{\nu_{R}}", from=1-1, to=1-2]
	\arrow["{\underline{map}^{\ast}(\nu_{R}(R), -)}", from=1-2, to=1-4]
\end{tikzcd}
\]
which we can identify with the $R$-module even filtration by \cref{lemma:filtered_endomorphism_ring_of_canonical_generator_of_additive_sheaves_is_the_even_filtration}. This ends the argument. 
\end{proof}

\begin{corollary}
Let $R$ be an $\mathbf{E}_{n}$-ring. Then the $R$-module even filtration functor
\[
\fil_{ev/R}^{\ast}(-) \colon \Mod_{R}(\spectra) \rightarrow \Mod_{\fil_{ev/R}^{\ast}(R)}(\Fil \spectra) 
\]
is canonically $\mathbf{E}_{n-1}$-monoidal. 
\end{corollary}

\begin{proof}
We can identify $R$ with an $\mathbf{E}_{n-1}$-algebra object in $\Alg_{\mathbf{E}_{1}}(\spectra)$. Since by construction the even filtration functor of \cref{theorem:even_filtration_is_lax_symmetric_monoidal_as_a_functor_on_pairs} fits into a commutative square 
\[
\begin{tikzcd}
	\Mod(\spectra) & {\Mod(\Fil\spectra)} \\
	{\Alg_{\mathbf{E}_{1}}(\spectra)} & {\Alg_{\mathbf{E}_{1}}(\Fil\spectra)}
	\arrow[from=1-1, to=1-2]
	\arrow[from=1-1, to=2-1]
	\arrow[from=1-2, to=2-2]
	\arrow["\fil_{ev/-}^{\ast}(-)", from=2-1, to=2-2]
\end{tikzcd},
\]
where horizontal arrows are symmetric monoidal cocartesian fibrations, passing to the fibre over $R$ yields the needed $\mathbf{E}_{n-1}$-monoidal functor. 
\end{proof}

\begin{remark}
\label{remark:unwrapped_definition_of_the_functoriality_of_the_even_filtration}
The construction of the functor of \cref{theorem:even_filtration_is_lax_symmetric_monoidal_as_a_functor_on_pairs} is somewhat elaborate, but by unwrapping the definitions the induced map on even filtrations can be made explicit in the following way. 

Suppose we have an arrow $(R, M) \rightarrow (S, N)$ in $\Mod(\spectra)$, which we can identify with a map $f \colon R \rightarrow S$ of $\mathbf{E}_{1}$-rings together with a homomorphism $\alpha \colon M \rightarrow N$ of $R$-modules. Left Kan extension yields a cocontinuous functor 
\[
\sheaves_{\Sigma}(\Perf(R)_{ev}, \spectra) \rightarrow \sheaves_{\Sigma}(\Perf(R)_{ev}, \spectra) 
\]
which is a morphism of $\mathbf{E}_{0}$-$\Fil \spectra$-algebras, so that it induces a morphism of filtered mapping spectra and comes with a preferred equivalence $f^{\ast} \nu_{R}(R) \simeq \nu_{S}(S)$. The induced map of even filtrations 
\[
\fil^{\ast}_{ev/R}(M) \rightarrow \fil^{\ast}_{ev/S}(N) 
\]
can then be obtained as the composite 
\[
\underline{\map}^{\ast}_{R}(\nu_{R}(R), \nu_{R}(M)) \rightarrow \underline{\map}^{\ast}_{S}(\nu_{S}(S), f^{\ast} \nu_{R}(M)) \rightarrow \underline{\map}^{\ast}_{S}(\nu_{S}(S), \nu_{S}(N)),
\]
where the first map is obtained by applying $f^{\ast}$ and the second one is induced by composing with 
\[
f^{\ast} \nu_{R}(M) \rightarrow \nu_{S}(S \otimes_{R} M) \rightarrow \nu_{S}(N) 
\]
where the first arrow is the one provided by \cref{remark:explicit_description_of_functoriality_of_synthetic_analogues} and the second one is induced by the adjoint $S \otimes_{R} M \rightarrow N$ of $\alpha$. Here, by $\underline{\map}^{\ast}_{R}$ (resp. $\underline{\map}^{\ast}_{S}$) we denote the filtered mapping spectrum in additive sheaves over $\Perf(R)_{ev}$ (resp. $\Perf(S)_{ev}$), which we can identify with the relevant even filtrations using \cref{lemma:filtered_endomorphism_ring_of_canonical_generator_of_additive_sheaves_is_the_even_filtration}. 
\end{remark}

\section{Calculus of evenness}
\label{section:calculus_of_evenness}

This section is devoted to the study of the various notions of evenness one can attach to an $R$-module, as well as their relationships. We will be interested in three different properties:
\begin{enumerate}
    \item \emph{homologically even} modules of \cref{definition:homologically_even_module}; that is, those $M$ such that the even sheaves $\evensheaf_{M}(q)$ vanish for half-integers $q \in \mathbb{Z} + \nicefrac{1}{2}$, 
    \item \emph{even flat} modules of \cref{definition:even_flat_module}; that is, those $M$ which can be written $M \simeq \varinjlim A_{\alpha}$ as a filtered colimit of perfect even modules, 
    \item \emph{$\pi_{\ast}$-even} modules; that is, those $M$ such that $\pi_{\ast}(M)$ is concentrated in even degrees. 
\end{enumerate}
Either of the second or the third conditions implies the first one, as in the diagram
\[
(\textnormal{even flat}) \Rightarrow (\textnormal{homologically even}) \Leftarrow (\pi_{\ast}\textnormal{-even}).
\]
In general, neither of these implications can be reversed. Instead, we will show in \cref{proposition:tensor_characterization_of_even_flat_modules} that even flat and $\pi_{\ast}$-even modules can be thought of as ``orthogonal classes of modules'' with respect to the tensor product. 

\begin{notation}
In this section, $R$ is a fixed $\mathbf{E}_{1}$-algebra and by a module we mean a \emph{left} $R$-module. Similarly, unless stated otherwise by a \emph{map} we mean a morphism of modules. 

When we work with right $R$-modules, we will be explicit about it. We will generally identify right $R$-modules with left modules over $R^{op}$, the opposite algebra. In particular, any of the various properties of left modules and their maps we introduce in this section apply to right $R$-modules by considering them as left modules over $R^{op}$. 
\end{notation}

\begin{recollection}[Linear duality]
If $M$ is a left $R$-module, we write 
\[
M^{\ast} \colonequals \map_{\Mod_{R}}(M, R) 
\]
for its $R$-linear dual, which is a right $R$-module through its action on the target. Analogously, if $M$ is right $R$-module, we write 
\[
M^{\ast} \colonequals \map_{\Mod_{R^{op}}}(M, R) 
\]
for its linear dual which is a left $R$-module. When restricted to perfect modules, these two constructions induce inverse equivalences
\[
\Perf(R)^{op} \simeq \Perf(R^{op}), 
\]
which restrict to an equivalence between the subcategories of perfect evens.
\end{recollection}

\subsection{Even flat modules}

In this subsection, we introduce the notion of an even flat $R$-module, which informally is a module ``flat from the point of view of modules with even homotopy groups''. We will make the latter characterization precise in \cref{proposition:tensor_characterization_of_even_flat_modules}. 

\begin{proposition}
\label{proposition:characterization_of_even_flatness_in_terms_of_maps_factoring_through_perfect_even}
The following two conditions are equivalent for an $R$-module $M$:
\begin{enumerate}
    \item it can be written $M \simeq \varinjlim M_{\alpha}$ as a filtered colimit of perfect evens, 
    \item any map $P \rightarrow M$ from a perfect $R$-module into $M$ factors through a perfect even. 
\end{enumerate}
\end{proposition}

\begin{proof}
Since perfect $R$-modules are compact as objects of $\Mod_{R}$, certainly $(1 \Rightarrow 2)$, and we only have to prove the converse. We can write $M$ 
\[
M \simeq \varinjlim_{P \in \Perf(R)_{- /M}} P
\]
as a colimit of perfect $R$-modules indexed by the overcategory $\Perf(R)_{-/M}$. We claim that if $M$ satisfies the condition $(2)$, then the inclusion
\[
\Perf(R)^{ev}_{-/M} \rightarrow \Perf({R})_{-/M}
\]
is cofinal, which will finish the argument. By Quillen's Theorem A \cite[4.1.3.1]{lurie_higher_topos_theory}, we have to check that for any map $f \colon P \rightarrow M$ from a perfect $R$-module, the under-over-category 
\[
\Perf(R)^{ev}_{P / - / M},
\]
whose objects are commutative triangles 
\begin{equation}
\label{equation:triangle_in_proof_of_criterion_for_even_flatness}
\begin{tikzcd}
	P && A \\
	& M
	\arrow["f", from=1-1, to=2-2]
	\arrow[from=1-3, to=2-2]
	\arrow[from=1-1, to=1-3]
\end{tikzcd}
\end{equation}
with $A$ perfect even, is weakly contractible. We will show that it is filtered, which implies weak contractibility. Suppose that 
\[
p \colon K \rightarrow \Perf(R)^{ev}_{P / - /M}
\]
is a finite diagram. As $\Perf(R)$ admits finite colimits, $p$ admits a colimit $C$ in the larger $\infty$-category $\Perf(R)_{P / - /M}$, where the middle module as in (\ref{equation:triangle_in_proof_of_criterion_for_even_flatness}) is perfect, but not necessarily even. Since the map $C \rightarrow M$ factors through a perfect even $C'$ by assumption, declaring $\widetilde{p}(\triangleright) \colonequals C'$ provides the necessary extension of $p$ to a diagram $\widetilde{p} \colon K^{\triangleright} \rightarrow \Perf(R)^{ev}_{A / - /M}$. 
\end{proof}

\begin{definition}
\label{definition:even_flat_module}
We say that an $R$-module $M$ is \emph{even flat} if it satisfies the equivalent conditions of \cref{proposition:characterization_of_even_flatness_in_terms_of_maps_factoring_through_perfect_even}. 
\end{definition}

We record that the class of even flats has good closure properties. 

\begin{proposition}
\label{proposition:even_flat_modules_closed_under_extensions_retracts_and_filtered_colimits}
The full subcategory $\Mod_{R}^{e\flat} \subseteq \Mod_{R}$ spanned by even flat modules is closed under extensions, retracts and filtered colimits. 
\end{proposition}

\begin{proof}
The property $(2)$ of \cref{proposition:characterization_of_even_flatness_in_terms_of_maps_factoring_through_perfect_even} is clearly closed under retracts and filtered colimits. We will verify that it is also closed under extensions. Suppose we have a cofibre sequence 
\[
M' \rightarrow M \rightarrow M''
\]
of $R$-modules such that $M'$ and $M''$ are even flat. Let $P \rightarrow M$ be a map from a perfect; we have to show that it factors through a perfect even. Using that $M''$ is even flat, we can factor the composite $P \rightarrow M''$ through a perfect even $A''$, obtaining a diagram 
\[
\begin{tikzcd}
	\mathrm{fib}(P \rightarrow A'') & P & A'' \\
	{M'} & M & {M''} 
	\arrow[from=2-1, to=2-2]
	\arrow[from=2-2, to=2-3]
	\arrow[from=1-2, to=2-2]
	\arrow[from=1-2, to=1-3]
	\arrow[from=1-3, to=2-3]
	\arrow[from=1-1, to=2-1]
	\arrow[from=1-1, to=1-2]
\end{tikzcd}
\]
where both rows are cofibre sequences and where $A''$ is perfect even. Using that $M'$ is even flat, we can factor the map $\mathrm{fib}(P \rightarrow A'') \rightarrow M'$ through a perfect even $A'$. Finally, let $A$ be defined by a pushout diagram 
\[
\begin{tikzcd}
	\mathrm{fib}(P \rightarrow A'') & P \\
	{A'} & A 
	\arrow[from=2-1, to=2-2]
	\arrow[from=1-2, to=2-2]
	\arrow[from=1-1, to=2-1]
	\arrow[from=1-1, to=1-2]
\end{tikzcd},
\]
so that the universal property of the pushout gives a factorization $P \rightarrow A \rightarrow M$. By construction we have a cofibre sequence 
\[
A' \rightarrow A \rightarrow A''
\]
so that $A$ is perfect even, as needed. 
\end{proof}

\begin{remark}
\label{remark:perfect_flats_generated_by_even_suspensions_under_filtered_colimits_retracts_extensions}
As perfect even modules are generated under retracts and extensions by $\Sigma^{2n} R$ for $n \in \mathbb{Z}$, as a consequence of \cref{proposition:even_flat_modules_closed_under_extensions_retracts_and_filtered_colimits} the subcategory $\Mod_{R}^{e \flat} \subseteq \Mod_{R}$ of even flats can be characterized as the smallest subcategory containing $\Sigma^{2n} R$ and closed under filtered colimits, retracts and extensions. 
\end{remark}

\subsection{Homotopy even envelopes} 
\label{subsection:homotopy_even_envelopes}

As modules with even homotopy groups have a particularly simple even filtration by \cref{lemma:pistar_even_means_homologically_even_and_even_filtration_is_whitehead}, a useful way to prove results about the even filtration of an arbitrary $R$-module is to map it into a module with even homotopy groups. The following describes a particularly well-behaved way to do so: 

\begin{definition}
\label{definition:pi_star_even_envelope}
Let $M$ be an $R$-module. We say that a map $M \rightarrow E$ is a \emph{$\pi_{*}$-even envelope} if 
\begin{enumerate}
    \item the cofibre $\mathrm{cofib}(M \rightarrow E)$ is even flat, 
    \item $\pi_{*}E$ is a concentrated in even degrees,
    \item every map $M \rightarrow F$ into an $R$-module such that $\pi_{*}F$ is even can be completed to a commutative diagram 
    \[
\begin{tikzcd}
	& M \\
	{E} && {F}
	\arrow[from=2-1, to=2-3]
	\arrow[from=1-2, to=2-1]
	\arrow[from=1-2, to=2-3]
\end{tikzcd}.
\]
\end{enumerate}
\end{definition}

\begin{remark}
\label{remark:extension_condition_in_envelope_as_relative_cohomology}
Note that the condition $(3)$ of \cref{definition:pi_star_even_envelope} is equivalent to saying that for any $\pi_{*}$-even $F$, the relative $R$-linear cohomology groups 
\[
F_{R}^{*}(E, M) \simeq F_{R}^{*}(\mathrm{cofib}(M \rightarrow E)) \simeq \pi_{-*} \map_{\Mod_{R}}(\mathrm{cofib}(M \rightarrow E), F)
\]
are concentrated in even degrees. 
\end{remark}

\begin{remark}
If $M \rightarrow E$ is $\pi_{*}$-even envelope, then the map $\evensheaf_{M} \rightarrow \evensheaf_{E}$ is a monomorphism. Indeed, we have an exact sequence
\[
\evensheaf_{\mathrm{cofib}(M \rightarrow N)}(\nicefrac{1}{2}) \rightarrow \evensheaf_{M} \rightarrow \evensheaf_{N}.
\]
where the left hand side vanishes, since the cofibre is even flat. 
\end{remark}

\begin{remark}
\label{remark:pistar_even_envelopes_inherit_evenness}
Since even flat (resp. homologically even) $R$-modules are closed under extensions, if $M \rightarrow E$ is a $\pi_{*}$-even envelope and $M$ is even flat (resp. homologically even), then so is $E$. 
\end{remark}

\begin{proposition}[{Existence of homotopy even envelopes}]
\label{proposition:pi_star_even_envelopes_exist}
Let $M$ be an $R$-module. Then there exists a $\pi_{*}$-even envelope $M \rightarrow E$. Moreover, if $R$ and $M$ are connective, then there exists an envelope such that $E$ is also connective and $\pi_{0} M \simeq \pi_{0} E$. 
\end{proposition}

\begin{proof}
We will use the small object argument to construct $E$. Consider the set of homotopy classes of maps 
\[
\Sigma^{k} R \rightarrow M
\]
where $k$ is odd. Taking a direct sum of all such maps, we obtain a cofibre sequence
\[
\bigoplus \Sigma^{k_{\alpha}} R \rightarrow M \rightarrow M'.
\]
We now inductively declare $M_{0} \colonequals M$ and $M_{n+1} \colonequals (M_{n})'$ as above. We claim that $E \colonequals \varinjlim M_{n}$ has the required properties. 

We first verify that $E$ is an $\pi_{*}$-even envelope. For property $(1)$, since 
\[
\mathrm{cofib}(M \rightarrow E) \simeq \varinjlim \mathrm{cofib}(M \rightarrow M_{n})
\]
and even flat modules are stable under filtered colimits, it is enough to show that each of $\mathrm{cofib}(M \rightarrow M_{n})$ is even flat. We argue by induction. For $n=0$, this cofibre vanishes and so is even flat. For $n \geq 0$ we have a cofibre sequence
\[
\mathrm{cofib}(M \rightarrow M_{n}) \rightarrow \mathrm{cofib}(M \rightarrow M_{n+1}) \rightarrow \bigoplus \Sigma^{k_{\alpha}+1} R.
\]
As the right hand side term is a direct sum of perfect evens and even flat modules are closed under extensions by \cref{proposition:even_flat_modules_closed_under_extensions_retracts_and_filtered_colimits}, we deduce that the middle term is also even flat. 

For property $(2)$, observe that since any map $\Sigma^{k} R \rightarrow E$ where $k$ is odd factors through some $M_{n}$, it vanishes in $M_{n+1}$ and hence in $E$. Thus, $\pi_{*}E$ is even as needed. 

We move on to property $(3)$. Using \cref{remark:extension_condition_in_envelope_as_relative_cohomology}, it is enough to show that if $\pi_{*}F$ is even, then the relative $R$-linear cohomology groups $F_{R}^{*}(E, M)$ are concentrated in even degrees. Since 
\[
\mathrm{cofib}(M \rightarrow E) \simeq \varinjlim \mathrm{cofib}(M \rightarrow M_{n}), 
\]
we have a Milnor exact sequence
\[
0 \rightarrow R^{1} \varprojlim F_{R}^{*-1}(M_{n}, M) \rightarrow F_{R}^{*}(E, M) \rightarrow \varprojlim F_{R}^{*}(M_{n}, M) \rightarrow 0.
\]
Since the left hand side term vanishes on Mittag-Leffler sequences, vanishing in odd degrees will follow if we can show that 
\begin{enumerate}
    \item $F_{R}^{k}(M_{n}, M)$ vanishes for $k$ odd and 
    \item $F_{R}^{k}(M_{n+1}, M) \rightarrow F_{R}^{k}(M_{n}, M)$ is an epimorphism for $k$ even. 
\end{enumerate}
Both follow at once (in the first case, by induction) from the long exact sequence
\[
\ldots \rightarrow F_{R}^{*}(M_{n+1}, M) \rightarrow F_{R}^{*}(M_{n}, M) \rightarrow F_{R}^{*}(\bigoplus R^{k_{\alpha}}) \rightarrow F_{R}^{*+1}(M_{n+1}, M) \rightarrow F_{R}^{*+1}(M_{n}, M) \rightarrow \ldots, 
\]
since
\[
F_{R}^{*}(\bigoplus R^{k_{\alpha}}) \simeq \prod F_{R}^{*}(R^{k_{\alpha}}) \simeq \prod \pi_{*-k_{\alpha}}(F)
\]
vanishes in even degrees by assumption as all $k_{\alpha}$ are odd. 

Finally, suppose that $R$ and $M$ are connective. Then in the inductive construction above it is enough to consider homotopy clases of maps $\Sigma^{k} R \rightarrow M$ such that $k$ is odd and positive. If the above inductive construction is performed with this restriction, the result will be connective with $\pi_{0} M \simeq \pi_{0} E$. 
\end{proof}

Of particular importance is the $\pi_{*}$-even envelope of $R$ itself, which generates envelopes of any perfect even in the following sense: 

\begin{proposition}
\label{proposition:pi_star_unit_envelope_of_unit_contains_envelopes_of_perfect_evens}
Let $R \rightarrow E$ and $A \rightarrow E'$ be $\pi_{*}$-even envelopes, where $A$ is perfect even. Then, $E'$ belongs to the subcategory generated by $E$ under direct sums, even (de)suspensions and retracts. 
\end{proposition}

\begin{proof}
Let $f \colon \bigoplus \Sigma^{k_{\alpha}} R \rightarrow E'$ be a $\pi_{*}$-surjective map from a direct sum of even (de)suspensions of $R$ which exists by the assumption that $\pi_{*}(E')$ is even. By assumption, we can extend this map to a commutative diagram 
\[
\begin{tikzcd}
	& {\bigoplus \Sigma^{k_{\alpha}}R} \\
	{\bigoplus \Sigma^{k_{\alpha}} E} && {E'}
	\arrow["p", from=2-1, to=2-3]
	\arrow[from=1-2, to=2-1]
	\arrow["f", from=1-2, to=2-3]
\end{tikzcd}.
\]
We claim that the map $p$ admits a section, which will imply that $E'$ is a retract of $\bigoplus \Sigma^{k_{\alpha}} E$, proving the result. This happens precisely when the map $E' \rightarrow C$ vanishes, where
\[
C \colonequals \mathrm{cofib}(\bigoplus \Sigma^{k_{\alpha}} E \rightarrow E').  
\]
Since $f$ is $\pi_{*}$-surjective on homotopy groups, so is $p$, and thus $\pi_{*}C$ is concentrated in odd degree. Since $A$ is perfect even, 
\[
\pi_{*} \map_{\Mod_{R}}(A, \bigoplus \Sigma^{k_{\alpha}} E) \rightarrow \pi_{*} \map_{\Mod_{R}}(A, E')
\]
is surjective and thus the structure map $A \rightarrow E'$ lifts to a map into $\bigoplus \Sigma^{k_{\alpha}} E$, so that the composite $A \rightarrow E' \rightarrow C$ vanishes. Thus, we have a factorization 
\[
E' \rightarrow \mathrm{cofib}(A \rightarrow E') \rightarrow C.
\]
The second map vanishes by \cref{remark:extension_condition_in_envelope_as_relative_cohomology} since $\pi_{*}C$ is odd, proving the result. 
\end{proof}

\begin{corollary}[Weak uniqueness of the $\pi_{*}$-even envelope of the unit] 
\label{corollary:weak_uniqueness_of_the_pi_star_envelopes}
If $R \rightarrow E$ and $R \rightarrow E'$ are $\pi_{*}$-even envelopes, then $E$ and $E'$ generate the same subcategory of $\Mod_{R}$ under direct sums, retracts and even (de)suspensions. This subcategory contains $\pi_{*}$-even envelopes of any perfect even. 
\end{corollary}

As a corollary of the construction of $\pi_{*}$-even envelopes, we have the following characterization of even flat modules in terms of tensor products: 

\begin{proposition}[Lazard's Theorem]
\label{proposition:tensor_characterization_of_even_flat_modules}
The following are equivalent for an $R$-module $M$: 
\begin{enumerate}
    \item $M$ is even flat in the sense of \cref{definition:even_flat_module} or 
    \item for any right $R$-module $E$ such that $\pi_{*}E$ is even, $\pi_{*}(E \otimes_{R} M)$ is even. 
\end{enumerate}
Moreover, if $R$ is connective, then it is enough to verify the second condition when $E$ is also connective. 
\end{proposition}

\begin{proof}
$(1 \Rightarrow 2)$ Since the subcategory of those $R$-modules $M$ such that $\pi_{*}(E \otimes_{R} M)$ is even is closed under retracts, filtered colimits and contains $\Sigma^{2k} R$, it necessarily contains all even flat modules. 

$(2 \Rightarrow 1)$ By \cref{proposition:characterization_of_even_flatness_in_terms_of_maps_factoring_through_perfect_even}, it is enough to show that every map $P \rightarrow M$ from a perfect $R$-module factors through a perfect even.  By dualizing, we can identify the given map of modules with a map of spectra $S^{0} \rightarrow P^{\ast} \otimes_{R} M$. 

Let $P^{\ast} \rightarrow E$ be a map such that its suspension is a $\pi_{*}$-even envelope in right $R$-modules, so that the fibre $F \colonequals \mathrm{fib}(P^{\ast} \rightarrow E)$ is even flat and $\pi_{*}E$ is concentrated in \emph{odd} degrees. By assumption, $\pi_{*}(E \otimes_{R} M)$ is also concentrated in odd degrees, so that the composite $S^{0} \rightarrow E \otimes_{R} M$ vanishes. It follows that we have a lift $S^{0} \rightarrow F \otimes_{R} M$. 

Since $F$ is even flat, it can be written $F \simeq \varinjlim A_{\alpha}$ as a filtered colimit of perfect even right $R$-modules, so that $F \otimes_{R} M \simeq \varinjlim A_{\alpha} \otimes_{R} M$. As $S^{0}$ is compact, the given map $S^{0} \rightarrow F \otimes_{R} M$ of spectra factors through $S^{0} \rightarrow A \otimes_{R} M$, where $A$ is a perfect even. By dualizing the resulting diagram of spectra 
\[
\begin{tikzcd}
	& {A \otimes_{R} M} \\
	{S^{0}} && {P^{\ast} \otimes_{R} M}
	\arrow[from=2-1, to=1-2]
	\arrow[from=2-1, to=2-3]
	\arrow[from=1-2, to=2-3]
\end{tikzcd}
\]
we obtain a diagram of $R$-modules 
\[
\begin{tikzcd}
	& {A{^\ast}} \\
	P && {M}
	\arrow[from=2-1, to=1-2]
	\arrow[from=2-1, to=2-3]
	\arrow[from=1-2, to=2-3]
\end{tikzcd},
\]
which since $A$ is perfect even provides the needed factorization. 

If $R$ is connective, then the homotopy groups of $P$ as above are bounded from below, so that by a sufficiently large even suspension we can assume that $P$ is connective. In this case, $E$ can also chosen to be connective by \cref{proposition:pi_star_even_envelopes_exist}. 
\end{proof}

\begin{remark}
Observe that any cofibre sequence $E \rightarrow F \rightarrow G$ of $\pi_{*}$-even right $R$-modules is necessarily short exact on homotopy; that is, 
\[
0 \rightarrow \pi_{*}E \rightarrow \pi_{*}F \rightarrow \pi_{*}G \rightarrow 0
\]
is exact. By \cref{proposition:tensor_characterization_of_even_flat_modules}, the same is true for 
\[
0 \rightarrow \pi_{*}(E \otimes_{R} M) \rightarrow \pi_{*}(F \otimes_{R} M) \rightarrow \pi_{*}(G \otimes_{R} M) \rightarrow 0,
\]
when $M$ is even flat since again everything is even. It is in this sense that an even flat module behaves in a ``flat'' manner, but only from the perspective of modules with even homotopy groups. 
\end{remark}

\subsection{Homologically even modules} 

Recall that we say that an $R$-module $M$ is homologically even if $\evensheaf_{M}(\nicefrac{1}{2}+k) = 0$ for all $k \in \mathbb{Z}$. The goal of this section is to characterize these modules in a variety of different ways. 

In the context of \cref{theorem:characterization_of_homologically_even_modules}, we say that a right $R$-module $E$ is said to be even flat if it is even flat when considered as a left module over the opposite algebra $R^{op}$. 

\begin{theorem}
\label{theorem:characterization_of_homologically_even_modules}
For an $R$-module $M$, the following conditions are equivalent: 
\begin{enumerate}
    \item $M$ is homologically even, 
    \item every map $A \rightarrow \Sigma M$ from a perfect even $R$-module into the suspension of $M$ factors through $A \rightarrow \Sigma B$, where $B$ is perfect even, 
    \item for any even flat, $\pi_{*}$-even right $R$-module $E$, $\pi_{*}(E \otimes_{R} M)$ is even, 
    \item if $A \rightarrow E$ is $\pi_{*}$-even envelope of a perfect even right $R$-module, then $\pi_{*}(E \otimes_{R} M)$ is even. 
    \item if $R \rightarrow E$ is $\pi_{*}$-even envelope of the unit in right $R$-modules, then $\pi_{*}(E \otimes_{R} M)$ is even. 
\end{enumerate}
\end{theorem}

\begin{proof}
$(1 \Leftrightarrow 2)$ The homotopy class of $A \rightarrow \Sigma M$ determines an element of $\pi_{0} \map (A, \Sigma M)$. Since $\evensheaf_{M}(\nicefrac{1}{2})$  is defined as the sheafication of the presheaf
\[
\pi_{0} \map(-, \Sigma M) \colon \Perf(R)_{ev}^{op} \rightarrow \Ab,
\]
it vanishes if for every such element there exists an even epimorphism $f \colon B \rightarrow A$ such that the composite
\[
B \rightarrow A \rightarrow \Sigma M
\]
is zero. This happens precisely when the second map factors through $\Sigma(\mathrm{fib}(f))$, which is a suspension of a perfect even as needed.  

$(2 \Rightarrow 3)$ Suppose we have a homotopy class of maps $S^{k} \rightarrow E \otimes_{R} M$ of spectra with $k$ odd. Since $E$ is even flat, it is a filtered colimit of perfect evens, so that the given map of spectra factors as $S^{k} \rightarrow A \otimes_{R} M$, where $A$ is a perfect even right $R$-module. This is determined by a map of modules  $\Sigma^{k} A^{\ast} \rightarrow M$ from the dual. As $A^{\ast}$ is perfect even and $k$ is odd, by assumption this map factors as
\[
\Sigma^{k} A^{\ast} \rightarrow B \rightarrow M,
\]
where $B$ is perfect even. By dualizing again, we obtain a commutative diagram 
\[
\begin{tikzcd}
	{S^{k}} & {A \otimes_{R} B} & {A \otimes_{R} M} \\
	& {E \otimes_{R}B} & {E \otimes_{R}M}
	\arrow[from=1-2, to=2-2]
	\arrow[from=2-2, to=2-3]
	\arrow[from=1-2, to=1-3]
	\arrow[from=1-3, to=2-3]
	\arrow[from=1-1, to=1-2]
\end{tikzcd}
\]
of spectra. Since $B$ is perfect even, the spectrum $\pi_{*}(E \otimes_{R} B)$ is concentrated in even degrees by \cref{proposition:tensor_characterization_of_even_flat_modules}. We deduce that the same is true for $\pi_{*}(E \otimes_{R} M)$. 

$(3 \Rightarrow 4)$ and $(4 \Rightarrow 5)$ follow from \cref{remark:pistar_even_envelopes_inherit_evenness}. $(5 \Rightarrow 4)$ is a consequence of \cref{proposition:pi_star_unit_envelope_of_unit_contains_envelopes_of_perfect_evens}.

$(4 \Rightarrow 2)$. Let $A \rightarrow \Sigma M$ be a map from a perfect even spectrum, which we can identify with a map of spectra $S^{-1} \rightarrow A^{\ast} \otimes_{R} M$. Let $A^{\ast} \rightarrow E$ be a $\pi_{*}$-even envelope in right $R$-modules which exists by \cref{proposition:pi_star_even_envelopes_exist} and which we can complete to a cofibre sequence. 
\[
A^{\ast} \rightarrow E \rightarrow C
\]
By assumption, $\pi_{*}(E \otimes_{R} M)$ is even, so that the map $S^{-1} \rightarrow A^{\ast} \otimes_{R} M$ lifts to $S^{-1} \rightarrow (\Sigma^{-1} C) \otimes_{R} M$. As $C$ is even flat, we can write it $C \simeq \varinjlim B_{\alpha}$ as a filtered colimit of perfect evens, so that $(\Sigma^{-1} C) \otimes_{R} M \simeq \varinjlim (\Sigma^{-1} B_{\alpha}) \otimes_{R} M$ and by compactness of the sphere we obtain factorization
\[
S^{-1} \rightarrow (\Sigma^{-1} B) \otimes_{R} M
\]
where $B$ is perfect even. Dualizing the resulting diagram 
\[
\begin{tikzcd}
	& {(\Sigma^{-1} B)\otimes_{R} M} \\
	{S^{-1}} && {A^{\ast} \otimes_{R} M}
	\arrow[from=2-1, to=1-2]
	\arrow[from=2-1, to=2-3]
	\arrow[from=1-2, to=2-3]
\end{tikzcd}
\]
gives a diagram of left $R$-modules 
\[
\begin{tikzcd}
	& {\Sigma B^{\ast}} \\
	{A} && {\Sigma M}
	\arrow[from=2-1, to=1-2]
	\arrow[from=2-1, to=2-3]
	\arrow[from=1-2, to=2-3]
\end{tikzcd}
\]
which ends the argument since $B^{\ast}$ is perfect even. 
\end{proof}

\begin{remark}
\label{remark:even_flat_modules_are_homologically_even}
Note that it follows from characterizations $(2)$ of \cref{theorem:characterization_of_homologically_even_modules} and \cref{proposition:characterization_of_even_flatness_in_terms_of_maps_factoring_through_perfect_even} that any even flat module is homologically even. Alternatively, we can also observe that $\evensheaf_{-}(\nicefrac{-1}{2})$ commutes with filtered colimits and vanishes on perfect evens by \cref{lemma:every_perfect_even_is_homogically_even}. 
\end{remark}

\begin{warning}
Beware that the implication of \cref{remark:even_flat_modules_are_homologically_even} cannot be reversed in general: that is, not every homologically even $R$-module is even flat. For a specific example, by \cref{proposition:even_flatnes_and_homological_flatnes_for_pistar_even_rings} below, $\mathbb{Z}/p$ is homologically even as a $\mathbb{Z}$-module in spectra, but it is not even flat. 
\end{warning}

\begin{remark}
Note that parts $(2)$ of \cref{proposition:tensor_characterization_of_even_flat_modules} and part $(3)$ of \cref{theorem:characterization_of_homologically_even_modules} can be combined in the following elegant manner: the homotopy groups $\pi_{*}(E \otimes_{R} M)$ of a tensor product of a $\pi_{*}$-even right $R$-module $E$ and a homologically even left $R$-module $M$ are concentrated in even degrees if at least one of the two is even flat. This mimics the classical observation that for $\Tor^{1}_{\mathbb{Z}}(A, B)$ to vanish it is enough for one of the $A$ or $B$ to be a flat as an abelian group.
\end{remark}

\subsection{Evenness over connective and homotopy even rings} 

In this section, we describe the various notions of evenness over rings which are connective or have homotopy groups concentrated in even degrees. 

\begin{proposition}
\label{proposition:even_flatnes_and_homological_flatnes_for_pistar_even_rings}
Let $R$ be an $\mathbf{E}_{1}$-ring such that $\pi_{*}R$ is concentrated in even degrees. Then an $R$-module $M$ 
\begin{enumerate}
    \item is even flat if and only if $\pi_{*}M$ is a flat $\pi_{*}R$-module concentrated in even degrees, 
    \item is homologically even if and only if $\pi_{*}M$ is concentrated in even degrees. 
\end{enumerate}

\begin{proof}
In both cases, the forward implication is clear. We begin with backwards implication for $(1)$, so suppose that $\pi_{*}M$ is flat and concentrated in even degrees. By \cref{proposition:tensor_characterization_of_even_flat_modules}, it is enough to show that for every right $R$-module $E$ with homotopy in even degrees, the same is true for $\pi_{*}(E \otimes_{R} M)$. However, by \cite[7.2.1.19]{higher_algebra} we have a conditionally convergent K\"{u}nneth spectral sequence 
\[
\mathrm{Tor}^{*, *}_{\pi_{*}R}(\pi_{*}E, \pi_{*}M) \Rightarrow \pi_{*}(E \otimes_{R} M).
\]
Observe that the left hand side is concentrated in $\mathrm{Tor}$-degree zero as $\pi_{*}M$ is flat. As it is concentrated in even internal degree by assumption, we deduce that the same is true for the right hand side, which is what we wanted to show. 

The backwards implication for $(2)$ is identical, using characterization $(3)$ of \cref{theorem:characterization_of_homologically_even_modules}. 
\end{proof}
\end{proposition}

\begin{theorem}
\label{theorem:even_flat_modules_over_connective_rings_detected_by_base_change_to_pi0}
Let $R$ be a connective $\mathbf{E}_{1}$-ring and let us write $R_{\leq 0} \colonequals \pi_{0} R$, considered as an $\mathbf{E}_{1}$-algebra in spectra. Then the following are equivalent for bounded below $R$-module $M$:
\begin{enumerate}
    \item $M$ is even flat, 
    \item $R_{\leq 0} \otimes_{R} M$ is even flat as a $R_{\leq 0}$-module, 
    \item $\pi_{*}(R_{\leq 0} \otimes_{R} M)$ is a flat $\pi_{0}R$-module, concentrated in even degrees. 
\end{enumerate}
\end{theorem}

\begin{proof}
The three given properties are all invariant under even (de)suspensions. Thus, by taking a suitably large even suspension of $M$, we can assume that it is connective. 

Observe that $(1 \Rightarrow 2)$ is clear, since even flat modules are closed under base-change, and $(2 \Leftrightarrow 3)$ follows from \cref{proposition:even_flatnes_and_homological_flatnes_for_pistar_even_rings}. Thus, we only have to show $(2 \Rightarrow 1)$. 

Suppose that $R_{\leq 0} \otimes_{R} M$ is even flat as a $R_{\leq 0}$-module. By \cref{proposition:tensor_characterization_of_even_flat_modules}, it is enough to verify that for every connective $R$-module $E$ with homotopy in even degrees, $\pi_{*}(E \otimes_{R} M)$ is also concentrated in even degrees. Since $M$ is connective, the truncation map $E \rightarrow \tau_{\leq 2n} E$ induces an isomorphism
\[
\pi_{*}(E \otimes_{R} M) \rightarrow \pi_{*}(\tau_{\leq 2n} E \otimes_{R} M)
\]
in degrees $* \leq 2n$. It follows that we can assume that $M$ is connective and bounded from above. 

In this case, $E$ admits a finite filtration whose subquotients are even suspensions of objects of $N \in \Mod_{R}^{\heartsuit} \simeq \Mod_{R_{\leq 0}}^{\heartsuit}$. In this case, we can write 
\[
N \otimes_{R} M \simeq N \otimes_{R_{\leq 0}} \otimes (R_{\leq 0} \otimes_{R} M)
\]
and as the right hand side has even homotopy groups by assumption, so does the left hand side. Since $E \otimes_{R} M$ is built using extensions from even suspensions of such $N \otimes_{R} M$, we deduce that $\pi_{*}(E \otimes_{R} M)$ is also concentrated in even degrees.
\end{proof}

\begin{warning}
\label{warning:base_change_of_homological_even_along_truncation_need_not_be_homologically_even}
Beware that the criterion of \cref{theorem:even_flat_modules_over_connective_rings_detected_by_base_change_to_pi0} does not apply to homologically even modules; that is, the base-change along $R \rightarrow R_{\leq 0}$ need neither preserve nor reflect the property of being homologically even.

For a specific counterexample, let $k$ be a field and $k[x]$ a free $\mathbf{E}_{1}$-algebra on a variable in degree two, so that $\pi_{*}(k[x])$ is a polynomial ring. Then 
\[
k \simeq \mathrm{cofib}(x \colon \Sigma^{2} k[x] \rightarrow k[x])
\]
is homologically even as a $k[x]$-module as a consequence of \cref{proposition:even_flatnes_and_homological_flatnes_for_pistar_even_rings}, but $k \otimes_{k[x]} k \simeq k \oplus \Sigma^{3} k$ is not homologically even over $k$. 
\end{warning}

\begin{warning}
Beware that \cref{theorem:even_flat_modules_over_connective_rings_detected_by_base_change_to_pi0} is not true if we do not assume that $M$ is bounded below. For a specific example, let $k[x]$ be the free $\mathbf{E}_{1}$-algebra on a variable in degree $2$. Then 
\[
\Sigma k(x) \colonequals \varinjlim \ (\Sigma k[x] \rightarrow \Sigma^{-1} k[x] \rightarrow \Sigma^{-3} k[x] \rightarrow \ldots) 
\]
is not even flat by \cref{proposition:even_flatnes_and_homological_flatnes_for_pistar_even_rings}, even though 
\[
k \otimes_{k[x]} (\Sigma k(x)) = 0
\]
which is even flat over $k$. 
\end{warning}

\section{Homological resolutions}

In this section, we describe how even cohomology of an $R$-module $M$ can be calculated by resolving it through $R$-modules with even homotopy groups. 

\begin{remark}[Adams resolutions]
The various constructions described in this section are standard in homotopy theory; essentially, we study a particular class of Adams resolutions with respect to the homological functor $\evensheaf_{-} \colon \Mod_{R}(\spectra) \rightarrow \sheaves_{\Sigma}(\Perf_{ev}, \Ab)$. For more background on Adams resolutions, see \cite[\S III.15]{adams1995stable} or \cite[\S 2.2]{patchkoria2021adams}. Since our aim is for the paper to be accessible to a possibly large audience, we do not assume any knowledge of Adams resolutions or the associated Adams spectral sequence. 
\end{remark}

Observe that if $M_{1} \rightarrow M_{2} \rightarrow M_{3}$ is a cofibre sequence of homologically even $R$-modules, then the long exact sequence of even sheaves shows that 
\[
0 \rightarrow \evensheaf_{M_{1}}(q) \rightarrow \evensheaf_{M_{2}}(q) \rightarrow \evensheaf_{M_{3}}(q) \rightarrow 0
\]
is short exact for any $q \in \mathbb{Z}$. Applying $\evencoh^{*}(R, -)$ of \cref{definition:even_cohomology_with_coefficients_in_a_module} we obtain a long exact sequence of even cohomology 
\[
\ldots \rightarrow \evencoh^{p-1, q}(R, M_{3}) \rightarrow \evencoh^{p, q}(R, M_{1}) \rightarrow \evencoh^{p, q}(R, M_{2}) \rightarrow \evencoh^{p, q}(R, M_{3}) \rightarrow \evencoh^{p+1, q}(R, M_{1}) \rightarrow \ldots.
\]
Since even cohomology of $\pi_{*}$-even modules is particularly simple by \cref{corollary:even_cohomology_of_even_modules}, it follows that to compute cohomology with coefficients in an $R$-module $M$, one can map it into a $\pi_{*}$-even module, which can be done in a particularly nice way as studied in \S\ref{subsection:homotopy_even_envelopes}. One elementary way to exploit this idea was described in the introduction, and we do so again here: 

\begin{construction}
\label{construction:cochain_complex_computing_even_cohomology}
Let $M$ be homologically even. Choose a $\pi_{*}$-envelope $M \rightarrow E_{0}$, which can be done by \cref{proposition:pi_star_even_envelopes_exist}, and set $C_{0} \colonequals \mathrm{cofib}(M \rightarrow E_{0})$. Inductively letting $C_{i} \rightarrow E_{i+1}$ be a $\pi_{*}$-even envelope and setting $C_{i+1} \colonequals \mathrm{cofib}(C_{i} \rightarrow E_{i})$ leads to a diagram of $R$-modules  
\begin{equation}
\label{equation:cochain_complex_of_even_modules_computing_even_cohomology}
\begin{tikzcd}
	& {E_{0}} & {E_{1}} & {E_{2}} & \ldots \\
	M & {C_{0}} & {C_{1}} & {C_{2}}
	\arrow[from=2-1, to=1-2]
	\arrow[from=1-2, to=2-2]
	\arrow[from=2-2, to=1-3]
	\arrow[from=1-2, to=1-3]
	\arrow[from=1-3, to=2-3]
	\arrow[from=2-3, to=1-4]
	\arrow[from=1-3, to=1-4]
	\arrow[from=1-4, to=2-4]
	\arrow[from=2-4, to=1-5]
	\arrow[from=1-4, to=1-5]
\end{tikzcd}
\end{equation}
where in the upper row the composite of any two maps is null. 
\end{construction}

\begin{proposition}
\label{proposition:even_cohomology_computed_using_cochain_complex_of_r_modules}
If $M$ is homologically even and (\ref{equation:cochain_complex_of_even_modules_computing_even_cohomology}) is a diagram as in \cref{construction:cochain_complex_computing_even_cohomology}, then there is a canonical isomorphism
\[
\evencoh^{p, q}(R, M) \simeq \mathrm{H}^{p}(\pi_{2q}E_{\bullet}),
\]
between the $(p,q)$-th even cohomology of $M$ and the cohomology of the cochain complex
\[
\pi_{2q} E_{0} \rightarrow \pi_{2q} E_{1} \rightarrow \pi_{2q} E_{2} \rightarrow \ldots 
\]
computed in the $p$-th spot. 
\end{proposition}

\begin{proof}
By construction, all of these $R$-modules are homologically even, so that for each $q \in \mathbb{Z}$ and $i \geq -1$ we have short exact sequences of even sheaves 
\[
0 \rightarrow \evensheaf_{C_{i}} \rightarrow \evensheaf_{E_{i+1}} \rightarrow \evensheaf_{C_{i+1}} \rightarrow 0 
\]
and thus long exact sequences in sheaf cohomology 
\[
\ldots \rightarrow \evencoh^{p-1, q}(R, C_{i+1}) \rightarrow \evencoh^{p, q}(R, C_{i}) \rightarrow \evencoh^{p, q}(R, E_{i+1}) \rightarrow \evencoh^{p, q}(R, C_{i+1}) \rightarrow \ldots.
\]
where $C_{-1} \colonequals M$. Splicing these together leads to an exact couple and hence a spectral sequence 
\[
\evencoh^{*, q}(E_{\bullet}) \Rightarrow \evencoh^{*, q}(M).
\]
By \cref{corollary:even_cohomology_of_even_modules}, the first page is concentrated in cohomological degree zero, where we have $\evencoh^{0, q}(E_{\bullet}) \simeq \pi_{2q} E_{\bullet}$. It follows that the spectral sequence collapses on the second page, inducing the needed isomorphism. 
\end{proof}

While \cref{proposition:even_cohomology_computed_using_cochain_complex_of_r_modules} gives an effective way of calculating even cohomology, it is often convenient to have a more refined version of \cref{construction:cochain_complex_computing_even_cohomology}, based on cosimplicial objects. This has the advantage that it gives a limit description of the even filtration itself, rather than just the calculation of the even cohomology. 

Recall that associated to an (augmented) cosimplicial object in an abelian category we have the unnormalized\footnote{Since the normalized and unnormalized chain complexes are quasi-isomorphic by \cite[Proposition 1.2.3.17]{higher_algebra}, in the context of \cref{definition:homlogical_resolution} one can equivalently work with the normalized chain complex. In other words, the condition is that the cochain corresponding to $\evensheaf_{X_{\bullet}}$ through the Dold-Kan correspondence of \cite[Theorem 1.2.3.7]{higher_algebra} is exact.} Moore complex whose connecting maps are given by the alternating sums of coboundary maps, see \cite[Proposition 1.2.3.17]{higher_algebra}. 

\begin{definition}
\label{definition:homlogical_resolution}
We say that an augmented semicosimplicial object $X \colon \Delta_{s, +} \rightarrow \Mod_{R}$ is a \emph{homological resolution} if for every $q \in \nicefrac{1}{2}\mathbb{Z}$, the induced Moore cochain complex
\[
0 \rightarrow \evensheaf_{X_{-1}}(q) \rightarrow \evensheaf_{X_{0}}(q) \rightarrow \evensheaf_{X_{1}}(q) \rightarrow \ldots
\]
is exact as a complex in the abelian category $\sheaves_{\Sigma}(\Perf_{ev}, \Ab)$. 
\end{definition}

The usefulness of homological resolutions comes down to the following simple observation: 

\begin{proposition}
\label{proposition:homological_resolutions_are_limits_on_even_filtrations_up_to_completion}
If $X \colon \Delta_{s, +} \rightarrow \Mod_{R}$ is a homological resolution, then the canonical comparison map
\[
\gr^{*}_{ev}(X_{-1}) \rightarrow \varprojlim \gr^{*}_{ev}(X_{m}), 
\]
is an equivalence of graded spectra, where the limit on the right is taken over $[m] \in \Delta_{s}$. In particular, the map 
\[
\fil_{ev}^{*}(X_{-1}) \rightarrow \varprojlim (\fil_{ev}^{*}(X_{m}))
\]
of filtered spectra is an equivalence after completion. 
\end{proposition}

\begin{proof}
Since the sections functor $\Gamma_{\Perf_{ev}}(R, -) \colon \sheaves_{\Sigma}(\Perf(R)_{ev}, \spectra)) \rightarrow \spectra$ of \cref{notation:section_of_a_sheaf} is exact, for any $R$-module we have 
\[
\gr^{q}_{ev}(M) \simeq \Gamma_{\Perf_{ev}}(R, \mathrm{cofib}(\tau_{\geq 2q+2} Y_{R}(M) \rightarrow \tau_{\geq 2q} Y_{R}(M))). 
\]
The sheaf whose sections we are considering is concentrated in degrees $2q \leq \ast \leq 2q-1$, so that its $2q$-connective cover can be completed to a cofibre sequences 
\[
\Sigma^{2q+1} \evensheaf_{M}(q+\nicefrac{1}{2}) \rightarrow \mathrm{cofib}(\tau_{\geq 2q+2} Y_{R}(M)) \rightarrow \tau_{\geq 2q} Y_{R}(M)) \rightarrow \Sigma^{2q} \evensheaf_{M}(q),
\]
where we identify the even sheaves with objects of the heart. Taking section, we obtain a cofibre sequence of spectra 
\[
\Sigma^{2q+1} \Gamma_{\Perf_{ev}}(R, \evensheaf_{M}(q+\nicefrac{1}{2})) \rightarrow \gr^{q}_{ev}(M) \rightarrow \Sigma^{2q} \Gamma_{\Perf_{ev}}(R, \evensheaf_{M}(q)).
\]
Since the section functor is corepresentable and hence continuous, we deduce that to prove the needed statement about the associated graded it is enough to verify that for each $q \in \nicefrac{1}{2} \mathbb{Z}$ the canonical map 
\begin{equation}
\label{equation:canonical_map_from_even_sheaf_of_m_to_limit_of_even_sheaves_of_homological_resolution}
\evensheaf_{X_{-1}}(q) \rightarrow \varprojlim \evensheaf_{X_{\bullet}}(q) 
\end{equation}
is an equivalence in $\sheaves_{\Sigma}(\Perf_{ev}, \spectra)$. The sheaf homotopy groups of the target can be calculated using the totalization spectral sequence of \cite[{\S 1.2.2}]{higher_algebra} which degenerates in this case and shows that 
\[
\pi_{-t} (\varprojlim \evensheaf_{X_{\bullet}}(q)) \simeq \Hrm^{t}(\evensheaf_{X_{\bullet}}(q)), 
\]
the cohomology of the (unaugmented) Moore cochain complex. By assumption that $X$ is a homological resolution, we see that (\ref{equation:canonical_map_from_even_sheaf_of_m_to_limit_of_even_sheaves_of_homological_resolution}) is an isomorphism on sheaf homotopy groups. Since both sides are coconnective, they are hypercomplete and the map is an equivalence. 
\end{proof}

As in the context of \cref{proposition:even_cohomology_computed_using_cochain_complex_of_r_modules}, homological resolutions are most useful when they consist of modules with only even homotopy groups.

\begin{theorem}
\label{theorem:existence_of_pistar_even_adams_resolutions}
Any homologically even $R$-module $M$ can be completed to a homological resolution $X \colon \Delta_{s, +} \rightarrow \Mod_{R}$ with $X_{-1} = M$ and $\pi_{*} X_{m}$ even for $m \geq 0$. 
\end{theorem}

\begin{proof}
We recall that augmented semicosimplicial objects can be constructed inductively, by choosing appropriate maps out of matching objects, see \cite[A.2.9.15, A2.9.16]{lurie_higher_topos_theory}.

We define $X_{-1} \colonequals M$, so that the $0$-th matching object is given by $M_{0}X \simeq X_{-1} \simeq M$. By \cref{proposition:pi_star_even_envelopes_exist}, there exists a $\pi_{*}$-even envelope 
\[
M_{0}X \rightarrow E_{0}
\]
and we set $X_{0} \colonequals E_{0}$. This extends $X$ to a diagram indexed by $\Delta_{s, +, \leq 0}$, so that the matching object $M_{1}X$ is well-defined. We then let $X_{1}$ be a $\pi_{*}$-even envelope of $M_{1}X$, extending our diagram to $\Delta_{s, +, \leq 1}$. Proceeding inductively in this manner, we obtain an augmented semicosimplicial object which by construction has $\pi_{*} X_{m}$ even for $m \geq 0$. We will show that it is a homological resolution. 

We first argue by induction that $M_{m}X$ is homologically even for all $m \geq 0$. The base-case is clear, since $M_{0} X \simeq M$, so suppose that we know that each of $M_{k}X$ is homologically even for $k < m$. As by construction, the map $M_{k}X \rightarrow X_{k}$ has homologically even cofibre, this will also show that for in this range $X_{k}$ is homologically even and $\evensheaf_{M_{k}X} \rightarrow \evensheaf_{X_{k}}$ is a monomorphism.

The long exact sequence of homology shows that homologically even modules are closed under pushouts along $\evensheaf_{-}$-monomorphisms, and that on this subcategory the association $M \mapsto \evensheaf_{M}$ commutes with such pushouts. It follows from \cite[A.2: Proposition 5 and Remark 4]{pstrkagowski2018synthetic} and the inductive step that $M_{m}X$ is homologically even and $\evensheaf_{M_{m}X} \simeq M_{m}(\evensheaf_{X})$. The first conclusion ends the induction. 

Since the cofibre of $M_{m}X \rightarrow X_{m}$ is homologically even, the second conclusion from the previous paragraph shows that
\[
M_{m}(\evensheaf_{X}) \simeq \evensheaf_{M_{m}X} \rightarrow \evensheaf_{X_{m}}
\]
is a monomorphism for all $m \geq 0$. If follows that if $I$ is an injective cogenerator of $\sheaves_{\Sigma}(\Perf_{ev}; \Ab)$, then $\Hom(\evensheaf_{X}, I)$ defines a hypercover in $\dcat_{\geq 0}(\mathbb{Z})$ and hence its Moore complex is exact. We deduce that the Moore complex of $\evensheaf_{X}$ itself is exact. Since the integral weight twists can be obtained from $\evensheaf_{-}$ using \cref{remark:local_grading_of_even_sheaves_compatible_with_weight_notation} and the half-weight twists vanish, we deduce the exactness of the Moore complex of $\evensheaf_{X}(q)$ for all $q$. 
\end{proof}

\begin{remark}
\label{remark:even_filtration_of_a_homological_even_as_decalage_of_its_resolution}
Suppose that $M$ is homologically even, so that by \cref{theorem:existence_of_pistar_even_adams_resolutions} we can complete it to a homological resolution $X \colon \Delta_{s, +} \rightarrow \Mod_{R}$ such that $\pi_{*}X_{m}$ is even for all $m \geq 0$. Then by a combination of \cref{proposition:homological_resolutions_are_limits_on_even_filtrations_up_to_completion} and \cref{lemma:pistar_even_means_homologically_even_and_even_filtration_is_whitehead} we see that the canonical map of filtered spectra
\[
\fil^{*}_{ev} M \rightarrow \varprojlim \tau_{\geq 2 \ast} X_{m}
\]
is an equivalence after completion. Here, the right hand side can be identified with the spectral analogue of Deligne's d\'{e}calage construction of \cite{deligne1971theorie}. 
\end{remark}

\section{Base-change and descent}

Suppose that $R$ is an $\mathbf{E}_{2}$-ring and $S$ is an $\mathbf{E}_{1}-R$-algebra, so that we have an associated cobar diagram of $\mathbf{E}_{1}$-rings 
\[
\begin{tikzcd}
	{R} & {S} & {S \otimes_{R} S} & \ldots
	\arrow[yshift=4pt, from=1-3, to=1-4]
 	\arrow[yshift=0pt, from=1-3, to=1-4]
 	\arrow[yshift=-4pt, from=1-3, to=1-4]
	\arrow[yshift=3pt, from=1-2, to=1-3]	          \arrow[yshift=-3pt, from=1-2, to=1-3]
	\arrow[from=1-1, to=1-2]
\end{tikzcd},
\]
where each coboundary map is induced by the unit $R \rightarrow S$, see \cite[Construction 2.7]{mathew2017nilpotence}. In this context, it is natural to ask if the even filtration associated to $R$ can be recovered from the even filtration of $\mathbf{E}_{1}$-rings $S^{\otimes_{R} n}$, perhaps up to some form of completion. 

\begin{warning}
We will think of the cobar diagram as an augmented semicosimplicial diagram $\Delta_{s, +} \rightarrow \Alg_{\mathbf{E}_{1}}(\spectra)$. While it can be naturally extended to a cosimplicial diagram, the coboundary maps do not respect multiplication and so the resulting extension is only a diagram of spectra.
\end{warning}
As the notion of homological resolution of \cref{definition:homlogical_resolution} gives us some natural conditions under which a semicosimplicial diagram of $R$-modules induces a limit on even filtrations, the main problem to tackle is how the even filtration of $S^{\otimes_{R} n}$ considered as a module over itself differs from the one where we consider it as a module over $R$. 

As explained in the next section, when studying descent properties of the even filtration, it is important to consider properties of $R$-algebras when considered as both left and right modules. To do so without introducing new notions, we will use the following convention: 

\begin{notation}
If $P$ is a property of left modules over $\mathbf{E}_{1}$-rings (such as being perfect even, homologically even or even flat), we will say that a right $R$-module $M$ has property $P$ if it has this property when considered as a left module over $R^{op}$, the opposite algebra. 
\end{notation}

\subsection{Evenness and extension/restriction of scalars} 
\label{subsection:evenness_and_extension_restriction_of_scalars}

Associated to a morphism $f \colon R \rightarrow S$ of $\mathbf{E}_{1}$-rings we have an adjunction
\[
S \otimes_{R} - \dashv R_{S. R} \colon \Mod_{R} \rightleftarrows \Mod_{S},
\]
where $R_{S. R}$ denotes the forgetful functor, and it is natural to ask how these two functors relate to the various notions introduced in the present work. 

\begin{lemma}
\label{lemma:extension_of_scalars_preserves_even_flat_modules_and_perfect_evens}
The functor $S \otimes_{R} - \colon \Mod_{R} \rightarrow \Mod_{S}$ preserves perfect even and even flat modules. 
\end{lemma}

\begin{proof}
This is clear, since $S \otimes_{R} R \simeq S$ and both classes are defined as closure of the unit under certain kinds of colimits and extensions.
\end{proof}

\begin{remark}
Note that we had seen in \cref{warning:base_change_of_homological_even_along_truncation_need_not_be_homologically_even} that extension of scalars does not in general preserve homologically even modules. 
\end{remark}

To get further properties, we need to make some assumptions on the map. Compatibility with the even filtration can be thought of as a form of exactness, so as motivation, let us analyze the classical situation, when $f \colon S \rightarrow R$ is a map of classical rings. In this case we have an induced extension/restriction of scalars adjunction between the categories of modules in abelian groups 
\[
\mathrm{Tor}^{0}_{R}(S, -) \dashv R_{S, R} \colon \Mod_{R}(\Ab) \rightleftarrows \Mod_{S}(\Ab),
\]
where we write $\mathrm{Tor}^{0}_{R}(S, -)$ to emphasize that here we mean the classical rather than derived tensor product. In this case 
\begin{enumerate}
    \item the restriction of scalars $R_{S, R}$ is always exact, as (co)limits in either $R$ or $S$-modules can be both calculated in abelian groups, 
    \item the extension of scalars $\mathrm{Tor}^{0}_{R}(S, -)$ is exact precisely when $S$ is flat as a right $R$-module.
\end{enumerate}

In the context of $\mathbf{E}_{1}$-rings, the behaviour of extension of scalars is somewhat similar to the one above. However, the situation with the forgetful functor is more subtle, as the even filtration varies with the ring. 

In particular, whether something is ``exact with respect to the even filtration'' depends on more than the underlying spectra. Instead, the behaviour of the forgetful functor depends on the structure of $S$ as a \emph{left} $R$-module. Thus, to get the best of both worlds we are forced to think of $S$ as \emph{both} a left and right $R$-module. 

\begin{definition}
\label{definition:even_flat_and_homologically_even_map_of_e1_rings}
We say that a map $f \colon R \rightarrow S$ of $\mathbf{E}_{1}$-rings is
\begin{enumerate}
    \item \emph{left (resp. right) even flat} if $S$ is even flat as a left (resp. right) $R$-module, 
    \item \emph{left (resp. right) homologically even} if $S$ is homologically even as a left (resp. right) $R$-module. 
\end{enumerate}
\end{definition}

\begin{remark}
\label{remark:for_e1_algebras_left_homologically_even_is_right_homologically_even}
If $R$ is $\mathbf{E}_{2}$ and $f \colon R \rightarrow S$ can be promoted to a unit map of an $\mathbf{E}_{1}$-$R$-algebra structure, then $f$ is left even flat if and only if it is right even flat, and similarly for homological evenness. In particular, this happens whenever $f$ can be promoted to a map of $\mathbf{E}_{2}$-rings. 
\end{remark}

\begin{lemma}
\label{lemma:forgetful_functor_preserves_along_even_flat_or_homologically_flat_map_of_rings_preserves_these_modules}
Let $f \colon R \rightarrow S$ be a map of $\mathbf{E}_{1}$-rings. Then the forgetful functor $\Mod_{S} \rightarrow \Mod_{R}$ 
\begin{enumerate}
    \item preserves even flat modules if $f$ is left even flat and 
    \item preserves homologically even modules if $f$ is left homologically even. 
\end{enumerate}
\end{lemma}

\begin{proof}
The first part is clear, since the forgetful functor sends the unit $S$ to an even flat module by assumption, and even flat modules are defined as a closure of the unit under various colimits which commute with the forgetful functor. 

For the second, by \cref{theorem:characterization_of_homologically_even_modules} it is enough to show that if $N$ is a homologically even $S$-module and $E$ is a $\pi_{*}$-even, even flat right $R$-module, then $\pi_{*}(E \otimes_{R} N)$ is concentrated in even degrees. We can rewrite this tensor product as 
\[
E \otimes_{R} N \simeq (E \otimes_{R} S) \otimes_{S} N.
\]
Note that $E \otimes_{R} S$ is even flat as a right $S$-module by \cref{lemma:extension_of_scalars_preserves_even_flat_modules_and_perfect_evens} and $\pi_{*}(E \otimes_{R} S)$ is concentrated in even degrees by \cref{proposition:tensor_characterization_of_even_flat_modules} and the assumption that $S$ is homologically even as a left $R$-module. Thus, $\pi_{*}((E \otimes_{R} S) \otimes_{S} N)$ is concentrated in even degrees, as needed. 
\end{proof}

\begin{lemma}
\label{lemma:even_flat_maps_preserve_homologically_even_modules}
Let $f \colon R \rightarrow S$ be a right even flat map of $\mathbf{E}_{1}$-rings. Then $S \otimes_{R} - \colon \Mod_{R} \rightarrow \Mod_{S}$ preserves homologically even modules. 
\end{lemma}

\begin{proof}
Let $M$ be a homologically even $R$-module. By \cref{theorem:characterization_of_homologically_even_modules}, to show that $S \otimes_{R} M$ is homologically even, it is enough to verify that if $F$ is a $\pi_{*}$-even, even flat right $S$-module, then $\pi_{*}(F \otimes_{S} S \otimes_{R} M)$ is concentrated in even degrees. We have
\[
F \otimes_{S} S \otimes_{R} M \simeq F \otimes_{R} M
\]
and the right hand side has homotopy groups concentrated in even degrees since $F$ is even flat as a right $R$-module by (the right module variant of) \cref{lemma:forgetful_functor_preserves_along_even_flat_or_homologically_flat_map_of_rings_preserves_these_modules}. 
\end{proof}

We will also need the following variant for bimodules: 

\begin{lemma}
\label{lemma:bimodules_preserving_homological_evenness}
Let $B$ be an $R$-bimodule which is homologically even as a left $R$-module and even flat as a right $R$-module. Then the functor
\[
B \otimes_{R} - \colon \Mod_{R} \rightarrow \Mod_{R}
\]
preserves homologically even modules. 
\end{lemma}

\begin{proof}
This is a combination of the arguments of \cref{lemma:forgetful_functor_preserves_along_even_flat_or_homologically_flat_map_of_rings_preserves_these_modules} and \cref{lemma:even_flat_maps_preserve_homologically_even_modules}.
\end{proof}

\begin{warning}
Beware that if $f \colon R \rightarrow S$ is merely right homologically even rather than even flat, then $S \otimes_{R} - \colon \Mod_{R} \rightarrow \Mod_{S}$ need not preserve homologically even modules. For a specific counterexample, consider the unit map $S^{0} \rightarrow \mathbb{Z}/p$. 

As $\MU$ admits an even cell structure and has even homotopy groups, the unit $S^{0} \rightarrow \MU$ is easily seen to be a $\pi_{*}$-even envelope in spectra and from \cref{theorem:characterization_of_homologically_even_modules} we deduce that a spectrum $M$ is homologically even as an $S^{0}$-module if and only if $\MU_{*}(M)$ is concentrated in even degrees. Since 
\[
\MU_{*}(\mathbb{Z}/p) \simeq \mathbb{Z}/p[b_{1}, b_{2}, \ldots]
\]
is a polynomial algebra concentrated in even degrees, we see that $\mathbb{Z}/p$ is homologically even as a spectrum. However, the base-change $\mathbb{Z}/p \otimes_{S^{0}} \mathbb{Z}/p$ is not homologically even as a $\mathbb{Z}/p$-module as a consequence of \cref{proposition:even_flatnes_and_homological_flatnes_for_pistar_even_rings}, since the dual Steenrod algebra 
\[
\mathcal{A}_{*} \simeq \pi_{*}(\mathbb{Z}/p \otimes_{S^{0}} \mathbb{Z}/p)
\]
is not concentrated in even degrees. 
\end{warning}

We record that there is a canonical map comparing even filtrations over different rings: 

\begin{construction}
\label{construction:comparison_map_between_even_filtrations}
Let $f \colon R \rightarrow S$ be a map of $\mathbf{E}_{1}$-rings and let $N$ be an $S$-module. Then the identity map $\mathrm{id}_{N} \colon N \rightarrow N$ is $R$-linear so that the pair $(f, \mathrm{id}_{N})$ defines a morphism in the $\infty$-category $\Mod(\spectra)$ of pairs of an $\mathbf{E}_{1}$-algebra and a module. Applying the even filtration functor of \cref{theorem:even_filtration_is_lax_symmetric_monoidal_as_a_functor_on_pairs} we obtain a map 
\[
\fil^{*}_{ev/R}(N) \rightarrow \fil^{*}_{ev/S}(N),
\]
of $\fil^{*}_{ev/R}(R)$-modules, which we will call the canonical comparison map. After passing to homotopy groups of the associated graded object of the even filtration, we obtain a comparison map 
\[
\evencoh^{p, q}(R, N) \rightarrow \evencoh^{p, q}(S, N)
\]
between even cohomology groups. 
\end{construction}

\begin{remark}
\label{remark:detecting_eqv_between_even_filtrations_on_cohomology}
Using the explicit description of the maps of additive sheaves inducing the comparison map between even filtrations given in \cref{remark:explicit_description_of_functoriality_of_synthetic_analogues} and the identification of the colimit of the even filtration of \cref{proposition:connectivity_of_even_filtration_maps_and_exhaustivity_of_the_even_filtration}, one observes that after passing to colimits the map 
\begin{equation}
\label{equation:comparison_map_between_even_filtrations_in_rmk_about_detecting_eqv_using_coh}
\varinjlim \fil_{ev/R}^{*}(N) \rightarrow \varinjlim \fil_{ev/S}^{*}(N)
\end{equation}
of \cref{construction:comparison_map_between_even_filtrations} can be identified with the identity $N \rightarrow N$. It follows that it is an equivalence if and only if it is an equivalence between associated graded pieces; that is, when it induces an isomorphism 
\[
\evencoh^{*, *}(R, N) \simeq \evencoh^{*, *}(S, N)
\]
between even cohomology groups. Indeed, if the latter holds, then the cofibre of the comparison map is a filtered spectrum whose associated graded object vanishes (hence it is a constant filtered spectrum) and whose colimit is zero. 
\end{remark}

\begin{lemma}
\label{lemma:homologically_even_maps_of_rings_have_covering_lifting_property_on_perfect_evens}
Let $f \colon R \rightarrow S$ be a left homologically even map of $\mathbf{E}_{1}$-rings. Then 
\[
S \otimes_{R} - \colon \Perf(R)_{ev} \rightarrow \Perf(S)_{ev}
\]
has the covering lifting property with respect to the Grothendieck topologies of even epimorphisms. 
\end{lemma}

\begin{proof}
We have to show that any pair of $A \in \Perf(R)_{ev}$ and an even epimorphism $q \colon M \rightarrow S \otimes_{R} A$ in $\Perf(S)_{ev}$ can be completed to a commutative diagram 
\[
\begin{tikzcd}
	&& M \\
	{S\otimes_{R}B} & {} & {S \otimes_{R}A}
	\arrow["q"', from=1-3, to=2-3]
	\arrow["{S \otimes_{R} p}", from=2-1, to=2-3]
	\arrow[from=2-1, to=1-3]
\end{tikzcd},
\]
where $p$ is an even epimorphism of perfect even $R$-modules. 

Since the fibre $\mathrm{fib}(q)$ is homologically even as an $R$-module by \cref{lemma:forgetful_functor_preserves_along_even_flat_or_homologically_flat_map_of_rings_preserves_these_modules}, long exact sequence of homology shows that the map $\evensheaf_{M} \rightarrow \evensheaf_{S \otimes_{R} A}$ is an epimorphism of sheaves on $\sheaves_{\Sigma}(\Perf(R)_{ev}, \Ab)$. It follows that that there exists an even epimorphism $p \colon B \rightarrow A$ of perfect even $R$-modules which can be completed to a commutative diagram 
\[
\begin{tikzcd}
	B & M \\
	A & {S  \otimes_{R}A}
	\arrow[from=2-1, to=2-2]
	\arrow["p"', from=1-1, to=2-1]
	\arrow[from=1-1, to=1-2]
	\arrow[from=1-2, to=2-2]
\end{tikzcd}
\]
The induced map $S \otimes_{R} p \colon S \otimes_{R} B \rightarrow S \otimes_{R} A$ factors through $M$, as needed. 
\end{proof}

\begin{theorem}
\label{theorem:even_filtration_commutes_with_restriction_of_scalars_for_homologically_even_maps}
Let $f \colon R \rightarrow S$ be a homologically even map of $\mathbf{E}_{1}$-rings. Then for any $S$-module $N$ the canonical comparison map 
\[
\fil_{ev/R}^{*}(N) \rightarrow \fil_{ev/S}^{*}(N)
\]
of \cref{construction:comparison_map_between_even_filtrations} is an equivalence. In particular, 
\[
\evencoh^{*, *}(R, N) \simeq \evencoh^{*, *}(S, N). 
\]
\end{theorem}

\begin{proof}
Since $S \otimes_{R} \colon \Perf(R)_{ev} \rightarrow \Perf(S)_{ev}$ has the covering lifting property by \cref{lemma:homologically_even_maps_of_rings_have_covering_lifting_property_on_perfect_evens}, the precomposition functor 
\[
f_{\ast} \colon \presheaves(\Perf(S)_{ev}, \spectra) \rightarrow \presheaves(\Perf(R)_{ev}, \spectra)
\]
commutes with sheafication by \cite[{Proposition A.13}]{pstrkagowski2018synthetic}. It follows that the map 
\[
\tau_{\geq 0} \map_{\Mod_{R}}(-, N) \simeq \tau_{\geq 0} \map_{\Mod_{S}}(S \otimes_{R} -, N) \simeq f_{\ast}(\tau_{\geq 0} \map_{\Mod_{S}}(-, N)) \rightarrow f_{\ast} \nu_{S}(N)
\]
identifies the target as a sheafication of the source, inducing an equivalence 
\[
\nu_{R}(N) \simeq f_{\ast} \nu_{S}(N). 
\]
It follows that its adjoint  
\begin{equation}
\label{equation:induced_map_between_nus_in_even_flat_base_change}
f^{\ast} \nu_{R}(N) \rightarrow \nu_{S}(N), 
\end{equation}
(which encodes the lax functoriality of the $\nu_{-}(-)$ construction, as explained in \cref{remark:explicit_description_of_functoriality_of_synthetic_analogues}) can be identified with the counit of the target. Since the induced map on even filtrations can be identified with the map 
\[
\underline{\map}^{\ast}_{R}(\nu_{R}(R), \nu_{R}(N)) \rightarrow \underline{\map}^{\ast}_{R}(f^{\ast} \nu_{R}(R), \nu_{S}(N))
\]
obtained by first applying $f^{\ast}$ and then composing with (\ref{equation:induced_map_between_nus_in_even_flat_base_change}) as explained in \cref{remark:unwrapped_definition_of_the_functoriality_of_the_even_filtration}, it must be an equivalence. 
\end{proof}

\subsection{Faithfully flat descent} 

In \cref{theorem:even_filtration_commutes_with_restriction_of_scalars_for_homologically_even_maps}, we had shown that if $f \colon R \rightarrow S$ is a homologically even map of $\mathbf{E}_{1}$-rings, then for any $S$-module its even filtration over $S$ agrees with the one relative to $R$. This shows that information can be moved ``up'' along a map of rings; that is, what happens over $S$ is already determined by $R$. 

More commonly, we instead want to move information ``down''; that is, to deduce results about $R$-modules from their base-change $S \otimes_{R} -$. As Grothendieck's theory of descent shows, this usually requires some variant of faithful flatness. As we discussed in \S\ref{subsection:evenness_and_extension_restriction_of_scalars}, in the context of even filtration this requires some control over $S$ as both a left and right $R$-module. 

\begin{definition}
\label{definiton:even_faithfully_flat_maps_of_rings}
We say that a map $f \colon R \rightarrow S$ of $\mathbf{E}_{1}$-rings is (left) \emph{faithfully even flat} if both $S$ and $\mathrm{cofib}(f)$ are even flat as right $R$-modules and $\mathrm{cofib}(f)$ is homologically even as a left $R$-module. 
\end{definition}

\begin{remark}
\label{remark:criterion_for_faithful_flatness_of_classical_rings}
The motivation for \cref{definiton:even_faithfully_flat_maps_of_rings} is given by the following classical observation: a  monomorphism $R \hookrightarrow S$ of (discrete) commutative rings is faithfully flat if and only if both $S$ and $\mathrm{coker}(R \rightarrow S)$ are flat as $R$-modules, see \cite[Addendum 3.9]{diracgeometry1}.
\end{remark}

\begin{remark}
Note that there is an obvious ``opposite'' notion of a right faithfully even flat morphism, which would be relevant in the context of working with right $R$-modules. 
\end{remark}

\begin{remark}
\label{remark:for_faithfully_even_flat_map_of_e1_rings_the_target_is_left_homologically_even}
If $f \colon R \rightarrow S$ is a faithfully even flat map of $\mathbf{E}_{1}$-rings, then $S$ is also homologically even as left $R$-module. This follows from the fact $R$ is homologically even, which is \cref{lemma:every_perfect_even_is_homogically_even}, and the fact that homologically even modules are closed under extensions. 
\end{remark}

Note that our definition of faithfully even flat is distinct from the notion of ``evenly faithfully flat'' given by Hahn-Raksit-Wilson in the context of $\mathbf{E}_{\infty}$-rings \cite[2.2.13]{hahn2022motivic}. In this paper, by faithfully even flat we will always mean the notion introduced in \cref{definiton:even_faithfully_flat_maps_of_rings}. The following shows that they agree on connective $\mathbf{E}_{\infty}$-rings: 

\begin{proposition}
\label{proposition:fef_maps_are_hrw_eff}
A map $R \rightarrow S$ of $\mathbf{E}_{\infty}$-rings which is faithfully even flat in the sense of \cref{definiton:even_faithfully_flat_maps_of_rings} is also evenly faithfully flat in the sense of Hahn-Raksit-Wilson; that is, for every map $R \rightarrow E$ into a $\pi_{*}$-even $\mathbf{E}_{\infty}$-ring, the base-change $E \otimes_{R} S$ is also $\pi_{*}$-even and the map $\pi_{*}(E) \rightarrow \pi_{*}(E \otimes_{R} S)$ of classical commutative rings is faithfully flat. If both $R$ and $S$ are connective, the converse holds as well. 
\end{proposition}

\begin{proof}
Since by \cref{proposition:tensor_characterization_of_even_flat_modules} a tensor product of an even flat module and a $\pi_{*}$-even module is $\pi_{*}$-even, if $f \colon R \rightarrow S$ is faithfully even flat then the sequence
\[
0 \rightarrow \pi_{*}(E) \rightarrow \pi_{*}(E \otimes_{R} S) \rightarrow \pi_{*}(E \otimes_{R} \mathrm{cofib}(f)) \rightarrow 0
\]
is short exact and concentrated in even degrees. As the middle and right terms are flat over $\pi_{*}E$ by \cref{proposition:even_flatnes_and_homological_flatnes_for_pistar_even_rings}, we deduce from \cite[Lemma 5.5]{lurie2004tannaka} that the first map is a faithfully flat map of classical rings. It follows that $f$ is evenly faithfully flat in the sense of Hahn-Raksit-Wilson. 

Conversely, suppose that $R$ is connective and that $f$ is evenly faithfully flat in the sense of Hahn-Raksit-Wilson. Let us write $R_{\leq 0} \simeq \pi_{0} R$ for the $0$-truncation. By assumption, $\pi_{*}(R_{\leq 0} \otimes_{R} S)$ is concentrated in even degrees and the first map in the sequence
\[
0 \rightarrow \pi_{*}(R_{\leq 0}) \rightarrow \pi_{*}(R_{\leq 0} \otimes_{R} S) \rightarrow \pi_{*}(R_{\leq 0} \otimes_{R} \mathrm{cofib}(f)) \rightarrow 0 
\]
is faithfully flat, hence injective. It follows that the sequence is short exact, concentrated in even degrees. Moreover, the third term is flat as a $\pi_{*}(R_{\leq 0}) \simeq \pi_{0}R$-module by \cref{remark:criterion_for_faithful_flatness_of_classical_rings}. It follows from \cref{theorem:even_flat_modules_over_connective_rings_detected_by_base_change_to_pi0} that both $S$ and $\mathrm{cofib}(f)$ are left even flat as $R$-modules, and thus also right even flat since this is a map of $\mathbf{E}_{\infty}$-rings. We deduce that $R \rightarrow S$ is faithfully even flat. 
\end{proof}

\begin{warning}
For non-connective $\mathbf{E}_{\infty}$-rings, a map which is evenly faithfully flat in the sense of \cite{hahn2022motivic} need not be faithfully even flat in the sense of \cref{definiton:even_faithfully_flat_maps_of_rings}, see \cref{remark:a_map_which_is_eff_but_not_fef}.
\end{warning}

\begin{lemma}
\label{lemma:any_module_admits_a_mono_on_even_homology_into_a_homological_even}
Let $M$ be an $R$-module. Then there exists a map of modules $M \rightarrow N$ such that $N$ is homologically even and such that $\evensheaf_{M} \rightarrow  \evensheaf_{N}$ is a monomorphism. 
\end{lemma}

\begin{proof}
By taking an appropriately large direct sum of perfect evens, we can find a map $\Sigma P \rightarrow M$ of $R$-modules where $P$ is homologically even and such that
\[
\evensheaf_{P} \simeq \evensheaf_{\Sigma P}(\nicefrac{1}{2}) \rightarrow \evensheaf_{M})(\nicefrac{1}{2})
\]
is a surjection. Since $P$ is homologically even, it follows from the long exact sequence of even sheaves of \cref{remark:long_exact_sequence_of_even_sheaves} that the map from $M \rightarrow \mathrm{cofib}(\Sigma P \rightarrow M)$ has the needed properties. 
\end{proof}

\begin{lemma}
\label{lemma:for_even_faithfully_map_tensoring_gives_ses_of_homology}
Let $f \colon R \rightarrow S$ be a left faithfully even flat map of $\mathbf{E}_{1}$-rings. If $M$ is an $R$-module, then for any weight $q$ the sequence
\[
0 \rightarrow \evensheaf_{M}(q) \rightarrow \evensheaf_{S \otimes_{R} M}(q) \rightarrow \evensheaf_{\cofib(f) \otimes_{R} M}(q) \rightarrow 0
\]
of abelian sheaves on $\Perf(R)_{ev}$ is short exact. 
\end{lemma}

\begin{proof}
Assume first that $M$ is homologically even or a suspension of one. In this case, since $\mathrm{cofib}(f)$ and $S$ are homologically even as left $R$-modules, the latter by \cref{remark:for_faithfully_even_flat_map_of_e1_rings_the_target_is_left_homologically_even}, so are $S \otimes_{R} M$ and $\cofib(f) \otimes_{R} M$ by \cref{lemma:bimodules_preserving_homological_evenness}. It follows that the above sequence is short exact. 

In the case of general $M$, since $\evensheaf_{M}(q) \simeq \evensheaf_{\Sigma^{-2q} M}$, it is enough to do the case when $q = 0$. We want to prove that the boundary map 
\[
\evensheaf_{\cofib(f) \otimes_{R} M} \rightarrow \evensheaf_{M}(-\nicefrac{1}{2})
\]
is zero. Applying \cref{lemma:any_module_admits_a_mono_on_even_homology_into_a_homological_even} to the desuspension of $M$, there exists an $\evensheaf_{-}(-\nicefrac{1}{2})$-monomorphism $M \rightarrow \Sigma N$ into a suspension of a homologically even module. We then have a commutative diagram
\[
\begin{tikzcd}
	{\evensheaf_{\cofib(f) \otimes_{R} M}} & {\evensheaf_{M}(-\nicefrac{1}{2})} \\
	{\evensheaf_{\cofib(f) \otimes_{R} \Sigma N}}  & {{\evensheaf_{\Sigma N}(-\nicefrac{1}{2})}}
	\arrow[from=2-1, to=2-2]
	\arrow[from=1-1, to=1-2]
	\arrow[from=1-2, to=2-2]
	\arrow[from=1-1, to=2-1]
\end{tikzcd}
\]
The bottom horizontal arrow is zero by the first paragraph. Since the right vertical arrow is a monomorphism by construction, we deduce that the top horizontal arrow is also zero, as needed. 
\end{proof}

\begin{notation}
We will say that a map $M \rightarrow M'$ of $R$-modules is an $\evensheaf_{-}$-monomorphism if 
\[
\evensheaf_{-} \colon \Mod_{R}(\spectra) \rightarrow \sheaves_{\Sigma}(\Perf(R)_{ev}, \abeliangroups) 
\]
takes it to a monomorphism of sheaves. Analogously, we say it is an $\evensheaf_{-}$-epimorphism if it is taken to an epimorphism of sheaves. 
\end{notation}

\begin{corollary}
\label{corollary:even_faithfully_flat_maps_detect_and_preserve_homology_monos}
Let $f \colon R \rightarrow S$ be left faithfully even flat map of $\mathbf{E}_{1}$-rings. Then a map $M \rightarrow M'$ of $R$-modules is an $\evensheaf_{-}$-monomorphism if and only if the base-change $S \otimes_{R} M \rightarrow S \otimes_{R} M'$, considered as a map of $R$-modules, is an $\evensheaf_{-}$-monomorphism. 
\end{corollary}

\begin{proof}
Assume first that $M \rightarrow M'$ is an $\evensheaf_{-}$-monomorpism. By \cref{lemma:any_module_admits_a_mono_on_even_homology_into_a_homological_even}, we can assume that $M'$ is homologically even, in which case so is the cofibre. By \cref{lemma:bimodules_preserving_homological_evenness}, 
\[
S \otimes_{R} \mathrm{cofib}(M \rightarrow M')
\]
is also homologically even and it follows that $\evensheaf_{S \otimes_{R} M} \rightarrow \evensheaf_{S \otimes_{R} M'}$ is a monomorphism, as needed. 

For the converse, consider the commutative diagram 
\[
\begin{tikzcd}
	{\evensheaf_{M}} & {\evensheaf_{M'}} \\
	{\evensheaf_{S \otimes_{R} M}}  & {\evensheaf_{S \otimes_{R} M'}}
	\arrow[from=2-1, to=2-2]
	\arrow[from=1-1, to=1-2]
	\arrow[from=1-2, to=2-2]
	\arrow[from=1-1, to=2-1]
\end{tikzcd}
\]
where both vertical arrows are monomorphisms by \cref{lemma:for_even_faithfully_map_tensoring_gives_ses_of_homology}. If the bottom horizontal arrow is a monomorphism, so must be the upper one, ending the argument. 
\end{proof}

\begin{recollection}
If $f \colon R \rightarrow S$ is a map of $\mathbf{E}_{1}$-rings, then the extension-restriction of scalars adjunction induces a monad on $R$-modules which we also denote by $S \otimes_{R} -$. It follows that every $R$-module $M$ determines an augmented cosimplicial diagram 
\[
S^{\otimes_{R} \bullet} \otimes_{R} M \colon \Delta_{s, +} \rightarrow \Mod_{R}
\]
of the form 
\[
\begin{tikzcd}
	{M} & {S \otimes_{R} M} & {S \otimes_{R} S \otimes_{R} M} & \ldots
	\arrow[yshift=4pt, from=1-3, to=1-4]
 	\arrow[yshift=0pt, from=1-3, to=1-4]
 	\arrow[yshift=-4pt, from=1-3, to=1-4]
	\arrow[yshift=3pt, from=1-2, to=1-3]	          \arrow[yshift=-3pt, from=1-2, to=1-3]
	\arrow[from=1-1, to=1-2]
\end{tikzcd}
\]
which we call the cobar resolution. In good cases, this is a limit diagram, giving a way to understand $M$ through its base-change. The associated spectral sequence
\[
\pi_{*}(S^{\otimes_{R} \bullet} \otimes_{R} M) \Rightarrow \pi_{*}(M)
\]
is called the Adams spectral sequence associated to $f$, or the descent spectral sequence. 
\end{recollection} 

\begin{theorem}[Faithfully flat descent]
\label{theorem:faithfully_flat_descent_for_modules}
Let $f \colon R \rightarrow S$ be a left faithfully even flat map of $\mathbf{E}_{1}$-rings and let $M$ be an $R$-module. Then the canonical map 
\[
\gr_{ev / R}^{*}(M) \rightarrow \gr_{ev / R}^{*}(S^{\otimes_{R} \bullet} \otimes_{R} M)
\]
induced by the cobar resolution is an equivalence of graded spectra. Thus, 
\[
\fil^{*}_{ev / R}(M) \rightarrow \varprojlim \fil^{*}_{ev / R}(S^{\otimes_{R} \bullet} \otimes_{R} M)
\]
induced by the cobar resolution is an equivalence of filtered spectra after completion. 
\end{theorem}

\begin{proof}
By \cref{proposition:homological_resolutions_are_limits_on_even_filtrations_up_to_completion}, it is enough to verify that $S^{\otimes_{R} \bullet} \otimes_{R} M \colon \Delta_{s, +} \rightarrow \Mod_{R}$ is a homological resolution; that is, that the Moore cochain complex 
\begin{equation}
\label{equation:moore_chain_complex_in_proof_of_fef_descent}
0 \rightarrow \evensheaf_{M}(q) \rightarrow \evensheaf_{S \otimes_{R} M}(q) \rightarrow \evensheaf_{S \otimes_{R} S \otimes_{R} M}(q) \rightarrow \ldots
\end{equation}
is exact for any $q$. By replacing $M$ by a suitable (de)suspension we can assume that $q = 0$. Exactness in degree zero is equivalent to the map $\evensheaf_{M} \rightarrow \evensheaf_{S \otimes_{R} M}$ being a monomorphism, which follows from \cref{lemma:for_even_faithfully_map_tensoring_gives_ses_of_homology}. 

We will now show that (\ref{equation:moore_chain_complex_in_proof_of_fef_descent}) is also exact in higher degrees. To avoid multiple tensor products in notation, it will be convenient to write $C^{\bullet}(M) \colonequals S^{\otimes_{R} \bullet} \otimes_{R} M$ for the cobar resolution of $M$. Considering cobar resolutions of $S \otimes_{R} M$ and $\mathrm{cofib}(f) \otimes_{R} M$, another application of \cref{lemma:for_even_faithfully_map_tensoring_gives_ses_of_homology} shows that we have a short exact sequence of chain complexes 
\[
\begin{tikzcd}
	0 & {\evensheaf_{C^{-1}(M)}} & {\evensheaf_{C^{0}(M)}} & {\evensheaf_{C^{1}(M)}} & \ldots \\
	0 & {\evensheaf_{C^{-1}(S \otimes_{R} M)}} & {\evensheaf_{C^{0}(S \otimes_{R} M)}} & {\evensheaf_{C^{1}(S \otimes_{R} M)}} & \ldots \\
	0 & {\evensheaf_{C^{-1}(\mathrm{cofib}(f) \otimes_{R} M)}} & {\evensheaf_{C^{0}(\mathrm{cofib}(f) \otimes_{R} M)}} & {\evensheaf_{C^{1}(\mathrm{cofib}(f) \otimes_{R} M)}} & \ldots
	\arrow[from=1-2, to=1-3]
	\arrow[from=2-2, to=2-3]
	\arrow[from=1-3, to=1-4]
	\arrow[from=1-4, to=1-5]
	\arrow[from=2-4, to=2-5]
	\arrow[from=3-4, to=3-5]
	\arrow[from=3-3, to=3-4]
	\arrow[from=2-3, to=2-4]
	\arrow[from=3-2, to=3-3]
	\arrow[from=1-1, to=1-2]
	\arrow[from=2-1, to=2-2]
	\arrow[from=3-1, to=3-2]
	\arrow[from=2-3, to=3-3]
	\arrow[from=2-2, to=3-2]
	\arrow[from=1-2, to=2-2]
	\arrow[from=1-3, to=2-3]
	\arrow[from=1-4, to=2-4]
	\arrow[from=2-4, to=3-4]
\end{tikzcd}.
\]
Since $S \otimes_{R} M$ is a coalgebra over the comonad $S \otimes_{R} -$, by  \cite[4.7.2.7]{higher_algebra} the cobar resolution $C^{\bullet}(S \otimes_{R} M)$ is split as an augmented cosimplicial object. The splitting provides a contraction of the corresponding Moore cochain complex, from which we deduce that the chain complex in the middle is exact. This shows that for each $p \geq 0$ we have isomorphisms between cohomology sheaves
\[
\Hrm^{p+1}(\evensheaf_{C^{\ast}(M)}) \simeq \Hrm^{p}(\evensheaf_{C^{\ast}(\mathrm{cofib}(f) \otimes_{R} M)}).
\]
Since the first zeroth cohomology sheaf vanishes by the first paragraph, we deduce by induction that they all do, ending the argument. 
\end{proof}

Note that in \cref{theorem:faithfully_flat_descent_for_modules}, all of the even filtrations considered are relative to $R$. This is necessary, since for a general map of $\mathbf{E}_{1}$-rings, the tensor products $S \otimes_{R} \ldots \otimes_{R} S$ do not have a natural ring structure if they involve more than one factor. Thus, if we want to consider a variant of faithfully flat descent where the ring varies, we need to assume more structure on our map.

\begin{theorem}
\label{theorem:fef_descent_for_algebras_over_e2_rings}
Let $R$ be an $\mathbf{E}_{2}$-ring and let $S$ be an $\mathbf{E}_{1}$-$R$-algebra whose unit map is faithfully even flat as map of $\mathbf{E}_{1}$-rings. Then for any $R$-module $M$ the canonical map 
\[
\fil^{*}_{ev / R}(M) \rightarrow \varprojlim\fil^{*}_{ev / S^{\otimes_{R} \bullet}}(S^{\otimes_{R} \bullet} \otimes_{R} M)
\]
is an equivalence of filtered spectra after completion. In particular, this is true for 
\[
\fil^{*}_{ev / R}(R) \rightarrow \varprojlim \fil^{*}_{ev / S^{\otimes_{R} \bullet}}(S^{\otimes_{R} \bullet})
\]
\end{theorem}

\begin{proof}
Note that since $S$ is an $\mathbf{E}_{1}$-$R$-algebra, the left and right $R$-module structures can be identified. It follows that it is even flat as both a left and right $R$-module. 

Keeping in mind  \cref{theorem:faithfully_flat_descent_for_modules}, we just have to show that for any $m \geq 0$ the canonical comparison map 
\[
\fil_{ev / R}^{*}(S^{\otimes_{R} m} \otimes_{R} M) \rightarrow \fil_{ev / S^{\otimes_{R} m}}^{*}(S^{\otimes_{R} m} \otimes_{R} M)
\]
is an equivalence. By \cref{theorem:even_filtration_commutes_with_restriction_of_scalars_for_homologically_even_maps}, we just have to show that $S^{\otimes_{R} m}$ is homologically even for all $m \geq 1$. This is clear since all factors are even flat and even flat modules are stable under tensor products as a consequence of \cref{proposition:even_flat_modules_closed_under_extensions_retracts_and_filtered_colimits} and the fact that tensor products commute with colimits. 
\end{proof}

\begin{example}
\label{example:even_filtration_of_thh_bpn}
Let $\BPn$ be the truncated Brown-Peterson spectrum with
\[
\pi_{*}(\BPn) \simeq \mathbb{Z}_{(p)}[v_{1}, \ldots, v_{n}],
\]
which can be made into an $\mathbf{E}_{3}$-$\MU$-algebra by the work of Hahn-Wilson \cite[Theorem A]{hahn2022redshift}. By the main results of \cite{hahn2022redshift}, the algebraic $K$-theory spectrum $\mathrm{K}(\BPn)$ is of height $n+1$ and satisfies an analogue of Lichtenbaum-Quillen conjectures. A key step in the proof of this remarkable theorem is the analysis of $\THH(\BPn)$ via descent along the map 
\[
\THH(\BPn) \rightarrow \THH(\BPn/\MU)
\]
into the relative topological Hochschild homology. As observed in \cite[Example 4.2.3, 4.2.4]{hahn2022motivic}, the resulting descent filtration can be identified with the $\mathbf{E}_{\infty}$-even filtration
\[
\fil^{*}_{\mathbf{E}_{\infty} \mhyphen ev / \THH(\MU)}(\THH(\BPn))
\]
relative to $\THH(\MU)$. We claim that this filtration is in fact the even filtration
\[
\fil^{*}_{ev}(\THH(\BPn))
\]
and so is an invariant of the ring spectrum $\THH(\BPn)$ itself, providing evidence that our filtration is the right way to obtain a good theory of motivic cohomology of $\mathbf{E}_{2}$-ring spectra. To see this, we need the following two facts observed in \cite[Example 4.2.4]{hahn2022motivic}:
\begin{enumerate}
    \item $\THH(\MU) \rightarrow \MU$ is faithfully even flat (by  \cref{proposition:fef_maps_are_hrw_eff}) and hence so is its base-change $\THH(\BPn) \rightarrow \THH(\BPn/\MU)$ 
    \item $\pi_{*} \THH(\BPn/\MU)$ is even. 
\end{enumerate}
The result then follows from \cref{theorem:fef_descent_for_algebras_over_e2_rings}. 
\end{example}

\section{Comparison with the \texorpdfstring{$\mathbf{E}_{\infty}$-even}{E-infinity-even} filtration} 
\label{section:comparison_with_hrw_filtration}

In this section, we compare the even filtration studied in the current work with the even filtration of $\mathbf{E}_{\infty}$-rings introduced by Hahn-Raksit-Wilson \cite{hahn2022motivic}. To avoid confusion, we refer to the latter as the $\mathbf{E}_{\infty}$-even filtration. 

We first recall the definition of the $\mathbf{E}_{\infty}$-even filtration; for details, see \cite[\S 2]{hahn2022motivic}. 

\begin{notation}
We write $\Mod_{\mathbf{E_{\infty}}}(\spectra)$ for the pullback of $\infty$-categories
\[
\Mod_{\mathbf{E_{\infty}}}(\spectra) \colonequals \Mod(\spectra) \times_{\Alg_{\mathbf{E}_{1}}(\spectra)} \Alg_{\mathbf{E}_{\infty}}(\spectra),
\]
where $\Mod(\spectra)$ is as in \cref{notation:infty_category_of_pairs_of_algebra_and_a_module}. Concretely, $\Mod_{\mathbf{E_{\infty}}}(\spectra)$ is the $\infty$-category of pairs $(A, M)$, where $A \in \Alg_{\mathbf{E}_{\infty}}(\spectra)$ and $M \in \Mod_{A}(\spectra)$. We write 
\[
\Mod^{ev}_{\mathbf{E_{\infty}}}(\spectra) \subseteq \Mod_{\mathbf{E_{\infty}}}(\spectra)
\]
for the full subcategory spanned by those pairs such that $\pi_{*}A$ is even. 
\end{notation}

\begin{definition}
\label{definition:e_infty_even_filtration}
The \emph{$\mathbf{E}_{\infty}$-even filtration} 
\[
\fil^{\ast}_{\mathbf{E}_{\infty} \mhyphen ev / -}(-) \colon \Mod_{\mathbf{E_{\infty}}}(\spectra) \rightarrow \Fil\spectra
\]
is the right Kan extension of the functor
\[
U_{\ast} \colon \Mod^{ev}_{\mathbf{E_{\infty}}}(\spectra) \rightarrow \Fil\spectra
\]
given by 
\[
U_{q}(A, M) \colonequals \tau_{\geq 2q} M. 
\]
along the inclusion $\Mod_{\mathbf{E}_{\infty}}^{ev}(\spectra) \hookrightarrow \Mod_{\mathbf{E}_{\infty}}(\spectra)$. 
\end{definition}

\begin{remark}
Concretely, if $(A, M) \in \Mod_{\mathbf{E_{\infty}}}(\spectra)$, then the $\mathbf{E}_{\infty}$-even filtration is given by the limit 
\[
\fil^{q}_{\mathbf{E}_{\infty} \mhyphen ev / A}(M) \simeq \varprojlim \tau_{\geq 2q} (B \otimes_{A} M),
\]
taken over all $\mathbf{E}_{\infty}$-ring maps $A \rightarrow B$ with $\pi_{*}B$ even, see \cite[{Remark 2.1.3}]{hahn2022motivic}.
\end{remark}

\begin{construction}
\label{construction:canonical_comparison_map_between_even_and_einfty_even_filtrations}
In \cref{theorem:even_filtration_is_lax_symmetric_monoidal_as_a_functor_on_pairs}, we refined the even filtration of \cref{definition:even_filtration_of_an_r_module} to a functor on the $\infty$-category $\Mod$. Considering the composite 
\[
\begin{tikzcd}
	\Mod_{\mathbf{E}_{\infty}}(\spectra) & \Mod(\spectra) & \Mod(\Fil \spectra) & {\Fil \spectra}
	\arrow[from=1-1, to=1-2]
	\arrow[from=1-2, to=1-3]
	\arrow[from=1-3, to=1-4]
\end{tikzcd},
\]
where the first arrow is induced by the forgetful functor $\Alg_{E_{\infty}}(\spectra) \rightarrow \Alg(\spectra)$ and the third one is the forgetful functor $(A^{\ast}, M^{\ast}) \mapsto M^{\ast}$, the even filtration defines a functor $\Mod_{\mathbf{E}_{\infty}}(\spectra) \rightarrow \Fil \spectra$. We will describe a canonical natural transformation
\begin{equation}
\label{equation:natural_transf_from_even_fil_to_einfty_even_fil}
\fil^{*}_{ev / -}(-) \rightarrow \fil^{*}_{\mathbf{E}_{\infty} \mhyphen ev / -}(-)
\end{equation}
into the $\mathbf{E}_{\infty}$-even filtration. 

Since the right hand side of (\ref{equation:natural_transf_from_even_fil_to_einfty_even_fil}) is defined as a right Kan extension, to construct the needed natural transformation it is enough to define it on the subcategory $\Mod^{ev}_{\mathbf{E}_{\infty}}(\spectra)$. We claim that on this subcategory, the two filtrations are canonically equivalent, providing the needed natural transformation. Indeed, if $\pi_{*}A$ is even, then both filtrations are given by
\[
(A, M) \mapsto \tau_{\geq 2*} M,
\]
by definition in the case of the $\mathbf{E}_{\infty}$-even filtration and by \cref{proposition:even_filtration_over_r_with_even_htpy_is_the_postnikov_filtration} in the case of the even filtration of \cref{definition:even_filtration_of_an_r_module}. 
\end{construction}

\begin{theorem}
\label{theorem:comparison_between_ev_and_einfty_ev_filtrations}
Let $R$ be an $\mathbf{E}_{\infty}$-ring which admits a faithfully even flat map $R \rightarrow S$ in the sense of \cref{definiton:even_faithfully_flat_maps_of_rings} into a $\pi_{*}$-even $\mathbf{E}_{\infty}$-ring $S$. Then for any $R$-module $M$ the comparison map of filtered spectra
\[
\fil^{*}_{ev / R}(M) \rightarrow \fil^{*}_{\mathbf{E}_{\infty} \mhyphen ev / R}(M)
\]
of \cref{construction:canonical_comparison_map_between_even_and_einfty_even_filtrations} exhibits the target as a completion of the source. In particular, it is an equivalence after completion. 
\end{theorem}

\begin{proof}
The $\mathbf{E}_{\infty}$-even filtration is always complete \cite[Remark 2.1.6]{hahn2022motivic}, so it is enough to show that the canonical map is an equivalence after passing to the associated graded objects. As both filtrations satisfy faithfully even flat descent, the even filtration by \cref{theorem:fef_descent_for_algebras_over_e2_rings} and the $\mathbf{E}_{\infty}$-even filtration by a combination of \cref{proposition:fef_maps_are_hrw_eff} and \cite[Corollary 2.2.14]{hahn2022motivic}, the comparison map can be identified with 
\[
\varprojlim \gr^{*}_{ev / S^{\otimes_{R} [m]}}(S^{\otimes_{R} [m]} \otimes_{R} M) \rightarrow \varprojlim \gr^{*}_{\mathbf{E}_{\infty} \mhyphen ev / S^{\otimes_{R} [m]}}(S^{\otimes_{R} [m]} \otimes_{R} M),
\]
where the limit is taken over $[m] \in \Delta$. 

Each of the tensor products $S^{\otimes_{R} [m]}$ has homotopy groups concentrated in even degrees by \cref{proposition:tensor_characterization_of_even_flat_modules} as it is a tensor product of even flat, $\pi_{\ast}$-even $R$-modules. In the case of $\pi_{\ast}$-even rings, the two filtration agree as observed in \cref{construction:canonical_comparison_map_between_even_and_einfty_even_filtrations}, ending the argument. 
\end{proof}

In \cite[\S 5]{hahn2022motivic}, Hahn-Raksit-Wilson show that the $\mathbf{E}_{\infty}$-even filtration of various rings recovers various classically studied and important filtrations, implying that they are in fact a functorial invariant of the $\mathbf{E}_{\infty}$-ring itself. In particular, they show that this is true for: 
\begin{enumerate}
    \item the sphere, for which the $\mathbf{E}_{\infty}$-even filtration is the Adams-Novikov filtration,
    \item $\mathrm{HH}(R/k)$, where $k \rightarrow R$ is a quasi-lci map of commutative rings, where one recovers the Hochschild-Kostant-Rosenberg filtration, 
    \item $\THH(R)^{\wedge}_{p}$, where $R$ is a $p$-quasisyntomic, $p$-complete commutative ring, where one recovers the Bhatt-Morrow-Scholze filtration of \cite{bhatt2019topological}, 
    \item $\THH(R)$ for $R$ a quasisyntomic ring, where one recovers the Bhatt-Lurie filtration of \cite{bhatt2022absolute}
\end{enumerate}

\begin{corollary}
\label{corollary:in_various_example_hrw_filtration_is_ev_filtration}
In each of the above four examples, the respective filtration coincides with the even filtration of \cref{definition:even_filtration_of_an_r_module}. 
\end{corollary}

\begin{proof}
The relevant comparison results with the $\mathbf{E}_{\infty}$-filtration are all proven in \cite{hahn2022motivic} using eff descent. By \cref{theorem:comparison_between_ev_and_einfty_ev_filtrations}, it is enough to verify that each of the maps appearing in the proof is in fact faithfully even flat in the sense of \cref{definiton:even_faithfully_flat_maps_of_rings}. This follows from \cref{proposition:fef_maps_are_hrw_eff} since all of the relevant $\mathbf{E}_{\infty}$-rings are connective. 
\end{proof}

\begin{remark}
One consequence of \cref{corollary:in_various_example_hrw_filtration_is_ev_filtration} is that the even filtration on $\THH(R)$ and its variants depends only on the $\mathbf{E}_{1}$-ring structure on topological Hochschild homology. Beware that if we think of $\THH(R)$ as a functor of $R$, this means that the even filtration on $\THH$ depends on the $\mathbf{E}_{2}$-ring structure of $R$.
\end{remark}

\begin{remark}
In \cite{hahn2022motivic}, Hahn-Raksit-Wilson also prove comparison results with motivic filtrations on $\TC^{-}(-)$, $\TC(-)$ and $\TP(-)$. However, in these cases their definition of an appropriate $\mathbf{E}_{\infty}$-even filtration is different from the one we recalled in \cref{definition:e_infty_even_filtration}, as one has to take the $S^{1}$-action into account. 

In upcoming joint work with Raksit \cite{motivic_cohomology_of_e2_rings}, we will show how the constructions in this paper can be applied to construct motivic filtrations on $\TC^{-}(-)$ and $\TP(-)$ of reasonable $\mathbf{E}_{2}$-rings. We expect that the same is possible in the case of $\TC(-)$, but it is not yet clear how to do so. 
 \end{remark}

Beware that for a general $\mathbf{E}_{\infty}$-ring, the $\mathbf{E}_{\infty}$-even and even filtrations can disagree. We learned the following instructive example from Robert Burklund:

\begin{example}
\label{example:roberts_example_of_even_and_einfty_even_being_different}
Let $R \colonequals \mathbb{F}_{2} \otimes \Sigma^{\infty}_{+} \Omega^{\infty} S^{1}$ be the free $\mathbb{F}_{2}$-$\mathbf{E}_{\infty}$-algebra on a class $x \in \pi_{1} R$. The homotopy groups of $R$ can be described in terms of Dyer-Lashof operations, namely 
 \[
\pi_{*} R \simeq \mathbb{F}_{2}[Q^{J}x]
 \]
 forms a polynomial ring in generators
 \[
 Q^{I}x \colonequals Q^{j_1} Q^{j_2} \ldots Q^{j_p} x
 \]
 satisfying $j_{i} \leq 2 j_{i+1}$ and $j_{1} - j_{2} - \ldots - j_{p} > | x | = 1$, see \cite[Example 1.5.10]{lawson2020n}. 

Observe that any $\mathbf{E}_{\infty}$-ring map $R \rightarrow S$ into a ring with $\pi_{*}S$ even necessarily sends $x$ to zero, from which it follows that it factors as 
 \[
R \rightarrow \mathbb{F}_{2} \rightarrow S 
 \]
so that all elements of positive degree are in the kernel of the induced map on homotopy. It follows that if $y \in \pi_{2n}R$ is an element of positive even degree, then it maps to zero in
\[
\gr^{n}_{\mathbf{E}_{\infty} \mhyphen ev}(R) \colonequals \varprojlim_{R \rightarrow S} \pi_{2n} S.
 \]
On the other hand, the structure of the even filtration of \cref{definition:even_filtration_of_an_r_module} is more straightforward; in particular, all polynomial generators in $\pi_{2n}R$ are detected in $\gr^{n}_{ev}(R)$. 

To see this, we will show the following more general statement: Let $A$ be an $\mathbf{E}_{2}$-$\mathbb{F}_{2}$-algebra such that $\pi_{*}A \simeq \mathbb{F}_{2}[x_{i}, y_{j}]$ is a polynomial algebra in odd degree variables $x_{i}$ and even degree variables $y_{j}$. Then 
 \begin{equation}
\label{equation:answer_for_even_coh_of_polys_over_f2}
\evencoh^{*, *}(A) \simeq \mathbb{F}_{2}[\widetilde{x}_{i}, \widetilde{y}_{j}]
\end{equation}
with $\widetilde{x}_{i} \in \evencoh^{1, \nicefrac{|x_{i}|+1}{2}}(A)$ and $\widetilde{y}_{j} \in \evencoh^{0, \nicefrac{|y_{j}|}{2}}(A)$ are detecting the corresponding elements of $\pi_{*}A$. In particular, the even spectral sequence of $A$ collapses. 
 
From now on, let us not distinguish in notation between even and odd degree generators and denote the totality of both by $(z_{i})_{i \in I}$, with the index set $I$ implicitly well-ordered. Observe that any finite subset $A \subseteq I$ determines a map of $\mathbf{E}_{1}$-$\mathbb{F}_{2}$-algebras
 \[
\bigotimes_{\mathbb{F}_{2}, i \in A} \mathbb{F}_{2}[z_{i}] \rightarrow \bigotimes_{\mathbb{F}_{2}} A \rightarrow A
 \]
from the tensor product of free $\mathbf{E}_{1}$-$\mathbb{F}_{2}$-algebras, where the second map is the multiplication of $A$, which is $\mathbf{E}_{1}$ as $A$ is $\mathbf{E}_{2}$. We deduce that as an $\mathbf{E}_{1}$-ring, $A$ can be written as
\begin{equation}
\label{equation:a_with_poly_htpy_a_filtered_colimit_of_frees}
 A \simeq \varinjlim _{A \subseteq I} (\bigotimes_{i \in A} \mathbb{F}_{2}[z_{i}]),
\end{equation}
a filtered colimit of finite tensor products. By \cref{proposition:even_cohomology_computed_using_cochain_complex_of_r_modules}, the even filtration of an $\mathbf{E}_{1}$-ring can be calculated by constructing an appropriate chain complex of modules
\[
E_{0} \rightarrow E_{1} \rightarrow \ldots,
\]
where each $E_{i}$ has even homotopy groups (among other properties, see \cref{construction:cochain_complex_computing_even_cohomology}), and taking cohomology of $\pi_{*} E_{\bullet}$. By  (\ref{equation:a_with_poly_htpy_a_filtered_colimit_of_frees}), the resolution of $A$ can be obtained as a filtered colimit of tensor products of resolutions of $\mathbb{F}_{2}[z_{i}]$. Note that here we use that we're working over $\mathbb{F}_{2}$, as this guarantees that a tensor product of modules with even homotopy groups also has even homotopy groups. We deduce that we have a K\"{u}nneth-style isomorphism
 \[
\evencoh^{*, *}(A) \simeq \varinjlim_{A \subseteq I} (\bigotimes_{i \in A, \mathbb{F}_{2}} \evencoh^{*, *}(\mathbb{F}_{2}[z_{i}]))
 \]
which reduces us to the case of a free $\mathbf{E}_{1}$-algebra on a single generator. If the generator $z$ is of even degree, then the even cohomology is as claimed in (\ref{equation:answer_for_even_coh_of_polys_over_f2}) by \cref{corollary:even_cohomology_of_even_modules}. If $z$ is of odd degree, then we have a resolution of the form 
\[
\begin{tikzcd}
	& k & {\Sigma^{a} k} & {\Sigma^{2a} k} & \ldots \\
	{k[z]} & {\Sigma^{a}k[z]} & {\Sigma^{2a} k[z]} & {\Sigma^{3a} k[z]}
	\arrow[from=1-2, to=2-2]
	\arrow[from=2-2, to=1-3]
	\arrow[from=1-2, to=1-3]
	\arrow[from=1-3, to=2-3]
	\arrow[from=2-3, to=1-4]
	\arrow[from=1-4, to=2-4]
	\arrow[from=1-3, to=1-4]
	\arrow[from=2-4, to=1-5]
	\arrow[from=1-4, to=1-5]
	\arrow[from=2-1, to=1-2]
\end{tikzcd},
\]
where $a = |z|+1$. Each of the diagonal arrows is surjective on homotopy groups, from which we deduce that the horizontal arrows are zero after passing to homotopy. We deduce that $\evencoh^{*, *}(\mathbb{F}_{2}[z]) \simeq \mathbb{F}_{2}[\widetilde{z}]$ with $|\widetilde{z}| = (1, \nicefrac{a}{2})$, as claimed in (\ref{equation:answer_for_even_coh_of_polys_over_f2}). 
\end{example}

\begin{remark}
In the context of \cref{example:roberts_example_of_even_and_einfty_even_being_different}, we observe that lax monoidality of the even filtration of \cref{theorem:even_filtration_is_lax_symmetric_monoidal} provides a canonical canonical map
\[
\bigotimes_{i \in A, \fil^{\ast}_{ev}(\mathbb{F}_{2})} \fil^{\ast}_{ev}(\mathbb{F}_{2}[z_{i}]) \rightarrow \fil^{\ast}_{ev}(A)
\]
of filtered $\mathbf{E}_{1}$-rings. We expect that the K\"{u}nneth-style isomorphism constructed ad hoc in \cref{example:roberts_example_of_even_and_einfty_even_being_different} using chain complexes of modules coincides with the one induced by the canonical comparison map, but we will not explore this matter in the current work. We believe that it would be a very interesting question to explore K\"unneth formulas for the even filtration in a wider context. 
\end{remark}
\begin{warning}
\label{warning:more_extreme_example_of_even_and_einfty_even-filtration_diverging}
The following variation on \cref{example:roberts_example_of_even_and_einfty_even_being_different} shows that the even and $\mathbf{E}_{\infty}$-even filtrations can diverge even more drastically once we leave the world of connective $\mathbf{E}_{\infty}$-rings.

Let $R$ be as in \cref{example:roberts_example_of_even_and_einfty_even_being_different}, so that $\pi_{*}R \simeq \mathbb{F}_{2}[Q^{j}x]$ is a polynomial algebra and let $e \in \pi_{2n} R$ be an even degree polynomial generator; for example, we can take $Q^{3}x \in \pi_{4}R$. As a localization, $R[e^{-1}]$ acquires a canonical $\mathbf{E}_{\infty}$-ring structure with $\pi_{*}(R[e^{-1}]) \simeq (\pi_{*}R)[e^{-1}]$. Observe that any map $f \colon R \rightarrow S$ into an even $\mathbf{E}_{\infty}$-ring sends $x$ to zero, so that also $f(e) = 0$. We deduce that the only map $R[e^{-1}] \rightarrow S$ from the localization to an even $\mathbf{E}_{\infty}$-ring is the zero map and consequently 
\[
\fil^{*}_{\mathbf{E}_{\infty} \mhyphen ev}(R[e^{-1}]) = 0. 
\]
On the other hand, an analysis analogous to the one given in \cref{example:roberts_example_of_even_and_einfty_even_being_different} shows that the the even filtration is complete and 
\[
\evencoh^{*, *}(R[e^{-1}]) \simeq \mathbb{F}_{2}[\widetilde{e}_{j}, \widetilde{v}_{j}][\widetilde{e}^{-1}],
\]
where $\widetilde{e}$ is the Hurewicz image of $e$.
\end{warning}

\begin{remark}
\label{remark:a_map_which_is_eff_but_not_fef}
Observe that the zero map $R[e^{-1}] \rightarrow 0$ from the $\mathbf{E}_{\infty}$-ring appearing in \cref{warning:more_extreme_example_of_even_and_einfty_even-filtration_diverging} is evenly faithfully flat in the sense of Hahn-Raksit-Wilson \cite[2.2.13]{hahn2022motivic}, but it is not faithfully even flat in the sense of \cref{definiton:even_faithfully_flat_maps_of_rings}. 
\end{remark} 

\section{Even cohomology of connective rings} 

In this section we study even cohomology of connective rings. We will use $R$ to denote an $\mathbf{E}_{1}$-ring and $M$ to denote an $R$-module in spectra. 

\subsection{Vanishing above the Milnor line} 

We first show that even cohomology of connective rings vanishes above the ``Milnor line'' $p = q$. 

\begin{theorem}
\label{theorem:vanishing_of_even_coh_of_connective_ring_above_milnor_line}
Let $R$ be connective and let $M$ be connective, homologically even. Then
\begin{enumerate}
\item $\evencoh^{0, 0}(R, M) \simeq \pi_{0} M$, 
\item $\evencoh^{p, q}(R, M) = 0$ for $p > q$. 
\end{enumerate}
\end{theorem}

\begin{proof}
We recall from \cref{proposition:even_cohomology_computed_using_cochain_complex_of_r_modules} a recipe for calculating the even cohomology of $M$. Using \cref{construction:cochain_complex_computing_even_cohomology}, we can find a diagram 
\begin{equation}
\label{equation:cochain_complex_of_even_modules_computing_even_cohomology_in_proof_of_milnor_vanishing}
\begin{tikzcd}
	& {E_{0}} & {E_{1}} & {E_{2}} & \ldots \\
	{C_{-1}} & {C_{0}} & {C_{1}} & {C_{2}}
	\arrow[from=2-1, to=1-2]
	\arrow[from=1-2, to=2-2]
	\arrow[from=2-2, to=1-3]
	\arrow[from=1-2, to=1-3]
	\arrow[from=1-3, to=2-3]
	\arrow[from=2-3, to=1-4]
	\arrow[from=1-3, to=1-4]
	\arrow[from=1-4, to=2-4]
	\arrow[from=2-4, to=1-5]
	\arrow[from=1-4, to=1-5]
\end{tikzcd}
\end{equation}
with the properties that $C_{-1} = M$ and that each $C_{i} \rightarrow E_{i} \rightarrow C_{i+1}$ is a cofibre sequence with the first map a $\pi_{*}$-even envelope. We then have a canonical isomorphism 
\begin{equation}
\label{equation:even_cohomology_as_cohomology_of_chain_cpx_in_proof_of_milnor_vanishing}
\evencoh^{p, q}(M) \simeq \mathrm{H}^{p}(\pi_{2q} E_{\bullet}). 
\end{equation}

By \cref{proposition:pi_star_even_envelopes_exist}, a $\pi_{*}$-even envelope of a connective $R$-module can be chosen so that the cofibre is $2$-connective. Since $M$ is connective, it follows that we can choose a diagram (\ref{equation:cochain_complex_of_even_modules_computing_even_cohomology_in_proof_of_milnor_vanishing}) with the property that $E_{i}$ is $(2i)$-connective for each $i \geq 0$. For such a diagram, both parts follow from (\ref{equation:even_cohomology_as_cohomology_of_chain_cpx_in_proof_of_milnor_vanishing}), as the groups $\pi_{2q} E_{p}$ vanish for $p > q$ and we have an isomorphism $\pi_{0} M \simeq \pi_{0}E_{0}$. 
\end{proof}

\begin{remark}[Vanishing in terms of the Adams grading]
As we discussed in \cref{remark:p_q_grading_convention_for_even_cohomology}, the even cohomology groups can be often identified with the $E_{2}$-term of a suitable Adams spectral sequence in $R$-modules using the identification 
\[
E_{2}^{s, t} \simeq \Hrm_{ev}^{s, \frac{1}{2} \cdot t}(R).
\]
Since even cohomology of an $\mathbf{E}_{1}$-ring vanishes when $p < 0$ by \cref{remark:even_cohomology_vanishes_in_negative_degrees_and_this_is_consistent_with_identification_with_graded_of_the_even_filtration} or when $q$ is a half-integerby \cref{lemma:every_perfect_even_is_homogically_even}, for an arbitrary $\mathbf{E}_{1}$-ring $R$ we have that 
\begin{enumerate}
    \item $E_{2}^{s, t} = 0$ if $s < 0$,
    \item $E_{2}^{s, t} = 0$ if $s \not\equiv (t-s) \mod 2$. 
\end{enumerate}
If $R$ is moreover connective, \cref{theorem:vanishing_of_even_coh_of_connective_ring_above_milnor_line} implies that 
\begin{enumerate}
    \item[(3)] $E_{2}^{s, t} = 0$ if $s > (t-s)$; 
\end{enumerate}
that is, in the $(s, t-s)$ plane we have a vanishing line of slope $1$. 

The above are three familiar properties of the Adams-Novikov spectral sequence, which we can identify with the even spectral sequence of the sphere by \cref{corollary:in_various_example_hrw_filtration_is_ev_filtration}. Thus, \cref{theorem:vanishing_of_even_coh_of_connective_ring_above_milnor_line} can be interpreted as saying that the even spectral sequence of an arbitrary connective $\mathbf{E}_{1}$-ring has vanishing properties akin to that of the Adams-Novikov spectral sequence. 
\end{remark}

\begin{theorem}[Completness of the even filtration]
\label{theorem:completness_of_the_even_filtration}
Let $R$ be connective and let $M$ be connective, homologically even.Then 
\begin{enumerate}
    \item $\fil_{ev}^{q}(M) \simeq M$ for $q < 0$ and 
    \item the even filtration $\fil^{*}_{ev}(M)$ is complete; that is, $\varprojlim \fil^{*}_{ev}(M) \simeq 0$. 
\end{enumerate} 
\end{theorem}

\begin{proof}
We start with the first part. Recall from \cref{theorem:associated_graded_of_even_filtration_and_even_cohomology} that the associated graded object of the even filtration satisfies 
\[
\pi_{t} \gr^{q}_{ev} M \simeq \evencoh^{2q-t, q}(R, M).
\]
Since even cohomology vanishes when $p < 0$, as we observed in \cref{remark:even_cohomology_vanishes_in_negative_degrees_and_this_is_consistent_with_identification_with_graded_of_the_even_filtration}, from \cref{theorem:vanishing_of_even_coh_of_connective_ring_above_milnor_line} we deduce that the even cohomology of $M$ vanishes in negative weight, so that the maps 
\[
\fil^{0}_{ev}(M) \rightarrow \fil^{-1}_{ev}(M) \rightarrow \fil^{-2}_{ev}(M) \rightarrow \ldots
\]
are all equivalences. As their colimit is equivalent to $M$ by \cref{proposition:connectivity_of_even_filtration_maps_and_exhaustivity_of_the_even_filtration}, we deduce that $\fil^{q}_{ev}(M) \simeq M$ for all $q < 0$. 

We now show that $\varprojlim \fil^{q}_{ev}(M)$ is $(-2)$-connective. Using the Milnor exact sequence associated to an inverse limit, it is enough to show that $\fil^{q}_{ev}(M)$ is $(-1)$-connective for each $q > 0$. We've already seen that this is true when $q = 0$, and we'll prove it for positive $q$ via induction. 

The identification between even cohomology and homotopy groups of the associated graded object together with \cref{theorem:vanishing_of_even_coh_of_connective_ring_above_milnor_line} show that for each $q \geq 0$, $\gr_{ev}^{q}(M)$ is $0$-connective (in fact, $q$-connective, but we will not need it). Since we have a cofibre sequence 
\[
\Sigma^{-1} \gr_{ev}^{q}(M) \rightarrow \fil^{q+1}_{ev}(M) \rightarrow \fil^{q}_{ev}(M),
\]
the $(-1)$-connectivity of $\fil^{q+1}_{ev}$ follows from that of $\fil^{q}$, ending the inductive argument. This shows that $\varprojlim \fil^{q}_{ev}(M)$ is $(-2)$-connective, as claimed. 

We will now inductively show that for any $k \in \mathbb{Z}$, the spectrum $\varprojlim \fil^{q}_{ev}(M)$ is $k$-connective. The case $k = -2$ was proved above, so we assume that $k \geq -1$. Using \cref{proposition:pi_star_even_envelopes_exist}, we can find a cofibre sequence 
\[
M \rightarrow E \rightarrow C 
\]
where both $E$ and $C$ are also homologically even, $\pi_{*}E$ is even and $C$ is $2$-connective. This yields a cofibre sequence
\[
\varprojlim \fil^{*}_{ev}(M) \rightarrow \varprojlim \fil^{*}_{ev}(E) \rightarrow \varprojlim \fil^{*}_{ev}(C)
\]
The middle term vanishes by \cref{lemma:pistar_even_means_homologically_even_and_even_filtration_is_whitehead}, giving an equivalence
\begin{equation}
\label{equation:equivalences_of_limits_of_the_even_filtration}
\varprojlim \fil^{*}_{ev}(M) \simeq \Sigma^{-1} (\varprojlim \fil^{*}_{ev}(C)) \simeq \Sigma (\varprojlim \fil^{*}_{ev}(\Sigma^{-2} C))
\end{equation}
where we use that
\[
\fil^{q}_{ev}(\Sigma^{-2} C) \simeq \Sigma^{-2} \fil_{ev}^{q+1}(C)
\]
which follows immediately from the definition. As $\Sigma^{-2} C$ is connective, the right hand side of (\ref{equation:equivalences_of_limits_of_the_even_filtration}) is $k$-connective by the inductive assumption. We deduce that the same is true for the left hand side. It follows that $\varprojlim \fil^{q}_{ev}(M)$ is $k$-connective for all $k \in \mathbb{Z}$ and thus must vanish, proving the second part. 
\end{proof}

\begin{corollary}
\label{corollary:complete_convergence_of_the_even_spectral_sequence_for_homologically_even_connectives}
If $R$ is connective and $M$ is homologically even, bounded below, then the even spectral sequence 
\[
\evencoh^{p, q}(R, M) \Rightarrow \pi_{2q-p}(M)
\]
converges strongly in the sense of Boardman. 
\end{corollary}

\begin{proof}
By considering an appropriately large even suspension of $M$, we can assume that it is connective. Since the even spectral sequence of \cref{definition:even_spectral_sequence} is the spectral sequence associated to the filtered spectrum $\fil_{ev}^{*}(M)$, conditional convergence in the sense of Boardman follows from completeness of the even filtration, which we've shown in \cref{theorem:completness_of_the_even_filtration}.

As the differentials are of bidegree $|d_{r}| = (2r-1, r-1)$, the vanishing line of \cref{theorem:vanishing_of_even_coh_of_connective_ring_above_milnor_line} implies that for any fixed bidegree $(p, q)$, the group $E^{p, q}_{r}$ stabilizes for large $r$. This implies strong convergence by \cite[{Theorem 7.3}]{boardman1999conditionally}. 
\end{proof}

\begin{remark}[The case of the $\mathbf{E}_{\infty}$-even filtration]
The Hahn-Raksit-Wilson filtration attached to an $\mathbf{E}_{\infty}$-ring $R$ is always complete by construction, and the question of whether the associated spectral sequence converges to $\pi_{*}R$ instead depends on whether the filtration is exhaustive; that is, whether $\varinjlim \fil_{\mathbf{E}_{\infty} \mhyphen ev}^{*}(R) \simeq R$. This was shown to be the case if $R$ is connective by Achim Krause and Robert Burklund, giving an analogue of \cref{theorem:completness_of_the_even_filtration} also in this context\footnote{Private communication with Achim Krause and Robert Burklund.}.
\end{remark}

\subsection{Cohomology in low weights} 
\label{subsection:cohomology_in_low_weights}

In \cref{theorem:vanishing_of_even_coh_of_connective_ring_above_milnor_line}, we have shown that the weight zero even cohomology of a connective $R$-module is concentrated in cohomological degree zero, where $\evencoh^{0, 0}(M) \simeq \pi_{0} M$. In this section, we give a calculation of groups in weight one and deduce a calculation of $\evencoh^{2,2}(M)$. 

\begin{proposition}
\label{proposition:weight_one_cohomology_of_a_connective_homologically_even_module}
Let $R$ be connective and $M$ homologically even and connective. Then
\begin{enumerate}
    \item $\evencoh^{0, 1}(M) \simeq \mathrm{coker}(\pi_{1} R \otimes_{\mathbb{Z}} \pi_{1} M \rightarrow \pi_{2} M)$, cokernel of the multiplication map, 
    \item $\evencoh^{1, 1}(M) \simeq \pi_{1} M$.
\end{enumerate}
\end{proposition}

\begin{proof}
Each homotopy class $m \in \pi_{1} M$ determines a map $\Sigma R \rightarrow M$. We will consider the direct sum 
\[
\bigoplus_{m \in \pi_{1} M} \Sigma R \rightarrow M 
\]
of all of these maps and the corresponding cofibre sequence 
\[
M \rightarrow N \rightarrow C,
\]
where $C \simeq \bigoplus \Sigma^{2} R$ is the suspension of the direct sum above. This is a cofibre sequence of homologically even $R$-modules, so that we have a short exact sequence
\[
0 \rightarrow \evensheaf_{M} \rightarrow \evensheaf_{N} \rightarrow \evensheaf_{C} \rightarrow 0.
\]
This induces a long exact sequence of even cohomology 
\begin{equation}
\label{equation:long_exact_seq_of_even_coh_in_weight_1_calculation}
0 \rightarrow \evencoh^{0, 1}(M) \rightarrow \evencoh^{0, 1}(N) \rightarrow \evencoh^{0, 1}(C) \rightarrow \evencoh^{1, 1}(M) \rightarrow \evencoh^{1, 1}(N) \rightarrow 0.
\end{equation}

Since $C$ is $2$-connective, by a combination of \cref{remark:even_cohomology_of_suspension_in_terms_of_weights} and \cref{theorem:vanishing_of_even_coh_of_connective_ring_above_milnor_line} we know that its weight zero even cohomology vanishes and $\evencoh^{0, 1}(C) \simeq \pi_{2} C$. As the first possibly non-zero odd degree homotopy group of $N$ is $\pi_{3}N$, we can construct a $\pi_{\ast}$-even envelope 
\[
N \rightarrow E 
\]
whose cofibre is $4$-connective and so its even cohomology vanishes below weight two by another application of \cref{theorem:vanishing_of_even_coh_of_connective_ring_above_milnor_line}. It follows that 
\[
\evencoh^{0, 1}(N) \simeq \evencoh^{0, 1}(E) \simeq \pi_{2}E \simeq \pi_{2} N
\]
and 
\[
\evencoh^{1, 1}(N) \simeq \evencoh^{1, 1}(E) = 0. 
\]

Taking the previous paragraph into account, (\ref{equation:long_exact_seq_of_even_coh_in_weight_1_calculation}) becomes 
\[
0 \rightarrow \evencoh^{0, 1}(M) \rightarrow \pi_{2} N \rightarrow \pi_{2} C \rightarrow \evencoh^{1, 1}(M) \rightarrow 0.
\]
Since $C \simeq \bigoplus \Sigma^{2} R$ and 
\[
\ldots \rightarrow \pi_{3}C \rightarrow \pi_{2} M \rightarrow \pi_{2} N \rightarrow \pi_{2} C \rightarrow \pi_{1} M \rightarrow 0
\]
is exact we deduce that 
\[
\evencoh^{0, 1}(M) \simeq \ker(\pi_{2} N \rightarrow \pi_{2} C) \simeq \mathrm{coker}(\bigoplus \pi_{1} R \rightarrow \pi_{2} M) \simeq \mathrm{coker}(\pi_{1} R \otimes_{\mathbb{Z}} \pi_{1} M \rightarrow \pi_{2} M).
\]
This shows the first needed isomorphism. Similarly, we have 
\[
\evencoh^{1, 1}(M) \simeq \coker(\pi_{2} N \rightarrow \pi_{2} C) \simeq \pi_{1} M. 
\]
\end{proof}

\begin{corollary}
\label{corollary:bidegree_two_two_even_coh_of_connective_hom_even_module}
Let $R$ be connective and $M$ homologically even, connective. Then 
\[
\evencoh^{2, 2}(M) \simeq \mathrm{im} (\pi_{1} R \otimes_{\mathbb{Z}} \pi_{1} M \rightarrow \pi_{2} M),
\]
the image of the multiplication map. 
\end{corollary}

\begin{proof}
By \cref{corollary:complete_convergence_of_the_even_spectral_sequence_for_homologically_even_connectives}, in the case at hand the even spectral sequence 
\[
E_{2}^{p, q} \simeq \evencoh^{p, q}(M) \Rightarrow \pi_{2q-p}(M)
\]
converges strongly. As the differentials are of bidegree $(2r-1, r-1)$ for $r \geq 2$, the vanishing line of \cref{theorem:vanishing_of_even_coh_of_connective_ring_above_milnor_line} together with the fact that even cohomology vanishes when $p < 0$ imply that all of the differentials originating or ending at $\evencoh^{2, 2}(M)$ or $\evencoh^{0, 1}(M)$ vanish. It follows that we have a short exact sequence 
\[
0 \rightarrow \evencoh^{2, 2}(M) \rightarrow \pi_{2}(M) \rightarrow \evencoh^{0, 1}(M) \rightarrow 0, 
\]
and the identification of the last group in \cref{proposition:weight_one_cohomology_of_a_connective_homologically_even_module} yields the needed result. 
\end{proof}

\subsection{Base-change around the Milnor line} 

As we have seen in \cref{theorem:vanishing_of_even_coh_of_connective_ring_above_milnor_line}, if $M$ is a homologically even, connective module over a connective $\mathbf{E}_{1}$-ring $R$, then $\gr^{n}_{ev}(M)$ is always $n$-connective, with lowest homotopy group given by
\[
\pi_{n} \gr^{n}_{ev}(M) \simeq \evencoh^{n, n}(R, M). 
\]
Since the $n$-th homotopy group detects effective epimorphisms in the $\infty$-category of $n$-connective spectra, the Milnor line group $\evencoh^{n, n}(R, M)$ plays a large role in the study of connective modules. In this section, we prove the following result about how the even cohomology groups vary in a neighbourhood of the Milnor line:  

\begin{theorem}
\label{theorem:base_change_around_milnor_diagonal}
Let $f \colon R \rightarrow S$ be a right homologically even map of connective $\mathbf{E}_{1}$-rings and let $M$ be an even flat $R$-module. Then the base-change of the canonical comparison map 
\[
\pi_{0}S \otimes_{\pi_{0} R} \evencoh^{p, q}(R, M) \rightarrow \evencoh^{p, q}(S, S \otimes_{R} M)
\]
is a surjection for $p \geq q-1$. 
\end{theorem}

\begin{proof}
We will show this by induction on weight. The case of $q = 0$ follows from \cref{theorem:vanishing_of_even_coh_of_connective_ring_above_milnor_line}. 

Let us assume that $q > 0$ and let $M \rightarrow E$ be a $\pi_{*}$-even envelope, which by \cref{proposition:pi_star_even_envelopes_exist} we can choose so that the cofibre $C$ is $2$-connective. Since both $M$ and $C$ are even flat, so is $E$. Since $S$ is homologically even as a right $R$-module, it follows from \cref{theorem:characterization_of_homologically_even_modules} that $\pi_{*}(S \otimes_{R} E)$ is even. In particular, $S \otimes_{R} E$ is homologically even as an $S$-module by \cref{lemma:pistar_even_means_homologically_even_and_even_filtration_is_whitehead}. Consider the map of cofibre sequences 
\[
\begin{tikzcd}
	M & E & C \\
	{S \otimes_{R}M} & {S \otimes_{R} M} & {S \otimes_{R}C}
	\arrow[from=1-3, to=2-3]
	\arrow[from=1-1, to=2-1]
	\arrow[from=1-1, to=1-2]
	\arrow[from=1-2, to=1-3]
	\arrow[from=2-1, to=2-2]
	\arrow[from=2-2, to=2-3]
	\arrow[from=1-2, to=2-2]
\end{tikzcd}.
\]
Using functoriality of the even filtration, this induces a map of long exact sequences 
\[
\begin{tikzcd}
	{\evencoh^{p, q}(R, E) } & {\evencoh^{p, q}(R, C)} & {\evencoh^{p+1, q}(R, M)} & {\evencoh^{p+1, q}(R, E)} \\
	{\evencoh^{p, q}(S, S \otimes_{R} E) } & {\evencoh^{p, q}(S, S \otimes_{R} C)} & {\evencoh^{p+1, q}(S, S \otimes_{R} M)} & {\evencoh^{p+1, q}(S, S \otimes_{R} E)}
	\arrow[from=1-1, to=1-2]
	\arrow[from=1-2, to=1-3]
	\arrow[from=1-3, to=1-4]
	\arrow[from=2-1, to=2-2]
	\arrow[from=2-2, to=2-3]
	\arrow[from=2-3, to=2-4]
	\arrow[from=1-1, to=2-1]
	\arrow[from=1-2, to=2-2]
	\arrow[from=1-3, to=2-3]
	\arrow[from=1-4, to=2-4]
\end{tikzcd},
\]
As $p \geq q-1 \geq 0$, the two groups in the right-most column vanish by \cref{corollary:even_cohomology_of_even_modules}. Thus, in the square
\[
\begin{tikzcd}
	{\evencoh^{p,q}(R, C)} & {\evencoh^{p+1, q}(R, M)} \\
	{\evencoh^{p,q}(S, S \otimes_{R} C)} & {\evencoh^{p+1, q}(S, S \otimes_{R} M)}
	\arrow[two heads, from=1-1, to=1-2]
	\arrow[two heads, from=2-1, to=2-2]
	\arrow[from=1-2, to=2-2]
	\arrow[from=1-1, to=2-1]
\end{tikzcd}
\]
the two horizontal maps are surjective. As a base-change of an epimorphism is an epimorphism, the same is true in
\[
\begin{tikzcd}
	{\pi_{0} S \otimes_{\pi_{0} R} \evencoh^{p,q}(R, C)} & {\pi_{0} S \otimes_{\pi_{0}R} \evencoh^{p+1, q}(R, M)} \\
	{\evencoh^{p,q}(S, S \otimes_{R} C)} & {\evencoh^{p+1, q}(S, S \otimes_{R} M)}
	\arrow[two heads, from=1-1, to=1-2]
	\arrow[two heads, from=2-1, to=2-2]
	\arrow[from=1-2, to=2-2]
	\arrow[from=1-1, to=2-1]
\end{tikzcd}
\]
Since $\evencoh^{p, q}(R, C) \simeq \evencoh^{p, q-1}(R, \Sigma^{-2}C)$ by \cref{remark:even_cohomology_of_suspension_in_terms_of_weights} and $\Sigma^{-2} C$ is connective, the left vertical map is surjective by the inductive assumption. It follows that so is the right vertical map, ending the argument. 
\end{proof}

\bibliographystyle{amsalpha}
\bibliography{piotr_bibliography}

\end{document}